\documentclass[reqno]{amsart}

\usepackage[all]{xy}

\def\E{\ifmmode{\mathbb E}\else{$\mathbb E$}\fi} 
\def\N{\ifmmode{\mathbb N}\else{$\mathbb N$}\fi} 
\def\R{\ifmmode{\mathbb R}\else{$\mathbb R$}\fi} 
\def\Q{\ifmmode{\mathbb Q}\else{$\mathbb Q$}\fi} 
\def\C{\ifmmode{\mathbb C}\else{$\mathbb C$}\fi} 
\def\H{\ifmmode{\mathbb H}\else{$\mathbb H$}\fi} 
\def\Z{\ifmmode{\mathbb Z}\else{$\mathbb Z$}\fi} 
\def\P{\ifmmode{\mathbb P}\else{$\mathbb P$}\fi} 
\def\T{\ifmmode{\mathbb T}\else{$\mathbb T$}\fi} 
\def\SS{\ifmmode{\mathbb S}\else{$\mathbb S$}\fi} 
\def\DD{\ifmmode{\mathbb D}\else{$\mathbb D$}\fi} 

\newcommand{\e}{\varepsilon}

\newcommand{\del}{\partial}

\newcommand{\ben}{\begin{enumerate}}
\newcommand{\een}{\end{enumerate}}
\newcommand{\be}{\begin{equation}}
\newcommand{\ee}{\end{equation}}
\newcommand{\bea}{\begin{eqnarray}}
\newcommand{\eea}{\end{eqnarray}}
\newcommand{\beastar}{\begin{eqnarray*}}
\newcommand{\eeastar}{\end{eqnarray*}}
\newcommand{\bc}{\begin{center}}
\newcommand{\ec}{\end{center}}

\theoremstyle{theorem}
\newtheorem{thm}{Theorem}[section]
\newtheorem{cor}[thm]{Corollary}
\newtheorem{lem}[thm]{Lemma}
\newtheorem{prop}[thm]{Proposition}

\theoremstyle{definition}
\newtheorem{defn}[thm]{Definition}
\newtheorem{rem}[thm]{Remark}

\newtheorem{exm}[thm]{Example}

\newtheorem*{thm*}{Theorem}

\numberwithin{equation}{section}

\def\R{{\mathbb R}}
\def\osc{{\hbox{\rm osc}}}
\def\Crit{{\hbox{Crit}}}

\def\E{{\mathbb E}}
\def\Z{{\mathbb Z}}
\def\C{{\mathbb C}}
\def\R{{\mathbb R}}
\def\P{{\mathbb P}}

\def\N{{\mathbb N}}

\def\11{{\mathbb I}}

\def\delbar{{\overline \partial}}

\def\C{\mathbb{C}}
\def\Z{\mathbb{Z}}

\def\T{\mathbb{T}}

\def\Q{\mathbb{Q}}

\def\E{\ifmmode{\mathbb E}\else{$\mathbb E$}\fi} 
\def\N{\ifmmode{\mathbb N}\else{$\mathbb N$}\fi} 
\def\R{\ifmmode{\mathbb R}\else{$\mathbb R$}\fi} 
\def\Q{\ifmmode{\mathbb Q}\else{$\mathbb Q$}\fi} 
\def\C{\ifmmode{\mathbb C}\else{$\mathbb C$}\fi} 
\def\H{\ifmmode{\mathbb H}\else{$\mathbb H$}\fi} 
\def\Z{\ifmmode{\mathbb Z}\else{$\mathbb Z$}\fi} 
\def\P{\ifmmode{\mathbb P}\else{$\mathbb P$}\fi} 
\def\SS{\ifmmode{\mathbb S}\else{$\mathbb S$}\fi} 
\def\DD{\ifmmode{\mathbb D}\else{$\mathbb D$}\fi} 

\def\R{{\mathbb R}}
\def\osc{{\hbox{\rm osc}}}
\def\Crit{{\hbox{Crit}}}
\def\E{{\mathbb E}}
\def\Z{{\mathbb Z}}
\def\C{{\mathbb C}}
\def\R{{\mathbb R}}

\def\N{{\mathbb N}}

\def\delbar{{\overline \partial}}

\def\e{\varepsilon}

\def\CA{{\mathcal A}}

\def\CE{{\mathcal E}}

\def\CJ{{\mathcal J}}
\def\CK{{\mathcal K}}
\def\CL{{\mathcal L}}
\def\CM{{\mathcal M}}

\def\CQ{{\mathcal Q}}
\def\CR{{\mathcal R}}

\def\darr#1{\raise1.5ex\hbox{$\leftrightarrow$}
\mkern-16.5mu #1}

\def\roughly#1{\raise.3ex\hbox{$#1$\kern-.75em
\lower1ex\hbox{$\sim$}}}

\def\opname#1{\mathop{\kern0pt{\rm #1}}\nolimits}

\def\dim{\opname{dim}}

\def\Int{\opname{Int}}

\def\supp{\operatorname{supp}}

\def\dist{\operatorname{dist}}

\begin{document}
\quad \vskip1.375truein

\def\mq{\mathfrak{q}}
\def\mp{\mathfrak{p}}
\def\mH{\mathfrak{H}}
\def\mh{\mathfrak{h}}
\def\ma{\mathfrak{a}}
\def\ms{\mathfrak{s}}
\def\mm{\mathfrak{m}}
\def\mn{\mathfrak{n}}
\def\mz{\mathfrak{z}}
\def\mw{\mathfrak{w}}
\def\Hoch{{\tt Hoch}}
\def\mt{\mathfrak{t}}
\def\ml{\mathfrak{l}}
\def\mT{\mathfrak{T}}
\def\mL{\mathfrak{L}}
\def\mg{\mathfrak{g}}
\def\md{\mathfrak{d}}
\def\mr{\mathfrak{r}}

\title[Seidel's exact triangle]
{Seidel's long exact sequence on Calabi-Yau manifolds}
\author{Yong-Geun Oh}
\address{University of Wisconsin, Department of Mathematics,
Madison, WI 53706}
\date{February 7, 2010}

\begin{abstract} In this paper, we generalize construction of Seidel's
long exact sequence of Lagrangian Floer cohomology to that of compact Lagrangian
submanifolds with vanishing Malsov
class on general Calabi-Yau manifolds. We use the framework of anchored
Lagrangian submanifolds developed in \cite{fooo:anchor} and
some compactness theorem of \emph{smooth} $J$-holomorphic sections
of Lefschetz Hamiltonian fibration for a generic choice of $J$.
The proof of the latter compactness theorem
involves a study of proper pseudoholomorphic curves in the setting of
noncompact symplectic manifolds with cylindrical ends.
\end{abstract}


\thanks{The present work is partially supported by the NSF grant \# 0904197}

\maketitle


\tableofcontents

\section{Introduction}
\label{sec:intro}

To put the content of this paper in perspective,
we first recall a long exact sequence for symplectic
Floer cohomology of Lagrangian submanifolds, which was
constructed by Seidel \cite{seidel:triangle} originally for the category
of exact Lagrangian submanifolds on (non-compact) exact symplectic manifolds.

\subsection{Dehn twists and Seidel's long exact sequence}
\label{subsec:DehnSeidel}

Let $(M,\omega,\alpha)$ be an exact symplectic manifold
with contact type boundary: $\alpha$ is a contact one-form on $\del M$
which satisfies $d\alpha = \omega|_{\del M}$ and makes $\del M$ convex.
Assume $[\omega,\alpha] \in H^2(M,\del M;\R)$ is zero so that $\alpha$
can be extended to a one-form $\theta$ on $M$ satisfying $d\theta = \omega$.

\begin{thm}[Seidel \cite{seidel:triangle}]
Let $L$ be an exact Lagrangian sphere in $M$ together with a preferred
diffeomorphism $f: S^2 \to L$. Denote by $\tau_L =\tau_{(L,[f])}$ be the
Dehn twist associated to $(L,[f])$. For any two compact exact Lagrangian
submanifolds $L_0, \, L_1 \subset M$, there is a long
exact sequence of Floer cohomology groups
\be\label{eq:exacttriangle}
\longrightarrow HF(\tau_L(L_0),L_1) \longrightarrow HF(L_0,L_1) \longrightarrow
HF(L,L_1) \otimes HF(L_0,L) \longrightarrow.
\ee
\end{thm}

Due to the well-known difficulties
in the construction of Lagrangian intersection Floer cohomology
and a new non-trivial compactness issue arising from the singularities
of the Lefschetz fibration used for the construction,
Seidel put the exactness assumption to
avoid these difficulties and work entirely with exact Lagrangian
category and left the extension to
the more general situation like for the closed Calabi-Yau manifolds
as an open problem \cite{seidel:triangle}.
While the limitation to the exact Lagrangian category
simplifies the analysis of holomorphic disc bubbles, it also forces
him to work entirely with the language of \emph{exact Lefschetz fibrations}
and to make sure that he does not go out of this domain largely for the
consistency of his exposition, as Seidel himself indicated.
Because of this, \cite{seidel:triangle}
develops a fair amount of geometry of exact Lefschetz fibrations
some of which are not directly relevant to the construction of
the long exact sequence. Partly due to this digression,
it took some effort and time for the author of the present paper
to get to the main point of Seidel's construction in \cite{seidel:triangle}.

The cases of closed Calabi-Yau manifolds or Fano toric manifolds
are the one that is physically most relevant to the mirror symmetry:
According to Kontsevich \cite{konts:rutgers} and Seidel \cite{seidel:triangle},
the symplectic Dehn-twists correspond to a particular class of autoequivalences,
``twist functors along spherical objects'', of derived
categories of coherent sheaves, and this long exact sequence
corresponds to an exact sequence of the same form
in the mirror Calabi-Yau. Therefore it is important to establish
the long exact sequence for a class of Lagrangian branes that is
\emph{closed} under the action of symplectic Dehn twists. The class of
exact Lagrangian submanifolds in exact symplectic manifolds is
one such class, which Seidel considered in \cite{seidel:triangle}.

One of the point Seidel tried to ensure by working with the
exact Lagrangian category is to have not only single-valuedness of
the action functional on the path space but \emph{coherence} of
the definition of the action functional between different exact
Lagrangian submanifolds: this then allows one to have the energy
estimate for the Floer trajectories, and more importantly to allow one
to have \emph{thick-thin decomposition} of the Floer moduli
spaces entering in the construction. This decomposition then
enables him to apply the spectral sequence argument and derive
the desired conclusion based on the contribution
coming from the \emph{thin} part of the Floer moduli spaces which can be
explicitly analyzed.

\subsection{Calabi-Yau Lagrangian branes}
\label{subsec:CYbranes}

In regard to extending Seidel's construction to closed
Calabi-Yau manifolds, we highlight two points that we take in
this paper.

The first point is our restriction to the class of Lagrangian
submanifolds with zero Maslov class. This class is \emph{closed}
under the action by symplectic Dehn twists and enables one to
consider the involved cohomology as a \emph{$\Z$-graded group} which
is essential in the point of view of mirror symmetry. The second
point is the usage of the notion of \emph{anchors} and
\emph{anchored Lagrangian submanifolds} introduced in \cite{fooo:anchor}. This
notion has its origin in the preprint \cite{fooo00} when the authors
take the \emph{based point of view} of Lagrangian submanifolds
in relation to the coherence of the definitions of
various Maslov-type indices and of action functionals when one
considers several Lagrangian submanifolds altogether as one studies
Fukaya category. In the technical point of view, consideration of
anchored Lagrangian submanifolds enables one to keep
consistency of the definitions both of action functionals and of the
absolute gradings on the Calabi-Yau Lagrangian branes. Most
importantly this also enables us to provide a coherent filtration in
the relevant Floer complexes and to have thick-thin decomposition of
the relevant Floer moduli spaces of the kind as Seidel considered in
\cite{seidel:triangle}.

Now we introduce a class of decorated Lagrangian submanifolds
on Calabi-Yau manifold $(M,\omega)$ which we call
\emph{Calabi-Yau Lagrangian branes}. It is
expected that this class of Lagrangian submanifolds
`generates' the Fukaya category of a Calabi-Yau manifold
that is mirror to the derived category of coherent sheaves
on the mirror Calabi-Yau. We refer readers to the main part of
the paper for various undefined terms in the statement. We also
omit the important datum of flat line bundles on $L$ in this
definition because it will not play much role in our proof
but can be easily incorporated in the construction.

\begin{defn}\label{CYbrane} Let $y \in M$ be a base point and $\Lambda_y \subset
T_yM$ a fixed Lagrangian subspace. Suppose $\Theta$ is a quadratic
complex volume form on $(M,\omega,J)$ with $\langle \Theta(y), \Lambda_y \rangle = 1$.
We consider the quadruple
$((L,\gamma),s,[b])$, which we call an \emph{(anchored) Calabi-Yau Lagrangian
brane}, that satisfies the following data:
\begin{enumerate}
\item $L$ a Lagrangian submanifold of $M$
such that the Maslov index of $L$ is zero and
$[\omega] \in H^2(M,L;\Z)$. We also enhance $L$ with flat complex
line bundle on it.
\item $\gamma$ is an anchor of $L$ relative to $y$.
\item $s$ is a spin structure of $L$.\par
\item $[b] \in \mathcal M(L)$ is a bounding cochain described in subsection
\ref{subsec:objects}.
\end{enumerate}
We denote by $\CE^{CY}_{brane}$ the collection of Calabi-Yau Lagrangian
branes and define $Fuk(\CE^{CY}_{brane})$ to be the Fukaya category generated by $\CE^{CY}_{brane}$.
\end{defn}

We remark that the notion of \emph{anchor} to $L$ is introduced to
solve the problems of grading and filtration on the Floer complex
in a uniform way in \cite{fooo:anchor}. In particular it provides a canonical
filtration on the associated Floer complex of anchored Lagrangian submanifolds
which is needed to apply some spectral sequence argument in the
proof. See the end of section \ref{sec:thickthin} in particular.

\subsection{Statement of the main result and compactness issue}
\label{subsec:statement}

The main purpose of the present paper is to construct an
exact sequence for the Calabi-Yau Lagrangian branes on Calabi-Yau
manifolds, which is the analog to Seidel's \cite{seidel:triangle}.

We first note that each Dehn twist $\tau_L$ along a given Lagrangian sphere
$L \subset M$ acts on $\CE^{CY}_{brane}$. We denote this action by
$$
(\tau_L)_*: \CE^{CY}_{brane} \to \CE^{CY}_{brane}
$$
and the image of $\CL$ under this action  by $\tau_L(\CL)=
(\tau_L)_*\CL$. This action defines an auto-equivalence on
$Fuk(\CE^{CY}_{brane})$, whose non-anchored versions should
correspond to twist functors along spherical objects of derived
categories of coherent sheaves alluded in the beginning of this
introduction.

\begin{thm}
Let $(M,\omega)$ be a compact (symplectic) Calabi-Yau and $y \in M$ be
a base point $y$. Let $L \subset M$ be a Lagrangian sphere with $y \not \in L$ together with a preferred
diffeomorphism $f: S^2 \to L$ and $\CL= ((L,\gamma),s_{st},0)$
be the associated Calabi-Yau Lagrangian brane. Denote by $\tau_L =\tau_{(L,[f])}$ the
Dehn twist associated to $(L,[f])$.

Consider any Calabi-Yau Lagrangian branes $\CL_0, \, \CL_1$.
Then there is a long
exact sequence of $\Z$-graded Floer cohomologies
\be\label{eq:exacttriangle}
\longrightarrow HF((\tau_L)_*\CL_0,\CL_1)
\longrightarrow HF(\CL_0,\CL_1)
\longrightarrow  HF(\CL,\CL_1) \otimes
HF(\CL_0,\CL_1) \longrightarrow
\ee
as a $\Lambda_{nov}$-module where the Floer cohomologies involved are the
deformed Floer cohomology constructed in \cite{fooo:book}.
\end{thm}
We also have the long exact sequence for the non-anchored version of the Floer
cohomology. See section \ref{subsec:wrapping}.
\par

Once we are to use these frameworks, construction of the long
exact sequence largely follows Seidel's strategy: We use the framework
of Lefschetz fibration with Lagrangian boundary conditions for the construction
of various operators appearing in the Floer theory, and
use the spectral sequence for the $\R$-filtered groups based on
the thick-thin decomposition of the Floer moduli spaces.
However unlike the exact Lagrangian case, the definition of
Lagrangian Floer cohomology for Calabi-Yau Lagrangian branes meet
obstruction as described in \cite{fooo:book}. Because of this
we have to use the Maurer-Cartan elements $b_i$ and use
the associated \emph{deformed Floer cohomology} appearing in
the statement of the main theorem above. (Since Lagrangian submanifolds
with zero Maslov class in Calabi-Yau manifolds are \emph{semi-positive},
the related transversality issue is relatively standard which is
one of the advantages considering this class of Lagrangian submanifolds.)
For the readers' convenience and the readability of the paper, we
borrow a fair amount of materials from \cite{fooo:book} in our exposition.
For the same reason, we also borrow much exposition from \cite{seidel:triangle}
and refer to the two for further details.
In a way, most of the materials used in this paper are not new but
has already been present in the literature one way or the other. We organize them
in a coherent way to be able to construct the required long exact sequence.
Familiarity of the scheme in the paper \cite{seidel:triangle} would be useful
for the readers to follow the stream of the arguments used in this paper,
especially those presented in sections \ref{sec:seidel's}-\ref{sec:sequence}.

However there is one nontrivial analytical issue that needs to be overcome.
This concerns the issue of compactification of \emph{smooth} pseudo-holomorphic
section of \emph{Lefschetz (Hamiltonian) fibration} when the fibration
has non-empty critical fibers. By the definition of Lefschetz Hamiltonian fibration
given in Definition \ref{defn:Lefschetz},
any smooth section will avoid critical points of the fibration. However a priori a sequence of
smooth sections may approach critical points\emph{if the derivatives of
the sections in the sequence blow up}. When applied to a sequence of
pseudo-holomorphic sections, the bubble could touch the critical points.
Therefore to define the relative Gromov-Witten type invariants in the
Lefschetz fibration setting, one
should study the behavior of pseudo-holomorphic sections approaching
the critical points. This compactness result may be mathematically the most novel part of the
present paper which is carried out in section \ref{sec:pseudo}.

In this regard, we prove the following

\begin{thm}[Theorem \ref{thm:maximum}]
Let $\pi: E \to \Sigma$ be a Lefschetz Hamiltonian fibration with
Lagrangian boundary $Q \subset E_{\del \Sigma}$ such that $E$ is
fiberwise Calabi-Yau and $Q$ has vanishing fiberwise Maslov class.
Then there exists a dense subset of $j$-compatible $J$'s for which
we have a constant $C > 0$ depending only on $(E,J,j)$ and the
section class $B \in \pi_2^{sec}(E,Q)$, but independent of $s$, such
that we have \be\label{eq:dist}
\operatorname{dist}(\operatorname{Im}s,E^{crit}) \geq C \ee for any
\emph{smooth} section $s: \Sigma \to E$ with $[s,\del s] \in B \in
\pi_2^{sec}(E,Q)$.
\end{thm}

The proof of this theorem turns out to involve the compactification and
the Fredholm theory
in the setting of symplectic field theory \cite{EGH} in which we
regard a bubble touching a critical point $x_0$ as a proper
pseudo-holomorphic curve on $\C$ in a punctured fiber $E_{z_0} \setminus
\{x_0\}$. See \cite{fooo:chap10} and \cite{oh-zhu} for relevant
studies of such compactification and Fredholm theory.

Once this theorem is established, study of compactification of smooth
pseudo-holomorphic sections in the current case
is essentially the same as the case of
smooth Hamiltonian fibrations as studied in \cite{entov}, \cite{mcduff-sal04}.

The result in the present paper was first announced in Eliashberg's 60-th
Birthday Conference: ``New Challenges and Perspectives in
Symplectic Field Theory'' held at Stanford University, June 25 - 29, 2007
and then presented in various seminars and in conferences afterwards.
We apologize readers of a long delay of appearance of the present paper.

As always, we would like to thank Fukaya, Ohta and Ono for our years-long
collaboration on Lagrangian Floer theory, especially writing together the book
\cite{fooo:book} and the paper \cite{fooo:anchor} whose frameworks we adopt throughout
the present paper. We also thank Seidel for some useful e-mail communications
in the early stage of the present work, concerning compactness of
the moduli space of holomorphic sections of exact Lefschetz fibrations in
his paper \cite{seidel:triangle}.
Our thank also goes to KIAS and NIMS in Korea for their financial supports and hospitality
during our stay when a large chunk of writing of this paper was carried out.

\section{Basic facts on symplectic Dehn twists}
\label{sec:dehn}

In this section, we summarize basic facts on the Dehn twists in
the symplectic point of view which Seidel extensively studied in a
series of papers \cite{seidel:jdg,seidel:triangle}. We borrow the basic
facts on the symplectic Dehn twists from them with a slight
variation of the exposition that will be necessary for the purpose
of the present paper.

Assume that $L \subset (M,\omega)$ be a embedded Lagrangian sphere
together with an equivalence class $[f]$ of diffeomorphisms $f:S^2
\to L$: two $f_1, \, f_2$ are equivalent if and only if $f_2^{-1}f_1$ can be
deformed inside $Diff(S^n)$ to an element of $O(n+1)$. To any such
$(L,[f])$ Seidel associates a Dehn twist $\tau_L = \tau_{(L,[f])}
\in Symp(M)$ using a \emph{model Dehn twist} on the cotangent
bundle $T: = T^*S^n$.
Let $f: S^n \to L \subset M$ be a representative of the
equivalence class $[f]$. Denote by $T(r) \subset
T^*S^n$  the disc bundle of radius $r$ in terms of the
standard metric on the unit sphere $S^n = S^n(1) \subset
\R^{n+1}$. Identifying $T=TS^n$ with respect to the standard metric, one
consider the map
\be\label{eq:sigmat}
\sigma_t(u,v) = \left( \cos(t) u - \sin(t) \|u\|v, \cos(t) v +\sin(t)
\frac{u}{\|u\|} \right)
\ee
for $0 < t < \pi$. $\sigma_\pi$ is the antipodal involution $A(u,v) = (-u,-v)$.
Next we fix a function $R \in C^\infty(\R,\R)$ such that
\bea\label{eq:R}
\supp R & \subset & T(1)\nonumber \\
R(-t) & = & R(t) - t \quad \mbox{for }\, |t| \leq \frac{1}{2}.
\eea
Then we consider the re-scaled function
\be\label{eq:Rlambda}
R_\lambda(t) = \lambda R\left(\frac{t}{\lambda}\right)
\ee
for all $0 < \lambda \leq 1$. Then $R_\lambda$ is supported in $T(\lambda)$ and
satisfies
\be
R_\lambda(-t) = R_\lambda(t) - t \quad \mbox{for }\, |t| \leq \frac{\lambda}{2}.
\ee
Insertion of one-parameter $\lambda$ in our choice of $R$ is deliberate
which will be later explicitly related to the parameter that enters in the Lagrangian
surgery.

The following lemma is a slight variation of
Lemma 1.8 \cite{seidel:triangle} whose proof is referred thereto.

\begin{lem} Let $\mu:T \setminus T(0) \to \R$ be the length function
$\mu(u,v) = \|v\|$ and $H_\lambda = R_\lambda\circ \mu$ on $T \setminus T(0)$.
Then $\phi_{H_\lambda}^{2\pi}$ extends smoothly over $T(0)$ to a
symplectic diffeomorphisms $\phi_\lambda$ of $T$. The function
$K_\lambda = 2\pi(R'_\lambda\circ \mu - R\circ \mu)$ also extends smoothly
over $T(0)$, and satisfies
$$
\phi_\lambda^*\theta_T - \theta_T = dK_\lambda.
$$
These $\phi_\lambda$ are called \emph{model Dehn twists}.
\end{lem}

The model Dehn twists, denoted by $\tau_\lambda$, have the explicit formula
\be\label{eq:taulamda}
\tau_\lambda(y) = \begin{cases} \sigma_{2\pi R_\lambda'(\mu(y))}(y) \quad &
y \in T(\lambda) \setminus T(0) \\
A(y) \quad & y \in T(0)
\end{cases}
\ee
where the angle of rotation goes from $2\pi R_\lambda'(0) = \pi$
to $2\pi R_\lambda'(\lambda) = 0$:
Note that as $\lambda \to 0$, we have
$$
2\pi R_\lambda'(\mu(y)) = 2\pi R'\left(\frac{\mu(y)}{\lambda}\right)
$$
and it changes from $2\pi R'(0) = \pi$ to $2\pi R'(1) = 0$.

Now we take a Darboux-Weinstein chart, or a symplectic embedding $\iota:
T(1) \to M$ such that
$$
\iota|_{o_{T^*S^n}} = f, \quad \iota^*\omega = \omega_T (= -d\theta_T)
$$
for a representative of the framed Lagrangian sphere $(L,[f])$.
Take a model Dehn twist $\tau$ supported in the interior of $T(1)$.

We denote $U = \operatorname{im} \iota$ and fix the Darboux neighborhood
once and for all, and consider the one-parameter family of Dehn
twists $\tau_r$ any of which we denote by $\tau_L$
\be\label{eq:tauLf}
\tau_L = \tau_{(L,[f];r)} = \begin{cases} \iota \circ \tau_r \circ \iota^{-1}
\quad & \mbox{on }\, \operatorname{im}(\iota) = U \\
id \quad & \mbox{elsewhere}.
\end{cases}
\ee
We quote the following basic fact on the Dehn twist $\tau_{(L,[f])}$
from \cite{seidel:triangle} with a slight variation of the statements.

\begin{prop}[Proposition 1.11 \cite{seidel:triangle}]
Let $(L,[f])$ be a framed Lagrangian sphere in $M$.
There is a one-parameter family of
Lefschetz fibrations $(E^L_\lambda,\pi^L_\lambda) \to D(\lambda)$
together with an isomorphism $\phi^L_\lambda: E_\lambda^L \to M$
of symplectic manifolds, such that
\begin{enumerate}
\item Consider the re-scaling map $R_\lambda: D(\lambda) \to D(1)$ defined by
$z \mapsto \frac{z}{\lambda}$. Then
$$
(R_\lambda)^*E^L_1 = E^L_\lambda.
$$
\item If $\rho^L_\lambda$ is the symplectic
monodromy around $\del \overline D(\lambda)$, then $\phi^L_\lambda \circ
\rho^L \circ (\phi^L_\lambda)^{-1}$ is a Dehn twist along $(L,[f])$.
\item There exists a decomposition
$$
E^L = E \cup \overline D(\lambda) \times (M \setminus \iota(T(\lambda)) \setminus
V))
$$
such that $E$ is the standard fibration $q: \C^{n+1} \to \C$
defined by
$$
q(z_1, \cdots, z_{n+1}) = z_1^2 + \cdots + z_{n+1}^2.
$$
\end{enumerate}
We denote any of these maps by $\tau_L$.
\end{prop}

An important point on this Dehn twist $\tau_{(L,[f];\lambda)}$ is that
its support can be put into a Darboux neighborhood of the given
Lagrangian sphere $L$ which can be made as close as to $L$ by choosing
$\lambda>0$ small, whose derivative can be controlled: One can choose $R$ so that
for some $\delta > 0$ we have
\bea\label{eq:Rdelta}
R'(t) \geq 0 \quad &{}& \mbox{for all $t \geq 0$}\\
R''(t) < 0 \quad &{}& \mbox{for all $t \geq 0$ such that $R'(t) \geq \delta$}
\eea
and then consider $R_\lambda$ for any sufficiently small $\lambda > 0$.
According to Seidel's terminology, the corresponding Dehn twist is
\emph{$\delta$-wobbly}.

\section{Action, grading and anchored Lagrangian submanifolds}
\label{sec:anchored}

In this section, we consider the general Lagrangian submanifolds treated as
in \cite{fooo:book}. For the fine chain level analysis of the Floer
complex, it is essential to analyze the \emph{$\R$-filtration}
on $CF(L_0,L_1)$ that is provided by
the action functional $\CA$ on $\Omega(L_0,L_1)$. This action functional
is not single valued on $\Omega(L_0,L_1)$ itself even for the pair $(L_0,L_1)$
of Calabi-Yau Lagrangian branes, but single-valued only on some covering space.
For the purpose of studying the Fukaya category and
carrying out various constructions in the Floer homology in
a coherent manner, we need to consider
a whole collection of Lagrangian submanifolds and assign these
auxiliary data to each pair of the given collection in a consistent way.
For this purpose, \cite{fooo:book,fooo:anchor} uses the notion of
\emph{anchored} Lagrangian submanifolds. This auxiliary data
is important later for consistency of definition of
action functionals and in turn for the analysis of \emph{thick-thin} decomposition of the
various Floer moduli spaces entering in the construction of
boundary map, chain map, chain-homotopy and pants products.
The above mentioned covering space is defined in terms of a reference path $\ell_0$
chosen in $\Omega(L_0,L_1)$, which we denote by $\widetilde\Omega(L_0,L_1;\ell_0)$.

Furthermore to provide an absolute
grading to each critical point $[p,w]$ of the action functional
$\CA: \widetilde\Omega(L_0,L_1) \to \R$, we also need to fix a
section $\lambda$ of $\ell_0^*\Lambda(M,\omega)$ where $\Lambda(M,\omega)$ is the
bundle of Lagrangian Grassmanians on $(M,\omega)$.

From now on, we denote either $(L,\gamma)$ or $(L,\gamma,\lambda)$ just by
$\CL$ depending on the circumstances.

\subsection{The $\Gamma$-equivalence}
\label{subsec:gamma-equiv}

We consider a covering space of
$\Omega(L_0, L_1;\ell_0)$ by modding out the space of
paths in $\Omega(L_0,L_1;\ell_0)$ by an equivalence relation,
which is weaker than the homotopy.
The deck transformation group of this covering space wis {\it abelian} by
construction.
\par
Note that when we are given two pairs $(\ell,w)$ and $(\ell, w')$
from $\Omega(L_0, L_1;\ell_0)$, the concatenation
$$
\overline w \# w': [0,1] \times [0,1] \to M
$$
defines a loop $c: S^1 \to \Omega(L_0,L_1;\ell_0)$.
One may regard this loop as a map
$
C: S^1 \times [0,1] \to M
$
satisfying the boundary condition
$
C(s,0) \in L_0$, $C(s,1) \in L_1.
$
Obviously the symplectic area of $C$, denoted by
$$
I_\omega(c) = \int_C\omega
$$
depends only on the homotopy class of $C$
and so defines a homomorphism on $\pi_1(\Omega(L_0,L_1;\ell_0))$,
which we also denote by
$$
I_\omega: \pi_1(\Omega(L_0,L_1;\ell_0)) \to \R.
$$
Next we note that for the map $C: S^1 \times [0,1] \to M$
satisfying the given Lagrangian boundary condition, it associates a symplectic
bundle pair $(\mathcal V,\lambda)$ defined by
$$
\mathcal V_C = C^*TM, \, \lambda_C = c_0^*TL_0 \sqcup c_1^*TL_1
$$
where $c_i: S^1 \to L_i$ is the map given by
$c_i(s) = C(s,i)$ for $i = 0, \, 1$. This allows us to define another
homomorphism
$$
I_\mu: \pi_1 (\Omega(L_0, L_1), \ell_0) \rightarrow \Z;
\quad I_\mu(c) = \mu(\mathcal V_C,\lambda_C)
$$
where $\mu(\mathcal V_C,\lambda_C)$ is the Maslov index of the bundle pair
$(\mathcal V_C,\lambda_C)$. See section 2.2.1 \cite{fooo:book} for details.
\par
Using the homomorphisms $I_\mu$ and $I_\omega$, we define an
equivalence relation $\sim$ on the set of all pairs $(\ell, w)$. For given such pair $w,
\, w'$, we denote by $\overline w \#w'$ the concatenation of
$\overline w$ and $w'$
along $\ell$, which defines a loop in $\Omega(L_0,L_1;\ell_0)$
based at $\ell_0$.

\begin{defn} We say that $(\ell,w)$ is
$\Gamma$-equivalent to $(\ell, w')$ and write $(\ell, w) \sim (\ell, w')$
if the following conditions are satisfied
$
I_\omega(\overline w \# w') = 0 = I_\mu(\overline w \# w')
$.
We denote the set of equivalence classes $[\ell, w]$ by
$
\widetilde \Omega (L_0, L_1;\ell_0)
$
and call the {\it Novikov covering space}.
\end{defn}

There is a canonical lifting of $\ell_0
\in \Omega(L_0,L_1;\ell_0)$ to
$\widetilde\Omega(L_0,L_1;\ell_0)$: this is just
$
[\ell_0,\widetilde\ell_0] \in \widetilde\Omega(L_0,L_1;\ell_0)
$
where $\widetilde\ell_0$ is the map $\widetilde \ell_0: [0,1]^2
\to M$ with $\widetilde\ell_0(s,t)= \ell_0(t)$. In this way,
$\widetilde\Omega(L_0,L_1;\ell_0)$ also has a natural base point
which we suppress from the notation.
\par
We denote by $\Pi(L_0,L_1;\ell_0)$ the group of deck transformations
of the covering space $\widetilde\Omega(L_0,L_1;\ell_0)
\to \Omega(L_0,L_1;\ell_0)$. It is easy to see that
the isomorphism class of $\Pi(L_0,L_1;\ell_0)$ depends only on the
connected component containing $\ell_0$.
\par
The two homomorphisms
$I_\omega$ and $I_\mu$ push down to homomorphisms
$$
E: \Pi(L_0,L_1;\ell_0) \to \R, \qquad \mu: \Pi(L_0,L_1;\ell_0) \to \Z
$$
defined by
$$
E(g) = I_\omega[C], \quad \mu(g) = I_\mu[C]
$$
for any map $C: S^1 \times [0,1] \to M$ representing
the class $g \in \Pi(L_0,L_1;\ell_0)$. The group $\Pi(L_0,L_1;\ell_0)$ is an abelian group.

We now define the Novikov ring $\Lambda (L_0,L_1;\ell_0)$ associated the
abelian covering
$
\widetilde \Omega(L_0,L_1;\ell_0) \to \Omega(L_0,L_1;\ell_0)
$
as a completion of the group ring $R[\Pi(L_0,L_1;\ell_0)]$.
Here $R$ is a commutative ring with unit.
\begin{defn} $\Lambda_k^R (L_0,L_1;\ell_0)$ denotes the set of
all (infinite) sums $$\sum_{g\in \Pi(L_0,L_1;\ell_0) \atop \mu (g) = k} a_g
[g]$$ such that $a_g \in R$ and for each $C$, the set
$
\{ g \in \Pi(L_0,L_1;\ell_0) \mid E(g) \leq C, \,\,a_g \not = 0\}
$
is of finite order. We put
$$
\Lambda^R (L_0, L_1;\ell_0) = \bigoplus_k \Lambda^R_k (L_0, L_1;\ell_0) .
$$
\end{defn}

The ring structure on $\Lambda^R (L_0, L_1;\ell_0)$ is defined by
the convolution product
$$
\left (
\sum_{g \in \Pi(L_0,L_1;\ell_0)} a_g [g] \right ) \cdot
\left (
\sum_{g \in \Pi(L_0,L_1;\ell_0)} b_g [g] \right )
= \sum_{g_1, g_2 \in \Pi(L_0,L_1;\ell_0)} a_{g_1} b_{g_2} [g_1g_2].
$$
It is easy to see that the term in the right hand side is indeed
an element in $\Lambda^R (L_0, L_1;\ell_0)$. Thus $\Lambda^R (L_0, L_1;\ell_0) = \oplus_k
\Lambda_k^R (L_0, L_1;\ell_0)$ becomes a graded ring under this
multiplication.  We call this graded ring the {\it Novikov ring} associated
to the pair $(L_0,L_1)$ and the connected component containing $\ell_0$.
\par
We also use the universal Novikov ring $\Lambda_{\text{\rm nov}}$ in this paper.
We recall its definition here.
An element of $\Lambda_{\text{\rm nov}}$ is a formal sum
$\sum a_iT^{\lambda_i}e^{\mu_i}$ with $a_i \in \C$, $\lambda_i \in \R$,
$\mu_i \in \Z$ such that $\lambda_i \leq \lambda_{i+1}$ and $\lim_{i\to \infty} \lambda_i = \infty$.
$T$ and $e$ are formal parameters. We define
a valuation $\frak v: \Lambda_{\text{\rm nov}} \to \R_{\ge 0}$ defined by
$$
\frak v\left(\sum_{i=1}^\infty a_iT^{\lambda_i}e^{\mu_i}\right) = \lambda_1.
$$
We denote the corresponding `valuation ring' by
$$
\Lambda_{0,\text{\rm nov}} = \left\{ \sum_{i=1}^\infty a_iT^{\lambda_i}e^{\mu_i} \in \Lambda_{\text{\rm nov}}
\mid \lambda_i \geq 0 \right\}.
$$
It carries a unique maximal ideal consisting of
$\sum a_iT^{\lambda_i}e^{\mu_i}$ with $\lambda_i > 0$ for all $i$ which we denote
by $\Lambda_{0,\text{\rm nov}}^+$. We have a natural embedding
$$
\Lambda^R (L_0, L_1;\ell_0) \to \Lambda_{\text{\rm nov}}
$$
given by
\be\label{eq:nov-embedding}
\sum_{g \in \Pi(L_0,L_1;\ell_0)} b_g [g] \mapsto
\sum_{g \in \Pi(L_0,L_1;\ell_0)} b_g \, T^{\omega(g)} e^{\mu(g)/2}.
\ee
\par
Now for a given pair $(\ell,w)$, we define the {\it action functional}
$$
\mathcal A: \widetilde \Omega(L_0,L_1;\ell_0) \to \R
$$
by the formula
$$
\mathcal A(\ell,w) = \int w^*\omega.
$$
It follows from the definition of $\Pi(L_0,L_1;\ell_0)$ that
the integral depends only on the $\Gamma$-equivalence class $[\ell,w]$
and so pushes down to a well-defined functional on the
covering space $\widetilde \Omega(L_0,L_1;\ell_0)$.

\begin{lem} The set
$
Cr(L_0,L_1;\ell_0)
$
of critical points of $\mathcal A$ consists of the pairs of the type
$
[\ell_p,w]
$
where $\ell_p$ is the constant path with $p \in L_0\cap L_1$
and $w$ is as in $(2.2.2)$.
$Cr(L_0,L_1;\ell_0)$ is invariant under the action of
$\Pi(L_0,L_1;\ell_0)$ and so forms a principal bundle over a subset of $L_0\cap L_1$
with its fiber isomorphic to $\Pi(L_0,L_1;\ell_0)$.
\end{lem}
We put
$$
Cr(L_0,L_1) = \bigcup_{\ell_{0,i}} Cr(L_0,L_1;\ell_{0,i})
$$
where $\ell_{0,i}$ runs over the set of
base points of connected components of $\Omega(L_0.L_1)$.
\par
Next, we assign an {\it  absolute} Morse index to each critical
point of $\mathcal A$. In general, assigning such an absolute
index is not a trivial matter because the obvious Morse index of
$\mathcal A$ at any critical point is infinite.
For this purpose, we will use the
Maslov index of certain bundle pair naturally associated to
the critical point $[\ell_p,w] \in Cr(L_0,L_1;\ell_0)$.
\par
We call this Morse index of $[\ell_p,w]$ the {\it  Maslov-Morse index}
(relative to the base path $\ell_0$) of the
critical point.
The definition of the index will somewhat resemble that of
$\mathcal A$. However to define this, we also need to fix a section
$\lambda^0$ of $\ell_0^*\Lambda (M)$ such that
$$
\lambda^0(0) = T_{\ell_0(0)}L_0, \quad \lambda^0(1) =
T_{\ell_0(1)}L_1.
$$
Here $\Lambda(M)$ is the bundle of Lagrangian Grassmanians of $TM$
$$
\Lambda(M) = \bigcup_{p\in M}\Lambda(T_pM)
$$
where $\Lambda(T_pM)$ is the set of Lagrangian subspaces of the
symplectic vector space $(T_pM, \omega_p)$.
\par
Let $[\ell_p,w] \in Cr(L_0,L_1;\ell_0) \subset
\widetilde\Omega(L_0,L_1;\ell_0)$ be an element whose projection
corresponds to the intersection point $p \in L_0\cap L_1$.
\par
As before we associate a symplectic bundle pair $(\mathcal V_w,\lambda_w)$
over the square $[0,1]^2$, which will be defined uniquely up to
the homotopy. We first choose $\mathcal V_w = w^*TM$. To define $\lambda_w$,
let us choose
a path $\alpha^p: [0,1] \to \Lambda(T_pM,\omega_p)$ satisfying
$$
\left\{\aligned
\alpha^p(0) & = T_pL_0, \, \alpha^p(1) = T_pL_1 \subset T_pM, \\
(\alpha^p)(t) & \oplus T_pL_0 = T_pM ,\\
\alpha^p(t) &\in U_0(T_pL_0) \quad \text{for small $t$},
\endaligned\right.
$$
where $U_0(T_pL_0)$ is as above.

Then we consider a continuous Lagrangian subbundle
$\lambda_w \to \del [0,1]^2$ of $\mathcal V|_{\del [0,1]^2}$ by the following
formula: the fiber at each point of $\del [0,1]^2$ is given as
$$\left\{
\aligned
\lambda_w(s,0) & =  T_{w(s,0)}L_0, \quad
\lambda_w(1,t) = \alpha^p(t), \\
\lambda_w(s,1) & =  T_{w(s,1)}L_1, \quad
\lambda_w(0,t) = \lambda^0(0,t) .
\endaligned
\right.
$$
It follows  that the homotopy type of
the bundle pair constructed as above does not depend on the choice of
$\alpha^p$ either.

\begin{defn}
We define the {\it Maslov-Morse index} of $[\ell_p,w]$
(relative to $\lambda^0$) by
$$
\mu([\ell_p,w];\lambda^0) = \mu(\mathcal V_w,\lambda_w) .
$$
\end{defn}

\subsection{Anchored Lagrangian submanifolds}
\label{subsec:anchored}

Now we would like to generalize this construction for a chain
$$
\CL= (L_0,\cdots, L_k)
$$
of more than two Lagrangian submanifolds, i.e., with $k \geq 2$.
\par
To realize this purpose, we need the notion of \emph{anchors}
of Lagrangian submanifolds. In this subsection, we briefly recall the definition of anchored Lagrangian
submanifolds introduced in \cite{fooo:anchor}.

\begin{defn}\label{defn:anchored}
Fix a base point $y$ of ambient symplectic manifold $(M,\omega)$.
Let $L$ be a Lagrangian submanifold of $(M,\omega)$. We define an
\emph{anchor} of $L$ to $y$ is a path $\gamma:[0,1] \to M$ such that
$$
\gamma(0) = y, \, \gamma(1) \in L.
$$
We call a pair $(L,\gamma)$ an \emph{anchored} Lagrangian submanifold.
\end{defn}
It is easy to see that any homotopy class of path in
$\Omega(L,L')$ can be realized by a path that passes through the
given point $y$. We denote the set of homotopy classes of the anchors
$\gamma$ to $y \in M$ by $\pi_1(y,L)$.

The following lemma is easy to check

\begin{lem}\label{lem:torsor} Suppose that $L$ is connected.
Then set $\pi_1(y,L)$ of homotopy classes relative to the ends is a
principal homogeneous space of $\pi_1(M,L)$, i.e., it is a
$\pi_1(M,L)$-torsor. We  call an element of $\pi_1(y,L)$ an \emph{anchor
class} of $L$ relative to $y$.
\end{lem}
\begin{proof} Since $L$ is connected, the natural map
$\pi_1(M) \to \pi_1(M,L)$ is  surjective and so $\pi_1(M,L) \cong
\pi_1(M)/\operatorname{im}(\pi_1(L) \to \pi_1(M))$ forms a group. It
is obvious to see that $\pi_1(M,L)$ acts on $\pi_1(y,L)$ by
concatenation of paths on the right. By definition, this action is
free. Transitivity is obvious by definition.
\end{proof}

For a given pair $(\CL, \CL')$ of anchored Lagrangians $\CL =
(L,\gamma), \, \CL' = (L', \gamma')$, we denote
$$
\Omega(\CL,\CL'): = \Omega(L,L'; \overline \gamma \# \gamma')
$$
where the latter is the path component of $\Omega(L,L')$ containing
$\overline \gamma \# \gamma'$. We also denote \be\label{eq:CLcapCL'}
\CL \cap \CL' =\{p \in L \cap L' \mid \widehat p \in
\Omega(\CL,\CL')\}.
\ee
Here $\widehat p$ is the constant path $\widehat p(t) \equiv p$.

When we are given a Lagrangian chain
$$
\frak L = (L_0, L_1, \cdots, L_k)
$$
we also consider a chain of anchors $\gamma_i: [0,1] \to M$ of $L_i$ to $y$
for $i = 0, \cdots, k$. These anchors give a systematic choice of a base path
$\ell_{ij} \in \Omega(L_i,L_j)$ by concatenating $\gamma_i$ and
$\gamma_j$ as
$$
\ell_{ij} = \overline \gamma_i * \gamma_j
$$
where $\overline \gamma$ is the time-reversal of $\gamma$ given by
$\overline \gamma(t) = \gamma(1-t)$.
The upshot of this construction is the following
overlapping property
\bea\label{eq:ellij}
\ell_{ij}(t) & = & \ell_{i\ell}(t) \quad \text{for } \, 0 \leq t
\leq \frac{1}{2} \nonumber\\
\ell_{ij}(t) & = & \ell_{\ell j}(t) \quad \text{for }\, \frac{1}{2}
\leq t \leq 1
\eea
for all $j, \, \ell$.
\begin{defn}\label{classB} Let $\CE = \{(L_i,\gamma_i)\}_{0 \leq i \leq k}$ be
a chain of anchored Lagrangian submanifolds.
A homotopy class $B \in \pi_2(\CL;\vec p)$
is called \emph{admissible} to $\CE$ if it
can be obtained by a polygon that is a gluing of $k$ bounding
strips $w_{i(i+1)}: [0,1] \times [0,1] \to M$ satisfying
\bea
w_{i(i+1)}(0,t) & = & \begin{cases}
\gamma_i(2t-1) \quad & 0 \leq t \leq \frac{1}{2} \\
\gamma_{i+1}(2t-1) \quad & \frac{1}{2} \leq t \leq 1
\end{cases} \\
w_{i(i+1)}(s,0) & \in & L_i, \quad w_{i(i+1)}(s,1) \in L_{i+1}\\
w_{i(i+1)}(1,t) & = & p_{i(i+1)}.
\eea
When this is the case, we denote the homotopy class $B$ as
$$
B = [w_{01}]\#[w_{12}] \# \cdots \# [w_{k0}]
$$
and the set of admissible homotopy classes by $\pi_2^{ad}(\CE;\vec p)$.
\end{defn}
We call such a tuple $\CE$ an \emph{anchored Lagrangian chain}.
\begin{rem} We remark that we denote by $\frak L$ a chain
$(L_0,\cdots,L_k)$ of Lagrangian submanifolds and by $\CE$ that
of anchored Lagrangian submanifolds.
\end{rem}

\emph{When} the collection
$\CE = \{(L_i,\gamma_i)\}_{0 \leq i \leq k}$ is given, we note that
not all homotopy classes in $\pi_2(\CL;\vec p)$ is admissible.
But we have the following basic lemma which will be enough for
the construction of Fukaya category, whose proof is easy and so
omitted.
\begin{lem} Let $w_{i(i+1)}$ be given for $i =1, \cdots, k$ and
$B \in \pi_2(\CL;\vec p)$. Then there exists a unique
$[w_{k0}]$ such that
$$
B = [w_{01}]\# \cdots \# [w_{(k-1)k}] \# [w_{k0}]
$$
\end{lem}

The following basic identity immediately follows from definitions

\begin{prop}\label{prop:coherence} Suppose $B \in \pi_2^{ad}(\CE,\vec p)$
given as Lemma \ref{classB} and provide the analytic coordinates
at the marked points $z_j$ so that all $z_j$ are outgoing.
Then we have $\omega(B) = \sum_{i=0}^{k}\CA([p_i,w_i])$
and
\bea
\omega(B) & = & \sum_{i=0}^k \CA_{\ell_{i(i+1)}}([p_i,w_i])\\
\mu(\CE,\vec v;B) & = &\sum_{i=0}^{k} \mu([p_i,w_i];\lambda_{i(i+1)}).
\eea
In particular the sum in the right hand sides do
not depend on the choice of $\lambda_i \subset \gamma_i^*TM$.
\end{prop}

Here the index $\mu([p_i,w_i];\lambda_{i(i+1)})$ is the
Maslov-Morse index relative to the path $\lambda_{i(i+1)}$
of Lagrangian planes as defined in \cite{fooo:book}, which provides
a coherent grading $\mu:\operatorname{Crit} \CA \to \Z$.

The action functional provides a canonical $\R$-filtration on the set
$$
\CA:\operatorname{Crit} \CA \to \R.
$$
In addition, inside the collection of anchored Lagrangian submanifolds
$(L,\gamma)$ we are given a coherent system of single
valued actions functionals
$$
\CA: \widetilde{\Omega_0}(L_i,L_j;\overline \gamma_i \#\gamma_j) \to \R
$$
at one stroke for any pair from the given collection
$\CE$ of anchored Lagrangian submanifolds.

\subsection{Relation to the graded Lagrangian submanifolds}
\label{subsec:relation}

Now we go back to the collection of Lagrangian submanifolds with
vanishing Maslov class on a Calabi-Yau manifolds. In this case, we
will be able to obtain a canonical Lagrangian path $\lambda$ along a
given anchor $\gamma$ of $L$ to $y$.

Let $J$ be a compatible almost complex structure. The assumption
$2c_1(M) = 0$ implies that the bundle
$\Delta = \Lambda^n(TM,J)^{\otimes 2}$ is trivial. Choose a section $\Theta$
of $\Delta^*$ that has length one everywhere in terms of the metric
$g = \omega(\cdot, J\cdot)$. This determines a map
$\det_\Theta^2: \Lambda(M,\omega) \to S^1$, and then an $\infty$-fold
Maslov covering by
\be\label{eq:LLinfty}
\Lambda_\infty = \{(\Lambda,t) \in \Lambda \times \R \mid \operatorname{det}^2_{\Theta}(\Lambda) =
e^{2\pi i t} \}
\ee
where $s_L: L \to \Lambda(M,\omega)|_L$ is the natural section defined by the
Gauss map $s_L(x) = T_xL$.
An $\Lambda_\infty$-grading of a Lagrangian submanifold $L \subset (M,\omega)$
according to Seidel \cite{seidel:grading} is just a lift to
$\R$ of the map
$$
\operatorname{det}^2_{\Theta}\circ s_L: L \to S^1.
$$
First of all the condition on $\mu_L = 0$ implies that there is such a
lifting to $\Lambda_\infty$.

We now explain how we give a coherent grading to Lagrangian
submanifolds with vanishing Maslov class. We choose a Lagrangian
path $\lambda$ over $\ell$ so that \be\label{eq:det2=1}
\operatorname{det}^2_{\Theta}(\gamma(t))(\lambda(t)) \equiv 1. \ee
Then we choose a lifting
$\widetilde{\operatorname{det}^2_{\Theta}\circ \gamma}$ of
$\operatorname{det}^2_{\Theta}\circ \gamma: [0,1] \to S^1$ so that
\be\label{eq:tildedet2=0}
\widetilde{\operatorname{det}^2_{\Theta}\circ\gamma}(1) = 0. \ee The
hypothesis $2c_1(M) = 0, \, \mu_L(0) = 0$ immediately implies that
there exists a unique lifting
$\widetilde{\operatorname{det}^2_{\Theta}\circ s_L}:L \to \R$ of
$\operatorname{det}^2_{\Theta}\circ s_L: L \to S^1$ satisfying $
\widetilde{\operatorname{det}^2_{\Theta}\circ s_L}(\gamma(1)) = 0. $
We denote this lifting by $\alpha_{(L,\gamma,\lambda)}: L \to \R$.
Obviously this does not depend on the choice of $\lambda$ as long as
$\lambda$ satisfies (\ref{eq:det2=1}).

\begin{defn} Let $(M,\omega)$ be such that $2c_1(M) = 0$ and
fix a base point $(y,\Lambda_y)$. Let $(L,\gamma)$ be a Calabi-Yau
anchored Lagrangian. We denote the above common lifting by $\alpha_{(L,(x_L,\gamma_L))}: L \to \R$
and call the \emph{canonical grading} of $(L,\gamma)$ relative to
$(y,\Lambda_y)$.
\end{defn}
We refer to section 9.2 \cite{fooo:anchor} for more detailed
explanation on the above discussion. Because of this presence of canonical
grading associated to $(L,\gamma)$, we will drop $\lambda$ from our
notation $\CL = (L,\gamma,\lambda)$ when we consider Calabi-Yau Lagrangian
branes later in this paper.

Adapting to the convention from \cite{konts:hms} and \cite{seidel:grading},
we denote by $\widetilde L[0] = (L,\alpha_{(L,\gamma)})$ and
$$
\widetilde L[k] = (L, \alpha_{(L,\gamma)} - k).
$$

\section{Calabi-Yau Lagrangian branes and Dehn twists}
\label{sec:branes}

We restrict ourselves to the case of
$(M,\omega)$ with $2c_1(M,\omega) = 0$ and $L \subset M$ whose
Maslov class vanishes from now on. We will give a precise definition of Calabi-Yau
Lagrangian branes in this section. This is the
case that is most relevant to mirror symmetry and to
extension of Seidel's long exact sequence of $\Z$-graded symplectic
Floer cohomology.

Now we introduce a class of decorated Lagrangian submanifolds
on Calabi-Yau manifold $(M,\omega)$ which we call
\emph{Calabi-Yau Lagrangian branes}.

\begin{defn}\label{CYbrane} Let $y \in M$ be a base point and $\Lambda_y \subset
T_yM$ a fixed Lagrangian subspace. Suppose $\Theta$ be a quadratic
complex volume form on $(M,\omega,J)$.
Let $\CE^{CY}$ be the Calabi-Yau Lagrangian collection of $(M,\omega)$.
We consider the triple $(\CL,s,[b])$, $\CL = (L,\gamma)$ which we call
a \emph{anchored Calabi-Yau Lagrangian brane}, that satisfies the following data:
\begin{enumerate}
\item $L$ a Lagrangian submanifold of $M$
such that the Maslov index of $L$ is zero and
$[\omega] \in H^2(M,L;\Z)$. We also enhance $L$ with flat complex
line bundle on it.
\item $\gamma$ is an anchor of $L$ to $y$.
\item $s$ is a spin structure of $L$.\par
\par
\item $[b] \in \mathcal M(L)$ is a bounding cochain described in subsection
\ref{subsec:objects}.
\end{enumerate}
We denote the collection of Calabi-Yau Lagrangian
branes by $\CE^{CY}_{brane}$, and the Fukaya category generated
by them by $Fuk\left(\CE^{CY}_{brane}\right)$.
\end{defn}

The first simplification arising from considering the Calabi-Yau
Lagrangian collection is that we have only to use the Novikov ring
of the form \be \Lambda_{nov}^{(0)} = \left\{\sum_{i=1}^\infty a_i
T^{\lambda_i} ~\Big|~ a_i \in \Q, \, \lambda_i \in \R, \, \lambda_i
\leq \lambda_{i+1}, \, \lim_{i\to \infty}\lambda_i = \infty \right\}
\ee which becomes a field. We also consider the subring \be
\Lambda_{0,nov}^{(0)} = \left\{\sum_{i=1}^\infty a_i T^{\lambda_i}
\in \Lambda_{nov} ~\Big|~ \lambda_i \geq 0 \right\}. \ee This is
because the Maslov index satisfies $\mu(w) = \mu_L(\del w) = 0$ for
any disc map $w:(D^2,\del D^2) \to (M,L)$ where $\mu_L \in
H^1(L;\Z)$ is the Maslov class of $L$.

\begin{rem}
Furthermore, as we mentioned in subsection \ref{subsec:relation}, the anchor
provides a canonical graded structure on a Calabi-Yau Lagrangian
brane. Therefore it provides a canonical $\R$-filtration and a $\Z$-grading on
$$
CF(\CL_0,\CL_1)= CF((L_0,\gamma_0),(L_1,\gamma_1)):=CF(L_0,L_1;\overline\gamma_0\# \gamma_1)
$$
for any pair $(\CL_0,\CL_1)$ of CY Lagrangian branes,
and hence on its cohomology $HF(\CL_0,\CL_1)$. See section \ref{sec:fooo} for
related discussion.
\end{rem}

Now we examine the effect of Dehn twists on the CY Lagrangian collection.

\begin{prop} Let $\CE^{CY}_{brane}$ be the
associated collection of anchored CY Lagrangian branes.
Then $\CE^{CY}_{brane}$ is closed under
the action of $\tau_{L}$'s for all framed Lagrangian sphere $(L,[f])$,
and so induces an auto-equivalence of $Fuk\left(\CE^{CY}_{brane}\right)$.
\end{prop}
\begin{proof} We note that the Dehn twist $\tau_L$ is a symplectic
automorphism. Therefore it pushes the spin structure of $L_0$ to the
image $\tau_L(L_0)$, and pull-backs the Maslov class. Therefore
the Maslov class of $\tau_L(L_0)$ is also zero.
Similarly we can push forward the anchor of $L$ to $\tau_L(L_0)$.
Finally Theorem B (B.2) \cite{fooo:book} states that the bounding cochain
can also be canonically pushes under the symplectic automorphism and hence
under the action of $\tau_L$. This finishes the proof.
\end{proof}

Therefore we can ask the question on how the Floer cohomology
changes under the Dehn twist along a Lagrangian sphere.
The answer is supposed to come from a long exact sequence
that Seidel introduced in \cite{seidel:triangle} for the context of
exact Lagrangian submanifolds. The rest of the paper will be occupied with
the construction of this long exact sequence for the Calabi-Yau
Lagrangian branes on Calabi-Yau manifolds.

\section{Lefschetz Hamiltonian fibration and coupling form}
\label{sec:fibration}

In this section, we first recall the basics on \emph{smooth} Hamiltonian fibrations
presented in \cite{GLS}, \cite{entov} and extend our discussion to fibrations with
Lefschetz-type singular fibers. Especially we generalize the notion of
\emph{coupling form} to the current singular fibration and prove the
uniqueness of the coupling form on a given Lefschetz Hamiltonian fibration,
when the fibration $\pi: E \to \Sigma$ is proper, e.g., when the fiber of $E$
is compact.

The notion of Hamiltonian fibrations introduced by Guillemin-Lerman-Sternberg
\cite{GLS} is the family of symplectic manifolds of a fixed isomorphism type,
which could be twisted on the parameter space $\Sigma$. On the other hand,
Seidel introduced the notion of exact Lefschetz fibrations which could have
a finite number of singular fibers of type $A_1$-singularity.

Combining \cite{GLS} and \cite{seidel:triangle}, we give the following definition

\begin{defn}[Lefschetz Hamiltonian fibration]\label{defn:Lefschetz}
A Lefschetz  Hamiltonian fibration over a compact surface
$\Sigma$ with boundary $\del\Sigma$ consists of the data
$(E,\pi,\Omega,J_0,j_0)$ as follows:
\begin{enumerate}
\item $\del E = \pi^{-1}(\del \Sigma)$ and $\pi|_{\del E} \to \del \Sigma$
forms a smooth fiber bundle.
\item $\pi: E \to \Sigma$ can have at most a finitely many critical points,
and no two may lie on the same fiber. Denote $E^{crit} \subset E$ and
$\Sigma^{crit} \subset \Sigma$.
\item $J_0$ is a complex structure on a neighborhood of $E^{crit}$,
$j_0$ is a positively oriented complex structure on a neighborhood of
$\Sigma^{crit}$, and $\pi$ is $(J_0,j_0)$-holomorphic near $E^{crit}$.
And the Hessian $D^2\pi$ at any critical point is nondegenerate as
a complex quadratic form.
\item $\Omega$ is a closed two form on $E$ which must be nondegenerate
on $\ker D\pi_x$ for each $x \in E$, and a K\"ahler from for
$J_0$ in some neighborhood of $E^{crit}$.
\end{enumerate}
We say that the fibration $\pi: E \to \Sigma$ is (symplectically)
\emph{Calabi-Yau} if $c_1(E_z^v) = 0$ at all $z \not\in \Sigma^{crit}$
\end{defn}
We would like to remark that one may allow more than one critical points
in the same fiber, which could be useful for the study of a family of
Lefschetz Hamiltonian fibrations.

\begin{rem} We would like to highlight that at a critical point $x \in E$
of $\pi$ Condition (4) implies that the form $\Omega_x$ is required to be
nondegenerate on the whole tangent space $T_eE$ since $\ker D \pi_x = T_xE$, while at a
regular point $\Omega_x$ it is required so only at on the vertical tangent
space $T_e^v E$ as $\ker D\pi_x = T_x^vE$.
\end{rem}

When a generic fiber $E_z$ with $z \in \Sigma \setminus \Sigma^{crit}$ is
compact, it is proved in \cite{GLS} for a smooth Hamiltonian fibration
that one can choose $\Omega$ is uniquely determined by the following
additional requirement
\be\label{eq:coupling}
\pi_*\Omega^{n+1} = 0
\ee
where $\pi_*$ is the integration along the fiber. Now we prove the following
analog to this result for the case of Lefschetz Hamiltonian fibrations.

\begin{thm}\label{thm:coupling} Let $(E,\pi,\Omega,J_0,j_0)$ be a
Lefschetz Hamiltonian fibration as in Definition \ref{defn:Lefschetz}. Then there exists a
closed 2-form $\Omega'$ smooth on $E \setminus E^{crit}$ that satisfies the following:
\begin{enumerate}
\item $\Omega'|_{T^vE} = \omega_e$ at all $e \in E \setminus E^{crit}$,
\item it satisfies \eqref{eq:coupling} on $E \setminus \pi^{-1}(\Sigma^{crit})$.
\end{enumerate}
Furthermore such a form $\Omega'$ is unique. We call such a form the \emph{coupling
form} of $(E,\pi,\Omega,J_0,j_0)$.
\end{thm}
\begin{proof} We first consider the subset $E\setminus E^{crit}$.
Then the form $\Omega$ induces a splitting
$$
\Gamma: \quad T_xE = T_xE^v \oplus T_x^hE
$$
at any regular point $x \in E \setminus E^{crit}$ where the horizontal space is given by
$$
T_x^h E = \{\eta \in T_xE \mid \Omega(\eta, \xi) = 0 \, \forall \xi \in T_x^vE\}
$$
and hence induces a natural (Ehresman) connection on $E \setminus E^{crit}$
whose monodromy is symplectic.

Because of the closedness of $\Omega$, the connection is Hamiltonian \cite{GLS}
in that its curvature $\operatorname{curv}(\Gamma)$ has its values
contained in $ham(E_{\pi(e)})$, the set of Hamiltonian vector fields of the
fiber $E_{\pi(e)}$, and hence  the restriction $\pi: E \setminus E^{crit} \to
\Sigma \setminus \Sigma^{crit}$ is a smooth Hamiltonian fibration
in the sense of \cite{GLS}. In particular, if we restrict to
$E \setminus \pi^{-1}(\Sigma^{crit}) \to \Sigma \setminus \Sigma^{crit}$
its fibers are all compact and so we can construct a closed 2-form
$\Omega'= \Omega + \pi^*\alpha$ on $E \setminus \pi^{-1}(\Sigma^{crit})$
for some closed two form $\alpha$ on $\Sigma \setminus \Sigma^{crit}$ that
satisfies \eqref{eq:coupling} thereon. In fact $\Omega'$ (and so $d\beta$) can be explicitly
constructed by requiring
\be\label{eq:curvature}
\Omega'(\eta_1^\sharp,\eta_2^\sharp) = H_{\eta_1,\eta_2}
\ee
where $H_{\eta_1,\eta_2}$ is the smooth function whose restriction to each
fiber over a point in $\Sigma \setminus \Sigma^{crit}$ is uniquely determined by
the two requirements
\begin{enumerate}
\item $H_{\eta_1,\eta_2}$ generates the Lie algebra element
$\operatorname{curv}_\Gamma(\eta_1, \eta_2)$ of $Ham(E_z,\omega_z)$
which is a Hamiltonian vector field
\item it satisfies the normalization condition
$$
\int_{E_z} H_{\eta_1,\eta_2} \, \omega_z^n = 0, \quad \omega_z = \Omega|_{E_z}
$$
for all $z \in \Sigma \setminus \Sigma^{crit}$.
\end{enumerate}
This finishes the proof.
\end{proof}

\begin{defn}[Coupling form] \label{defn:coupling} We call the above unique closed 2-form constructed
in Theorem \ref{thm:coupling} the \emph{coupling form} of the Lefschetz Hamiltonian
fibration $E \to \Sigma$.
\end{defn}

The following result was essentially proved by Seidel Lemma 1.6
\cite{seidel:triangle}: Seidel proved this for the context of exact
Lefschetz fibrations but the same proof applies if one ignores
his consideration of generating functions of $Q$ therein.

\begin{lem}[Compare with Lemma 1.6 \cite{seidel:triangle}]
Let $(E,\pi,\Omega,J_0,j_0)$ be a  Lefschetz Hamiltonian fibration, and $x_0$ be
a critical point of $\pi$. Then there are smooth families $\Omega^\mu \in \Omega^2(E)$
$0 \leq \mu \leq 1$, such that
\begin{enumerate}
\item $\Omega^0 = \Omega$
\item for all $\mu$, $\Omega^\mu = \Omega^0$ outside a small neighborhood of
$x_0$
\item each $(E,\pi,\Omega^\mu,J_0,j_0)$ is a Lefschetz  Hamiltonian fibration
\item there is a holomorphic Morse chart $(\xi,\Xi)$ around $x_0$
with $\Xi: V \subset \C^{n+1}\to E$ such that
$\Xi^*\Omega^1$ agree near the origin with the standard forms
$\omega_{\C^{n+1}} = \frac{i}{2}\sum dx_k \wedge d\overline x_k$.
\end{enumerate}
\end{lem}

In fact, if  near
$E^{crit}$ we are given a one form $\Theta$ with $\Omega = d\Theta$
as in the exact cases, we can also deform the one-form to $\Theta^\mu$ so that
$\Xi^*\Theta^1$ becomes the standard one-form
$$
\theta_{\C^{n+1}} = \frac{i}{4}\sum x_k d\overline x_k - \overline x_k dx_k.
$$
(See Lemma 1.6 \cite{seidel:triangle}.)

\section{Exact Lagrangian boundary condition and action estimates}
\label{sec:exact}

Now we consider a subbundle $i_Q: Q \to \del \Sigma$ of the symplectic
vector bundle $(E|_{\del \Sigma},\Omega|_{\del \Sigma})$
whose fiber $Q_z$ is a Lagrangian submanifold of $\Omega_z$
for each $z \in \del \Sigma$. We call such $Q$ a \emph{fiberwise-Lagrangian
submanifold} of $(E|_{\del \Sigma},\Omega|_{\del \Sigma})$.

\begin{defn}\label{exactfamily} We call a fiberwise-Lagrangian submanifold
$Q \subset \del E$ an \emph{exact Lagrangian boundary over $\del \Sigma$},
if there exists a one-form $\kappa_Q$ on $Q$ such that
$$
\kappa_Q|_{T(\del E)^v} \equiv 0, \quad \mbox{and } i_Q^*\Omega = d \kappa_Q
$$
where $i_Q: Q \to \del E$ is the inclusion map and  $T^v(\del E)$ is the vertical
tangent space of $\del E$.
\end{defn}

We note that when $\Sigma$ is oriented and the boundary orientation
on $\del \Sigma$ is provided by an orientation one-form,
denoted by $d\theta$ with $\theta \in \del \Sigma$, the connection induced by
the form $\Omega|_{\del E}$ enables us to express any such one-form $\kappa_Q$ as
$$
\kappa_Q(\theta,x) = h_i(\theta,x) \, d\theta
$$
for $(\theta,x) \in \del E$ with $\theta \in \del_i \Sigma$ and $h_i:\del E \to \R$
where $\del \Sigma = \coprod \del_i \Sigma$.
The function $h_i$ is unique up to the addition of the function
$c_i: \del \Sigma_i \to \R$.

\begin{defn} We define $\|\kappa_Q\|_{(1,\infty)}$ by
$$
\|\kappa_Q\|_{(1,\infty)} = \int_{\del\Sigma} \osc(h_\theta)\, d\theta
$$
with $\osc(h_\theta): = \max_{x \in E_x} h(\theta,x) -
\min_{x \in E_x} h(\theta,x)$, and call it the $L^{(1,\infty)}$-norm of $\kappa_Q$.
\end{defn}
We remark that $\|\kappa_Q\|_{(1,\infty)}$ does not depend on
the choice of the function $h$.

To give readers some insight on these definitions, we compare this with the
classical notion of \emph{exact Lagrangian isotopy} \cite{gromov}.

\begin{exm} Let $E = (\R \times [0,1]) \times (M,\omega)$ with
$\pi: (\R \times [0,1]) \times (M,\omega) \to \R \times [0,1]$ the
projection. Consider the two form $\Omega = \pi^*\omega$.
Let $L_i \subset (M,\omega)$ be a Lagrangian submanifold
and let $\psi_i: [0,1] \times L \to M$ be a Lagrangian isotopy
for $i = 1,\, 2$.
The isotopy $\psi_i$ is an exact Lagrangian isotopy if there is a
smooth function $h_i:[0,1] \times L \to \R$ such that
$\psi_i^*\omega = dh_i \wedge dt = d(h_i\,dt)$ \cite{gromov}. This definition
is a special case of Definition \ref{exactfamily}: Just consider
the embedding
$$
(h_i,\psi_i): \R \times L_i \to \R \times [0,1] \times M
$$
and set $Q_i = \operatorname{im} (h_i,\psi_i)$ the Lagrangian suspension,
and $\kappa_{Q_i}(t,x) = h_i(t,x) \, dt$.
\end{exm}

We now study the structure of $\pi_2(E,Q)$.
We start with the following relative version of \emph{section class}.
A class $B \in \pi_2(E,Q)$ is a section class if $\pi_*B \in \pi_2(\Sigma,\del\Sigma)$
is the positive generator with respect to the given orientation of $\Sigma$.
We say that $B$ is a \emph{fiber class} if it is in the image of $\pi_2(M,L) \to \pi_2(E,Q)$
induced from the inclusion of the fiber. The following is proved in Lemma 2.2 \cite{hu-lal}
for the smooth Hamiltonian fibration but the same proof applies to the current
case with singular fibers. We refer readers to \cite{hu-lal} for its proof.

\begin{lem} The following sequence of homotopy groups is exact at the
middle term:
$$
\pi_2(M,L) \to \pi_2(E,Q) \to \pi_2(\Sigma,\del \Sigma).
$$
\end{lem}

The following proposition is the reason why the notion of exact
Lagrangian boundary is relevant to the study of pseudo-holomorphic
curves with boundary later. Similar estimates were previously
obtained in \cite{entov}, \cite{oh:dmj} and \cite{seidel:triangle}
in somewhat different contexts.

\begin{prop} Suppose $\Sigma$ is oriented and denote by
$d\theta$ a given orientation one-form on $\del \Sigma$.
Let $Q \subset \del E$ over $\del \Sigma$ be an exact Lagrangian boundary of
$E$ and $\kappa_Q$ be a corresponding Hamiltonian one-form.
Consider a section $s: \Sigma \to E \setminus E^{crit}$ with $s(\del \Sigma) \subset
Q \subset \del E$. Then for each given section class  $[s,\del s]
\in \pi_2(E,Q;\Z)$, there exists a constant
$C = C(\kappa_Q,[s,\del s]) > 0$ such that the integral bound
$$
\left|\int_\Sigma s^*\Omega\right| \leq C = C(B)
$$
holds for any section $s$ in a fixed class $B = [s, \del s] \in H_2(E,Q;\Z)$.
In fact, we have
\be\label{eq:upperbound}
\left|\int_\Sigma s_2^*\Omega - \int_\Sigma s_1^*\Omega\right|
\leq \|\kappa_Q\|_{(1,\infty)}
\ee
for any two such sections with $[s_1,\del s_1] = [s_2,\del s_2]$.
\end{prop}
\begin{proof} Recall $i_Q^*\Omega = d \kappa_Q$ for a one-form $\kappa_Q$
which exists by definition of exact Lagrangian boundary $Q$.
Let $s_i$, $i = 1,\, 2$ be two sections of $E$ with
$s_i(\del \Sigma) \subset Q$ with $[s_1,\del s_1] = [s_2,\del s_2]$.
Then we have a geometric chain $(S,C)$ with
$$
\del S = s_1 \coprod s_2 \coprod C.
$$
and $\del C = \del s_2 - \del s_1$ as a chain in $Q$.

By Stokes' and closedness of $\Omega$, we have
$$
0 = \int S^*(d\Omega) = \int_{\del S}\Omega =
\int_\Sigma s_2^*\Omega - \int_\Sigma s_1^*\Omega-
\int C^*\Omega
$$
and hence
$$
\int_\Sigma s_2^*\Omega - \int_\Sigma s_1^*\Omega
=\int_C \Omega.
$$
But by the exactness of the fiberwise-Lagrangian subbundle $Q$ and since
$C$ has its image in  $Q$, we obtain
$$
\int_C \Omega = \int_C d\kappa_Q = \int_{C} \kappa_Q
= \int_{\del s_2} \kappa_Q  - \int_{\del s_1} \kappa_Q.
$$
Therefore we have obtained
$$
\int_\Sigma s_2^*\Omega - \int_\Sigma s_1^*\Omega
= \int_{\del s_2} \kappa_Q  - \int_{\del s_1} \kappa_Q =
\int_{\del \Sigma} (h \circ s_2 -  h\circ s_1)\, d\theta
$$
and so
$$
\left|\int_\Sigma s_2^*\Omega - \int_\Sigma s_1^*\Omega\right|
\leq \int_{\del \Sigma} (\max_{x\in E_x}h_\theta(x) -  \min_{x\in E_x}h_\theta(x))\, d\theta
\leq \|\kappa_Q\|_{(1,\infty)}.
$$
Since $\|\kappa_Q\|_{(1,\infty)}$ does not depend on $s$, this finishes the proof.
\end{proof}


Next we consider the topological index associated to the section
$(s,\del s)$ for $s$ which does not pass through critical points
$E^{crit}$. By considering the pull-back $s^*(TE^v)$, it defines a
symplectic bundle pair $(s^*TE^v, (\del s)^*TQ^v)$ where $TQ^v =
TQ^v = TQ \cap TE^v|_{\del \Sigma}$. Therefore we can associate the
Maslov index, which we denote by $\mu([s,\del s])$. (See
\cite{katz-liu}, \cite{fooo:book}.)

Now we examine topological dependence of $\mu([s,\del s])$.
Note that each section $(s,\del s)$ defines an element in $\pi_2(E,Q)$.
We denote the corresponding class by $s_*([\Sigma,\del \Sigma])$
where $[\Sigma,\del \Sigma]$ is the fundamental class which is a generator
of $H_2(\Sigma,\del\Sigma;\Z) \cong \Z$.
The following lemma immediately follows from the definition of the Maslov
index for the bundle pair.

\begin{lem}\label{lem:maslovindex} Suppose
$(s_1)_*([\Sigma,\del \Sigma]) = (s_2)_*([\Sigma,\del\Sigma])$. Then we have
$$
\mu([s_1,\del s_1]) = \mu([s_1,\del s_2]).
$$
\end{lem}
\begin{defn} We denote by $\pi_2^{sec}(E,Q) \subset \pi_2(E,Q)$ the subset of
section classes $[s,\del s]$ in $\pi_2(E,Q)$.
We say two section classes $B_1, \, B_2$ are $\Gamma$-equivalent if they satisfy
$$
\Omega(B_1) = \Omega(B_2), \quad \mu(B_1) = \mu(B_2)
$$
and denote by $\Pi(E,Q)$ the quotient group
$$
\Pi(E,Q) = \pi_2^{sec}(E,Q)/\sim.
$$
\end{defn}

For the Calabi-Yau Lefschetz fibrations, one can proceed the study of
Maslov indices following the exposition given in \cite{seidel:book}.
Consider the bundle of relative quadratic volume forms
$$
\CK^2_{E/\Sigma} = \pi^* {\det}_\C^2 (T\Sigma) \otimes
{\det}_\C^{\otimes -2}.
$$
By definition, if $E \to \Sigma$ is Calabi-Yau, we have nowhere zero
section $\eta^2_{E/B}$ of this on $E \setminus E^{crit}$. Furthermore, we can require
$\eta^2_{E/B}$ to satisfy
$$
\eta^2_{E/B} = \frac{(dz_1 \wedge \cdots \wedge dz_{n+1})^2} {(2z_1
dz_1 + \cdots + 2z_{n+1}dz_{n+1})^{\wedge 2}}
$$
in a neighborhood $U \subset E$ of each critical point under the given identification
$\pi: U \setminus \{x\} \to \pi(U \setminus \{x\}) $
with $q: \C^{n+1} \setminus \{0\} \to \C\setminus 0$.

\begin{defn} We say that a fiberwise Lagrangian submanifold $Q
\subset E$ is (relatively) \emph{graded} if there exists a function
$$
\alpha: Q \to \R
$$
such that
$$
\exp(2\pi \sqrt{-1} \alpha(x)) = \eta^2_{E/B}(T_x^vQ), \quad x \in Q.
$$
We call $\alpha$ an $\CL^\infty$-grading of $Q$.
\end{defn}

For a given pair of a fiberwise Lagrangian submanifold $Q_1, \, Q_2 \subset E$
intersecting transversely, we can associate a natural $\Z$-grading on the
intersection $Q_1 \cap Q_2$ in the following way.

We consider the two form
$$
\omega_{E,\lambda} = \Omega + \lambda \pi^*\omega_\Sigma
$$
which is nondegenerate on $E \setminus E^{crit}$. Existence of such $\lambda > 0$
is easy to check. The following lemma also immediately follows whose proof we
leave to the readers.

\begin{lem} Any fiberwise Lagrangian submanifold $Q \subset E$ is Lagrangian
for $\omega_{E,\lambda}$.
\end{lem}

\section{Pseudo-holomorphic sections}
\label{sec:pseudo}

In this section, we carry out various studies of geometry and analysis
of pseudo-holomorphic sections with Lagrangian boundary condition
in the setting of Lefschetz Hamiltonian fibrations.

Let $(E,\pi,\Omega,J_0,j_0)$ be a Lefschetz Hamiltonian fibration
and let $x_0 \in E^{crit}$. Denote by $(\xi, \Xi)$ be a \emph{holomorphic
Morse chart} at $x_0$ i.e., a $j_0$-holomorphic coordinates $\xi: U \to S$ with
$\xi(0) = z_0 = \pi(x_0)$ where
$U \subset \C$ is a neighborhood of the origin and $\Xi: W \to E$ be
a $J_0$-holomorphic chart with $W \subset \C^{n+1}$ a neighborhood of
the origin in $\C^{n+1}$ with $\Xi(0) = x_0$ such that
\be\label{eq:Morse}
(\xi^{-1}\circ\pi\circ \Xi)(x) = x_1^2 + \cdots + x_{n+1}^2.
\ee
With $(\xi,\Xi)$ fixed, we denote the model Lefschetz
fibration by $q: \C^{n+1} \to \C$ defined by
$$
q(z_1,\cdots, z_n) = z_1^2 + \cdots + z_{n+1}^2.
$$
Now for the given $(E,\pi,\Omega,J_0,j_0)$ and the holomorphic
Morse charts $(\xi,\Xi)$ at each critical points of $E$, we consider an
almost complex structure $J$ on $E$ that satisfies
\begin{enumerate}
\item $J = J_0$ in a neighborhood of $E^{crit}$,
\item $d\pi \circ J = j \circ d\pi$,
\item $\Omega(\cdot,J\cdot)|_{TE_x^v}$ is symmetric
and positive definite for any $x \in E$,
\end{enumerate}
Following Seidel \cite{seidel:triangle}, we call such $J$
\emph{compatible relative to $j$}. An immediate consequence of the definition
is the following lemma

\begin{lem} Let $J$ be compatible relative to $j$. Then for any
given area form $\omega_\Sigma$ on $\Sigma$ with $\int_\Sigma \omega_\Sigma= 1$,
the two form $\Omega + \lambda \pi^*\omega_\Sigma$ tames $J$ for all sufficiently
large $\lambda$.
\end{lem}

We refer to \cite{seidel:triangle} for a more complete explanation of the structure
of $J$'s compatible to $j$.

\subsection{Energy estimates and Hamiltonian curvature}
\label{subsec:energy}

Next we study the energy estimates of pseudo-holomorphic sections in
terms of the topological action $\int u^*\Omega$ and the contribution
coming from the curvature integral of $\int_\Sigma u^*\Omega$ of a
canonical symplectic connection of the Hamiltonian fibration $E \to
\Sigma$ associated to the coupling form $\Omega$ defined in Definition
\ref{defn:coupling}. This kind of
estimates has been studied in \cite{seidel:triangle}, \cite{oh:dmj},
\cite{mcduff-sal04}.

Using the connection associated to $\Omega$, we decompose $Du = (Du)^v +
(Du)^h$ into the vertical and horizontal components. Now we
consider the symplectic form
$$
\omega_E = \Omega + \lambda \pi^*\omega_\Sigma
$$
with $\omega_\Sigma$ an area form on $\Sigma$ with $\int_\Sigma
\omega_\Sigma = 1$. We like to remark that an almost complex
structure $J$ compatible to $j$ it is {\it not compatible} in the
usual sense in that the bilinear form
$$
\Omega(\cdot, J \cdot)
$$
may not be symmetric. However if $\lambda$ is sufficiently large,
it is tame to $\omega_E$ (see Lemma 2.1 \cite{seidel:triangle}).
Therefore we can symmetrize this bilinear form and define the
associate metric $g_{J}$ by \be\label{eq:gJ} \langle V,W \rangle =
g_{J}(V,W): = \frac{1}{2}(\Omega(V, J W) + \Omega(W, J V)). \ee We
call (\ref{eq:gJ}) the metric associated to $J$ and denote
$$
|V|^2 = |V|^2_{J} = g_{J}(V,V).
$$
With respect to this metric, we still have the following basic
identity whose proof we omit.

\begin{lem}\label{energy} Let $s: \Sigma \to E$ be any
$J$-holomorphic map $v$. Then we have
$$
{1 \over 2} \int |Ds|^2 = \int s^*\omega_E
$$
and \be\label{eq:Es} \frac{1}{2}\int |Ds|^2 = \int s^*\omega_E +
\int |\overline\del_J s|^2. \ee
\end{lem}

We decompose $Ds = (Ds)^v + (Ds)^h$ into vertical and horizontal
parts and write
$$
|Ds|^2 = |(Ds)^v|^2 + |(Ds)^h|^2 + 2 \langle (Ds)^v, (Ds)^h
\rangle.
$$
Then it is straightforward to prove \be\label{eq:Esh} |(Ds)^h|^2
\omega_\Sigma = 2(s^*\Omega + \lambda\omega_\Sigma) \ee by
the identity
\beastar \sum_{i=1}^2|(Ds)^h(e_i)|^2 & = &
\sum_{i=1}^2\omega_{E,\lambda}((Ds)^h(e_i),
J (Ds)^h(e_i)) \\
& = & \sum_{i=1}^2 (\Omega + \lambda
\pi^*\omega_{\Sigma})((Ds)^h(e_i),
J (Ds)^h(e_i)) \\
& = & 2 (\omega_E((Ds)^h(e_1), (Ds)^h(e_2)) + \lambda
\omega_\Sigma(e_1,e_2)) \eeastar for an orthonormal frame $\{e_1,
e_2\}$ and then applying the curvature identity
$$
d(\Omega(e_1^\#,e_2^\#)) = -[e_1^\#,e_2^\#] \rfloor \Omega:
\quad \mbox{fiberwise }
$$
Here $e_i^\#$ is the horizontal lift of $e_i$ and we have
$$
\operatorname{curv}_\Gamma(e_1, e_2) = \Omega(e_1^\#,e_2^\#)
$$
by the definition of coupling form (see (1.12)
\cite{GLS} but with caution on the sign convention). The following
is an immediate corollary of (\ref{eq:Esh}).

\begin{prop} Let $\Omega$ be the coupling form of $E$. Suppose that
$\Omega + \lambda \pi^*\omega_\Sigma$ for $\lambda > 0$ is
positive, i.e., symplectic for an area form $\omega_\Sigma$ on
$\Sigma$. Then we have the inequality as a two-form
\be\label{eq:s*K} s^*\Omega + \lambda\omega_\Sigma \geq 0
\ee
for any $J$-holomorphic section $s$.
In other
words, if we write $s^*\Omega = f(s) \omega_\Sigma$ for a function $f:
\Sigma \to \R$, then we have $f(s) + \lambda \geq 0$.
\end{prop}
\begin{proof} Just choose a complex structure $j$ on $\Sigma$ and
a $j$-compatible $J$ on $E$. Positivity (\ref{eq:s*K}) immediately
follows from (\ref{eq:Esh}).
\end{proof}

\subsection{Gromov-Floer moduli space of $J$-holomorphic sections}
\label{subsec:GF-moduli}

We first translate the anchor introduced in section
\ref{sec:anchored} in the setting of pointed Lagrangian boundary conditions in
Hamiltonian fibrations over the surface $\Sigma$ with boundary
$\del \Sigma \neq \emptyset$. This will be needed to study morphisms
between two Floer chain modules constructed via the moduli space of
pseudoholomorphic sections of Lefschetz fibrations over $\Sigma$.

Let $y \in M$ be a base point and let $(L,\gamma), \, (L',\gamma')$ be
two anchored Lagrangian submanifolds of $(M,\omega)$,
which intersect transversely. Let $\gamma$ and $\gamma'$ be
the paths $\gamma(0) = \gamma'(0) = y$ and $\gamma(1) \in L,\,
\gamma'(1) = L'$ given as in the anchor data.

Now to each intersection $p \in L \cap L'$,
we associate a Hamiltonian fibration over $[0,1]^2$.
The paths $\gamma$ and $\gamma'$ provide a path in $M$ along $\{0\}
\times [0,1]$
via the obvious concatenation of $\widetilde \gamma$ and $\gamma'$
with the midpoint given by $y$.

We take the trivial fibrations $E = [0,1]^2 \times (M,\omega)
\to [0,1]^2$. Then for each given pair $[p,w]$ with $p \in L \cap L'$ and
with a bounding strip $w:[0,1]^2 \to M$ such that
\bea\label{eq:bdyw}
w(0,t) & = & \widetilde\gamma \# \gamma'(t), \quad w(\zeta) \equiv p
\nonumber \\
w(s,0) & \in & L_0, \quad w(s,1) \in L_1,
\eea
we can associate a section of $s_{[p,w]}: [0,1]^2 \to E$ by
$$
(s,t) \mapsto ((s,t), w(s,t)) \; \,[0,1]^2 \to [0,1]^2 \times (M,\omega).
$$
We call this fibration over $[0,1]^2$ with a section $s_{[p,w]}$
an \emph{anchor cap} associated to $[p,w]$ relative to the given anchor.
For the notational convenience, we denote the corresponding fibration
with the fiberwise Lagrangian submanifolds by
$$
(E_{[p,w]}; [0,1] \times \{0\} \times L_0, [0,1] \times \{1\} \times L_1;
s_{[p,w]})
$$
or simply as $(E_{[p,w]};s_{[p,w]})$.

Note that the set of homotopy classes $[p,w]$ of
$w$ relative to the boundary condition (\ref{eq:bdyw}) has one-one correspondence
with the homotopy class of sections with the obvious corresponding
boundary condition and asymptotic boundary condition at
$\pm \infty$ respectively.

For a given compact surface $\Sigma$ with marked points $\vec\zeta$,
we consider the corresponding surface $\dot \Sigma =
\Sigma \setminus \vec \zeta$ with punctures. We denote the given preferred
holomorphic chart $\varphi_\zeta: D_\zeta \subset \Sigma \to D^+$ of the half disc
$D^+ = D \cap \{ \operatorname{im}(z) \geq 0\}$ with $\varphi_\zeta(\zeta) = 0$.
We also have a local trivialization
$$
\Phi_\zeta: E|_{D_\zeta\setminus \{\zeta\}} \to  D^+ \setminus \{0\}
\times M
$$
lying over $\varphi_\zeta$. When $Q \subset E$ is a Lagrangian boundary
condition, we have a unique pair $L_{\zeta,\pm}$
of Lagrangian submanifolds of $M$ such that
$$
\Phi_\zeta(Q) = [-1,0) \times L_{\zeta,-} \cup (0,1] \times L_{\zeta,+}.
$$
For the given ordered chain of Lagrangian boundaries
$\CQ = (Q_0,Q_1,\cdots, Q_k)$, denote $Q = \cup_i Q_i$. Then
we require the unique pair $L_{\zeta_i,\pm}$ at $\zeta_i$ to be
$$
\Phi_{\zeta_i}(Q) = [-1,0) \times L_i \cup (0,1] \times L_{i+1}.
$$
at each $\zeta_i$. In this way, for each given $(\pi: E \to \Sigma; \CQ)$,
we consider the moduli space
$$
\CM_J(E,\CQ;\vec p;B)
$$
of \emph{smooth} $J$-holomorphic sections for each section class $B \in \pi_2(E,\CQ;\vec p)$.

The following lemma immediately follows from definition.
\begin{lem}\label{actionB}
For $[p_i,w_i]$ with $i = 0, \cdots, k$ realizing
$B = (-[w_0])\# ([w_1] \# \cdots \# [w_k])$, we have the identity
\be\label{eq:actionwk}
\int w_0^*\omega = \sum_{i=1}^{k} \int w_i^*\omega - \Omega(B)
\ee
or equivalently
\be\label{eq:OmegaB}
\Omega(B) = \sum_{i=1}^{k} \int w_i^*\omega - \int w_0^*\omega.
\ee
\end{lem}
Since any section $s$ in class $[s,\del s] = B$ that satisfies
the asymptotic condition $s(z_i) = x_i$, with $x_i\in Q$ with
$x_i = (p_i,u(p_i)$ in the trivialization,
satisfies
$\int s^*\Omega = \Omega(B)$, Lemma \ref{energy}, \ref{actionB}
imply
\be\label{eq:Es}
\frac{1}{2}\int \|Ds\|^2 = \sum_{i=1}^{k} \int w_i^*\omega - \int w_0^*\omega
+ \lambda < \infty.
\ee
Once we have this energy estimate, it follows by a standard
compactness argument that there are only finitely many section
classes $B$ (and so $\CM_J(E,\CQ;\vec p;B)$)
is non-empty such that $\CM_J(E,\CQ;\vec p;B)$.

Finally we state the Gromov-Floer type compactness of $\CM_J(E,\CQ;\vec p;B)$
in a precise way for the later purpose.
For this we introduce additional marked points in the interior of
$\Sigma$ and on the boundary $\del \Sigma$ besides the punctures $\vec\zeta$.
\begin{defn}\label{config} A \emph{configuration} on $\dot\Sigma$
is the set of of finite points consisting of
\beastar
(x_j) & \in & \operatorname{Int}\Sigma \quad \mbox{for $j = 1, \cdots, m$} \\
(y^{(i)}_{j}) & \in & \del_i \dot \Sigma \quad \mbox{for $j = 1, \cdots, n_i$ and
for $i = 0, \cdots, k$}
\eeastar
We denote by $\widetilde C_{m;\vec n}$ the set of such configurations.
\end{defn}

We denote such a configuration by
$$
C = \left(\{(x_j)\}_{1\leq j \leq m}; \{(y^{(i)}_j\}_{1 \leq j
\leq n_i},\, 0 \leq i \leq k \right)
$$
in general. We note that there is no non-trivial holomorphic automorphisms
of $C$ except the following cases:
\begin{itemize}
\item $\#(\vec \zeta) = 0$ and $2m + \sum_{i=0}^k n_i \leq 2$
\item $\#(\vec \zeta) = 1$, $m=0$ and $k=1,\, n_1=1$.
\end{itemize}
We consider the pairs
$$
(s,C)\in \widetilde \CM_J(E,\CQ;\vec p;B) \times \widetilde C_{(m;\vec n)},
\quad \vec n = (n_0, n_1, \cdots, n_k)
$$
There is a natural $Aut(\dot \Sigma)$-action on the product
$\widetilde \CM_J(E,\CQ;\vec p;B) \times \widetilde C_{(m;\vec n)}$ defined by
$$
(s,C) \mapsto (s \circ \phi^{-1}, \phi(C))
$$
where $\phi \in Aut(\dot \Sigma)$.
We define $\CM_{J,(m;\vec n)}(E,\CQ;\vec p;B)$ to be the quotient
$$
\CM_{J,(m;\vec n)}(E,\CQ;\vec p;B) = \widetilde
\CM_J(E,\CQ;\vec p;B)\times \widetilde C_{(m;\vec n)}/Aut(\dot \Sigma).
$$
We note that when $C \neq \emptyset$
we have the natural evaluation maps
$$
ev:\CM_{(m;\vec n)}(E,\CQ;\vec p;B) \to E^m \times
\prod_{i=0}^k Q_i^{n_i}.
$$
which respect the above mentioned $Aut(\dot \Sigma)$-action and so well-defined.

We denote by $\overline \CM_1(E_z,J_z;\alpha_z)$ the stable
maps of genus 0 with one marked point, and by $\overline \CM_1(E_z,Q_z,J_z;\beta_z)$
the set of bordered stable maps with one marked point at a boundary
of the disc. We consider the fiber product
\bea
&\,&\widetilde \CM_{(m;\vec n)}(E,\CQ;\vec p;B_0;\{\alpha_i\}, \{\beta^{i}_j\}\})
: = \widetilde \CM_{(m;\vec n)}(E,\CQ;\vec p;B_0) \nonumber \\
&{}& \hskip0.3in {}_{ev}\times_{ev}
\left(\Pi_i \CM(E_{z_i},J_{z_i};\alpha_i) \times
\Pi_{i=0}^k \Pi_{j=1}^{n_i} \CM(E_{z_i},Q_{z_i},J_{z_i};\beta^{(i)}_j)\right)
\label{eq:tildeMMB0}
\eea
with respect to the obvious evaluation maps.

\subsection{Bubble may hit critical points}
\label{subsec:bubble}

In this subsection, we analyze the failure of convergence of
a sequence of \emph{smooth} pseudo-holomorphic sections $\CM_J(E,\CQ;\vec p;B)$
with
\be\label{eq:vir-dim}
\operatorname{vir.dim} \CM_J(E,\CQ;\vec p;B) = 0
\ee
of symplectic Lefschetz fibrations $E \subset \Sigma$ with the asymptotic
condition provided by $\vec p$. We will be especially
interested in bubble components passing through critical points.

By definition, the pull-back $\Xi^*J$ is the standard complex
structure on $\C^{n+1}$ and $\xi_*j$ the standard one on $\C$ and
$\xi^{-1} \circ \pi \circ \Xi = q$ on the
neighborhoods $W \subset \C^{n+1}$, $U \subset \C$ of the origins
respectively corresponding to the given holomorphic Morse chart $(\xi,\Xi)$ at
each $x_0 \in E^{crit}$. We denote by $W_{x_0}\subset E$, $U_{x_0}\subset \Sigma$
be the corresponding neighborhoods of $x_0$ and $\pi(x_0)$ respectively.
We also denote $B^{2n+2}(r)$ the ball of radius $r>0$ in $\C^{n+1}$
with its center at the origin, and $B_{x_0}(r)$ its image under
$\Xi$ at $x_0 \in E^{crit}$, and similarly for $B^2(\e)$ and $B_{z_0}(\e)$.

Since we assume that there are finitely
many critical points and different critical points lie in
different fibers of $\pi: E \to \Sigma$, we have constants $\e, \, r > 0$
such that
\be\label{eq:epsilon}
B_{\pi(x)}(\e) \cap B_{\pi(x')}(\e) = \emptyset \quad \mbox{for $z \neq z'$ with
$x, \, x' \in E^{crit}$ }
\ee
and
\be\label{eq:r}
\pi(B_{x_0}(r)) \supset B_{z_0}(\e) \quad \mbox{for all $x_0 \in E^{crit}$}.
\ee
\begin{lem}\label{lem:away-crit}
The graph of any \emph{differentiable} section
does not intersect $E^{crit}$.
\end{lem}
\begin{proof}
Since $s$ is a section, we have $\pi \circ s = id$. By differentiating this, we obtain $d\pi \circ
d s = Id$. In particular, $d\pi$ is surjective at any point
$s(z)$, i.e., $s(z)$ must be a regular point of $\pi$ and hence
$s(z) \in E \setminus E^{crit}$.
\end{proof}

Obviously we have the inequality
\be\label{eq:outside}
\dist(s(z),E^{crit}) \geq C>0  \quad \mbox{for }\, z \in \Sigma \setminus
\cup_{x \in E^{crit}} B_{\pi(x)}(\e)
\ee
where $C = C(\e,(\xi,\Xi))$ is a constant depending only on
$\e$ and the holomorphic Morse chart $(\xi,\Xi)$ independent of
$s$.
Now we consider the restriction of $s$ on $\cup_{x \in E^{crit}} B_{\pi(x)}(\e)$.
On this neighborhoods, we can identify the section $s$ to the
holomorphic map
$$
f: B^2(\e) \subset \C \to B^{2n+2}(r) \subset \C^{n+1}
$$
satisfying $q \circ f(z) = z$ for all $z \in B^2(\e)$, i.e.,
$$
f_1^2(z) + \cdots + f_{n+1}^2(z) = z.
$$
In particular, we have
$$
|f(z)|^2  = \sum_{j=1}^{n+1} |f_j^2(z)| \geq
|f_1^2(z) + \cdots + f_{n+1}^2(z)| = |z|
$$
for all $z \in B^2(\e)$. Therefore we obtain
\be\label{eq:bdy}
|f(z)| \geq \sqrt{\e} \quad \mbox{for $z \in \del B^2(\e)$}.
\ee

The following theorem is the main theorem proved in this section.

\begin{thm}\label{thm:maximum} Suppose that the Lefsechetz Hamiltonian
fibration with Lagrangian boundary $Q \subset E|_{\del \Sigma}$ such
that $E$ is relative Calabi-Yau and $Q$ with vanishing fiberwise
Maslov class. Then there exists a dense subset of $j$-compatible
$J$'s such that for any such $J$, there exists a constant $C > 0$
depending only on $(E,Q,J,j)$, the section class $[s]$ and $\e > 0$
such that we have \be\label{eq:dist-E}
\operatorname{dist}(\operatorname{Im}s,E^{crit}) \geq C \ee for any
\emph{smooth} section $s: \Sigma \to E$.
\end{thm}

We prove Theorem \ref{thm:maximum} by contradiction. Let
$J$ be any $j$-compatible and suppose that there is a sequence $s_i$
of smooth $J$-holomorphic sections such that \be\label{eq:distsito0}
\min_{z \in \Sigma} \dist(s_i(z),E^{crit}) \to 0. \ee Since
$E^{crit}$ is a finite set and by (\ref{eq:epsilon}), we may choose
critical point $x_0 \in E^{crit}$ and a sequence $z_i \in \Sigma$ so
that \be\label{eq:sizi} \dist(s_i(z_i),x_0) \to 0. \ee By choosing a
subsequence of $z_i$ if necessary, we may assume $z_i \to z_0$ and
so \be\label{eq:ziinB} z_i \in B_{z_0}(\e). \ee for all $i$. By the
Gromov-Floer convergence applied to $J$-holomorphic curves $s_i:
(\Sigma,j) \to (E,J)$, which are also $J$-holomorphic sections,
there exists a subsequence which converges to
$$
s_\infty = s_0 + \mbox{``bubble components''}
$$
where $s_0$ is a smooth section of $E \to \Sigma$ and each bubble
must be either a fiberwise pseudo-holomorphic sphere or a
fiberwise pseudo-holomorphic disc. And each disc bubble has its
boundary lying in the given Lagrangian boundary condition.

Due to the property (\ref{eq:ziinB}), at least one bubble must pass
through the critical point $x_0$ whose image is contained in
$E_{z_0}$. By the connectedness of the image of the limit, this
bubble is contained in a bubble tree rooted at a point $z_1 \in
\Sigma$ in the principal component $(s_0,\Sigma)$. The image of this
bubble tree itself must be contained in the same fiber $E_{z_0}$.
Denote this bubble tree by $(v, (C,z))$, $z \in C$ which is a stable map in $E$
such that
$$
v(z) = s_0(z_1) \in E.
$$
However since $E_{z_0}$ contains a singularity $x_0$ and so
is not a smooth manifold, we need further clarification on
the bubble component passing through $x_0$. Since $\pi: E \to \Sigma$ is
isomorphic to the standard Lefschetz fibration
$$
q(z_1,\cdots, z_{n+1}) = z_1^2 + \cdots + z_{n+1}^2
$$
near $x_0$, there is a well-defined ``multiplicity''
of the component at the critical point $x_0$. We now make this statement
precise in the next subsection.

\subsection{Proper holomorphic curves in $E_{z_0} \setminus \{x_0\}$}
\label{sec:properholomorphic}

Using the holomorphic Morse chart $(\xi,\Xi)$ at the critical
point $x_0 \in E_{z_0} \subset E$, we consider the decomposition
$$
E = B_{x_0}(\delta) \coprod (E \setminus B_{x_0}(\delta))
$$
where $B_{x_0}(\delta) = \Xi^{-1}(B^{2(n+1)}(\delta)) \cap E_{z_0}$
for $0 < \delta < \e$.

Now we consider the hypersurface
$$
q^{-1}(0) = \{ (x_1, \cdots,x_{n+1}) \mid x_1^2 + \cdots + x_{n+1}^2 = 0 \}.
$$
The only singularity of this hypersurface is $0 \in \C^{n+1}$ and so
$q^{-1}(0) \setminus \{0\}$ is a smooth complex hypersurface of $\C^{n+1}
\setminus 0$. We denote by $\theta_{\C^{n+1}}$ the one form
$$
\theta_{\C^{n+1}} = \frac{i}{4} \sum x_k d\overline x_k - \overline x_k dx_k
$$
and the standard K\"ahler form
$$
\omega_{\C^{n+1}} = - d\theta_{\C^{n+1}} (= \sum dq_k \wedge dp_k)
$$
where $x_k = q_k + i p_k$. We denote by $\theta$ and $\omega$ the
restriction of these to $q^{-1}(0) \setminus \{0\}$.

Following \cite{seidel:triangle}, we denote $T = T^*S^n$ and $T(0)
= $ the zero section of $T$ and by $\theta_T$ and $\omega_T = -
d\theta_T$ the standard Liouville one-form and the standard
symplectic form on the cotangent bundle $T$. We identify $T$ with
the subset
$$
\{(u,v) \in \R^{n+1} \times \R^{n+1} \mid \langle u, v \rangle = 0, \,
\|v\| = 1 \}
$$
and then consider the map
$$
\Phi: q^{-1}(0) \setminus \{0\} \to T \setminus T(0)
$$
defined by
\be\label{eq:Phi}
\Phi(x) = (\operatorname{im}(\widehat x) \|\operatorname{re}(\widehat x)\|,
\operatorname{re}(\widehat x)\|\operatorname{re}(\widehat x)\|^{-1})
\ee
where $se^{i \alpha}$ are polar coordinates on the base of $q: \C^{n+1} \to \C$,
and $\widehat x = e^{-i \alpha/2}$.
We note that this map is equivariant with respect to the canonical $O(n+1)$-actions
on $q^{-1} \setminus \{0\} \subset \C^{n+1} \setminus \{0\}$
and $T=T^*S^n$

The following is a consequence of straightforward calculation,
which is a restriction of the identity in p. 1014
\cite{seidel:triangle}.

\begin{lem}
$\Phi$ is a diffeomorphism such that
$$
\Phi^*\theta_T = \theta
$$
and so $\Phi^*\omega_T = \omega$.
In particular the symplectic manifold
$(q^{-1}(0) \setminus \{0\},\omega)$ is symplectomorphic to
the cotangent bundle $T = T^*S^n$.
\end{lem}

This lemma shows that $E_{z_0} \setminus \{x_0\}$ is a symplectic
manifold with \emph{negative} cylindrical end whose asymptotic
boundary is symplectomorphic to the unit co-sphere bundle
$S^1(T^*S^n)$. Furthermore complex structure on $\Xi^{-1}(W)$ is
required to be induced from the standard complex structure from
$q^{-1}(0) \setminus \{0\} \subset \C^{n+1} \setminus \{0\}$.
Therefore any $j$-compatible $J$ provides a translational
invariant almost complex structure with respect to the cylindrical
structure and any $J$-holomorphic curve is genuinely holomorphic
near the end of $q^{-1}(0) \setminus \{0\}$.

In particular, any such
curve converges to a Reeb orbit of $S^1(T^*S^n)$ with a finite
multiplicity. This motivates us to study the moduli problem of
proper $J$-holomorphic curves from $\C \cong \C P^1 \setminus
\{N\}$ with the asymptotic boundary condition given by $(\gamma,
k)$ where $\gamma$ is a simple Reeb orbit of $S^1(T^*S^n)$ and a
multiplicity $k \in \Z_+$. We denote by $\widetilde{\CR}_1(S^1(T^*S^n))$ the set
of parameterized Reeb orbits on $S^1(T^*S^n)$ with period $2\pi$ and
by $\CR_1(S^1(T^*S^n))$ the quotient by the natural $S^1$-action, i.e.,
the set of unparameterized Reeb orbits.

We denote by $(s,\Theta)$ the cylindrical coordinates of
$T \setminus T(0)$ where $s$ and $\Theta$ are defined by
$$
s(q,p) = \|p\|, \quad \Theta(q,p) = \left(q, \frac{p}{|p|}\right).
$$
Note that all geodesics on $S^n$ have the same period
which implies that the contact manifold $S^1(T^*S^n)$ is foliated by
Reeb orbits all of which have the same period.
We denote the corresponding cylindrical coordinates on
$B_{x_0}(\e)\setminus \{x_0\} \subset E \setminus \{x_0\}$
by the same letters $(s,\Theta)$. By a suitable translation
of $s$-coordinates, we may assume the identification
$$
(s,\Theta): B_{x_0}(\e) \setminus \{x_0\}
\to (-\infty,0] \times S^1(T^*S^n).
$$

Now we are ready to define the moduli space of our interest.
For the simplicity of notation, we denote
$$
E_{z_0} \setminus \{x_0\} = E^*_{z_0},
$$
and $\dot S = S \setminus \{z_0\}$ is
either an open Riemann surface isomorphic to $\C$ or
an open Riemann surface with boundary isomorphic to $\C \setminus D^1(1)$.
We fix an analytic coordinates $z = e^{\tau + it}$ near $z_0 \in $.
Let $u: \dot S \to E_{z_0} \setminus \{x_0\}$ be a pseudo-holomorphic curve
with Lagrangian boundary condition
$$
u(\del S) \subset Q_{z_0} \subset E_{z_0} \setminus \{x_0\}.
$$
Since the treatment of the latter is essentially similar to the former, we will
focus on the former case in the following exposition. We will briefly mention
the latter case in the end of our discussion.
\par
By the properness and exponential convergence property of $u$ as $\tau \to -\infty$,
we have
$$
\operatorname{im} u|_{(-\infty,-R] \times S^1} \subset B_{x_0}(\e)
$$
for a sufficiently large $R > 0$. It is proved in \cite{hofer} that
$$
T = \lim_{\tau \to -\infty} \int (\Theta\circ u_\tau)^*\lambda
$$
with $T = 2\pi k$ for some integer $k \geq 1$ where $u_\tau(t) := u(\tau,t)$.
Then it is proved in \cite{hofer,bourgeois} that there exist constants $C, \, \delta > 0$
depending only on $(S^1(T^*S^n), \lambda)$
such that
$$
\lim_{\tau -\infty} \operatorname{dist}(u(\tau/T,t/T),u_\gamma(\tau+\tau_0, t+t_0)) \leq C e^{-\delta |\tau|}
$$
for some simple Reeb orbit $\gamma$ of $S^1(T^*S^n)$ and $\tau_0 \in \R$ and
$t_0 \in S^1$. (See also \cite{fooo:chap10} for a proof of similar exponential
convergence result in the relative context.) Here
$u_\gamma: [0,\infty) \times S^1 \to (-\infty,0] \times S^1(T^*S^n)$
denotes the trivial cylinder map $u_\gamma(t) = (\tau,\gamma(t))$.

By the above discussion, we can now define the following moduli spaces for each given integer $k \geq 1$ and
a homotopy class $A$.

\begin{defn}\label{modulik}
Let $\gamma\in \CR_1(S^1(T^*S^n))$ and $k \in \Z_+$.
For each given $(\gamma,k)$, we define
\beastar
\CM_{z_0}^{SFT}(E^*,J_0,\gamma;A,k)
& = & \{ u: \dot \Sigma \to E^*_{z_0} \mid \delbar_{J_0} u = 0, \quad
\int u^*\Omega < \infty, \\
&{}& \quad \lim_{\tau \to -\infty} u(\tau/2\pi k,t/2\pi k) = u_\gamma(t),
[u] = A \}
\eeastar
We then define
$$
\CM_{z_0}^{SFT}(E^*,J_0;A,k) = \bigcup_{\gamma \in \CR_1(S^1(T^*S^n))}
\CM_{z_0}^{SFT}(E^*,J_0, \gamma;A,k).
$$
\end{defn}

The following general index formula can be derived from Corollary 5.4 \cite{bourgeois}.
In this regard, we note that the dimension of the space $\CR_1(S^1(T^*S^n))$ of simple Reeb orbits
of $S^1(T^*S^n)$ is $n$.

\begin{prop} Let $u \in \CM_{z_0}^{SFT}(E^*,J_0;A,k)$.
Then we have
\be \label{eq:index}
\operatorname{Index} D_u \delbar_J  = \left(- \mu_{CZ}(\gamma) + \frac{n}{2}\right)
+ (n-3) + 2 c_1(u;\phi_\gamma)
\ee
where $\mu_{CZ}$ is the generalized Conley-Zehnder index defined by Robbin and Salamon
\cite{robbin-sal}.
\end{prop}
For the reader's convenience,
we provide the precise definitions of $\mu_{CZ}(\gamma)$, and $c_1(u;\phi_\gamma)$
in the Appendix. Once the definitions are made precise, its proof
follows from that of \cite{bourgeois}.

The following transversality result is an easy consequence of
the standard argument whose proof is omitted: We first note that we have
$$
\mu_{CZ}(\gamma) = - \operatorname{Morse}(\gamma) + \frac{\dim \CR_{sim}}{2}
$$
for the Reeb orbit at the \emph{negative} end. (See e.g., \cite{mohnke}
and Corollary 1.7.4 \cite{EGH} for such a formula.)

\begin{prop}\label{transversality} Let $x_0 \in E^{crit}$.
There exists a dense subset $\CJ^{tr}(x_0)$ of the set
$\CJ(j)$ of $\Omega$-compatible almost complex structures
such that $\CM^{SFT}(E^*,J_0;A,k)$ is Fredholm-regular and so
becomes a smooth manifold of dimension
\be\label{eq:dimMsft}
- \operatorname{Morse}(\gamma) + (n-3) + 2 c_1(u;\phi_\gamma)
\ee
for any $k \geq 1$.
\end{prop}

An immediate corollary of this proposition is the following
vanishing result.

\begin{cor}\label{nobubbleatx0} Let $x_0 \in E^{crit}$ and
$E^* = E \setminus \{x_0\}$.
Suppose the relative Maslov class of $E \to \Sigma$ is zero.
Then for any $J_0 \in \CJ^{tr}(x_0)$,
$\CM^{SFT}(E^*,J_0;A,k) = \emptyset$ for all $A$ and $k$.
\end{cor}
\begin{proof} By the assumption, we have $c_1(A) = 0$ for all
$A \in \pi_2(E_{z_0}^*;z_0)$. Furthermore the Morse index for
the simple closed geodesic of $S^n$ is given by $n-1$ and greater than
$n-1$ for multiple geodesics. Therefore we derive
$$
\dim \CM^{SFT}(E_{z_0}^*,J_0;A,k) \leq -(n-1) + n-3 = -2
$$
for all $k \geq 1$ from (\ref{eq:dimMsft}) and hence the proof.
\end{proof}

Finally we briefly mention the case in which $\dot S \cong \C \setminus D^2(1)$ and
$u(\del S) \subset Q_{z_0}$. Since we assume that $Q$ has vanishing
fiberwise Maslov class, the corresponding moduli space
$$
\CM^{SFT}(E_{z_0}^*,Q_{z_0},J_0;A,k)
$$
has its dimension exactly same as that of $\CM^{SFT}(E_{z_0}^*,J_0;A,k)$.
Therefore the same dimension counting argument applies in the exactly same way.
(See e.g., \cite{mohnke} for a discussion on the dimension formula in the context
with Lagrangian boundary.)
\par
We would like to note that the above dimension counting argument strongly
relies on the fact that $E$ is (fiberwise) Calabi-Yau so that $c_1(u;\phi_\gamma) = 0$
and $Q$ has vanishing fiberwise Maslov class.
It seems to be interesting to investigate how the long
exact sequence will be transformed in other contexts like in the Fano case.

\subsection{Bubble does not hit critical points}

In this subsection, we restrict ourselves to the case of vanishing
relative first Chern class $c_1(E) = 0$.

We prove Theorem \ref{thm:maximum} in this subsection.
In fact, it will be enough to take
$$
\CJ^{tr}: = \bigcap_{x_0 \in E^{crit}} \CJ^{tr}(x_0)
$$
for the dense subset of $j$-compatible $J$'s.

\begin{proof}[Proof of Theorem \ref{thm:maximum}]
Let $J \in \CJ^{tr}$ defined as above.

We have derived before by the Gromov-Floer compactness applied to
$\CM_J(E,\CQ;\vec p;B)$ in $E$
that there exists a subsequence which converges to
$$
s_\infty = s_0 + \mbox{``bubble components''}
$$
where $s_0$ is a smooth section of $E \to \Sigma$ and each bubble tree
is contained in a fiber which consists of either fiberwise pseudo-holomorphic spheres
or discs. By Corollary \ref{nobubbleatx0}, there cannot be any bubble tree
passing through a critical point of $E$ and so contained in $E \setminus E^{crit}$.
Since the principal component $s_0$, which is smooth, cannot pass through a critical point
by Lemma \ref{lem:away-crit}, we have proved that the Gromov-Floer limit of
$s_i: \Sigma \to E \setminus E^{crit}$ does not pass any critical point
of $E$. Therefore the compactification $\overline{\CM}_J(E,\CQ;\vec p;B)$
of $\CM_J(E,\CQ;\vec p;B)$ has image contained in $E \setminus E^{crit}$.
\par
Once we have achieved this, the rest of the proof follows, in the same way as
in the case of \emph{smooth} Hamiltonian fibration, by
the standard dimension counting argument from the fact that
$E \to \Sigma$ is a Hamiltonian Lefschetz fibration with vanishing
relative Maslov class: We consider the evaluation maps
$$
ev:\CM_{(m;\vec n)}(E,\CQ;\vec p;B) \to E^m \times
\prod_{i=0}^k Q_i^{n_i}.
$$
and
$$
ev_1: \CM_1(E_z,J_z;\alpha_z) \to E_z, \quad ev_1:
\CM_1(E_z,Q_z,J_z;\beta_z) \to E_z
$$
for $z \in \Sigma \setminus \Sigma^{crit}$ and consider the fiber product
\bea
&\,&\widetilde \CM_{(m;\vec n)}(E,\CQ;\vec p;B_0;\{\alpha_i\}, \{\beta^{i}_j\}\})
: = \widetilde \CM_{(m;\vec n)}(E,\CQ;\vec p;B_0) \nonumber \\
&{}& \hskip0.3in {}_{ev}\times_{ev_1} \left(\Pi_i
\CM(E_{z_i},J_{z_i};\alpha_i) \times \Pi_{i=0}^k \Pi_{j=1}^{n_i}
\CM(E_{z_i},Q_{z_i},J_{z_i};\beta^{(i)}_j)\right) \eea with respect
to the obvious evaluation maps.

Note that the dimension of the moduli space of holomorphic sphere in
any class in a fiber $E_z \setminus E^{crit}$ has virtual dimension
given by $2n-6$ and so
$$
\operatorname{vir.dim} \bigcup_{z \in \Sigma} \CM_1(E_z,J_z;\alpha_z) = 2n-4.
$$
In particular for a generic choice of $J$ all somewhere injective
holomorphic sphere is regular and hence the condition $c_1 = 0$,
which implies semi-positivity, the standard argument from
\cite{ruan-tian} implies that the evaluation image of the moduli
space in each fiber has at least codimension 4 which obviously
avoids the images of pseudo-holomorphic sections in the classes
whose associated moduli space has dimension zero. This proves that
there cannot be any bubble passing through critical points of $E$ in
the limit and hence the above fiber product, for a generic choice of
$J$, becomes empty. The relevant Fredholm theory needed to carry out
this kind of dimension counting argument is by now standard. We
refer to section 2 \cite{seidel:triangle} for an elegant exposition
on this Fredholm theory in the context of exact Lefschetz fibrations
which applies to the current context of Lefschetz Hamiltonian
fibrations without change. This proves that there exist a dense
subset $\CJ^{reg,tr} \subset \CJ^{tr}$ for any given $(\vec p;B)$
the moduli spaces $\CM_{(m;\vec n)}(E,\CQ, J;\vec p;B)$ is compact
for $J \in \CJ^{reg,tr}$.

In particular there exists a constant $C = C(E,Q,\vec p;B) > 0$
$$
\min_{z \in \Sigma} \dist(s_\infty(z),E^{crit}) > C.
$$
But this together with Hausdorff-convergence the image of $s_i$ to
that of $s_\infty$ contradicts to the assumption $s_i(z_i) \to x_0
\in E^{crit}$. This finishes the proof.
\end{proof}

This proposition shows that as far as compactness property of
the set of smooth pseudo-holomorphic sections is concerned,
we can ignore presence of critical points of the fibration $E \to \Sigma$.

\subsection{Gromov-Floer moduli space of $J$-holomorphic trajectories}
\label{subsec:GFmoduli}

With Theorem \ref{thm:maximum} at our disposal,  we can safely ignore the critical points
in the study of \emph{smooth} $J$-holomorphic sections and their degenerations
for the Lefschetz Hamiltonian fibrations $\pi: E \to \Sigma$.
This, together with the energy estimates derived from
the previous section, makes the study essentially the same as for the case of
usual smooth Hamiltonian fibrations (without critical points)
as studied in \cite{entov}, \cite{mcduff-sal04}.

We formulate the definition of Fukaya-Oh-Ohta-Ono's $\mathfrak m_k$-maps in the setting of
fibrations. Obviously this discussion can be extended to the Lefschetz Hamiltonian
fibration $\pi: E \to \Sigma$ with punctures $\vec \zeta = \{\zeta_0, \cdots, \zeta_k\}
\subset \del \dot\Sigma$ and with a chain of Lagrangian boundary conditions
$$
\CQ = (Q_0,\cdots, Q_k)
$$
by considering the anchor caps attached to $[p_{ij},w_{ij}]$ with $p \in L_i \cap L_j$:
Here $Q_j$ is the parallel fiberwise Lagrangian submanifolds corresponding to $L_j$.
We decompose
$$
\del \dot\Sigma = \coprod_{j=0}^k \del_j \Sigma
$$
where $\del_j\Sigma$ is the $j$-th connected component of
$\del \dot \Sigma$. We briefly add some necessary modification
from that of \cite{fooo:book} to accommodate the possible critical point
in the fibration $\pi: E \to \Sigma$ following the notations from
\cite{seidel:triangle}.

For a given compact surface $\Sigma$ with marked points $\vec\zeta$,
we consider the corresponding surface $\dot \Sigma =
\Sigma \setminus \vec \zeta$ with punctures. We fix a given preferred
holomorphic chart $\varphi_\zeta: D_\zeta \subset \Sigma \to D^+$ of the half disc
$D^+ = D \cap \{ \operatorname{im}(z) \geq 0\}$ with $\varphi_\zeta(\zeta) = 0$.
We considering the moduli space
$$
\CM_J(E,\CQ;\vec p;B)
$$
as in subsection \ref{subsec:GF-moduli} for all the section class $B \in \pi_2(E,\CQ;\vec p)$ with
$$
\operatorname{vit.dim}\CM_J(E,\CQ;\vec p;B) = 0.
$$
More specifically, the
map is supposed to be given by
$$
\Phi_0^{rel}(E,\pi;\CQ)(\otimes_{j=1}^k [p_j,w_j]) = \sum_k \#
(\CM_J(E,\CQ;\vec p;B)) [p_k, w_k]
$$
with a suitable definition of the coefficient $\#(\CM_J(E,\CQ;\vec p;B))$.

Finally we recall the notion of broken Floer trajectory moduli
spaces i.e., the case corresponds to $k+1 = 2$. See \cite{fon} for the corresponding definition for the
closed case.

\begin{defn}\label{brokentraj} Let $J = \{J_t\}_{ 0 \leq t \leq 1}$ and
$x,\, y, \in L_+\cap L_-$.
A {\it stable broken Floer trajectory} from $p$ to $q$ is a triple
$$
u = \left( (u_1, \cdots, u_a); (\sigma_1,\cdots \sigma_m),
(\gamma^0_1,\cdots, \gamma^0_{n_0}), (\gamma^1_1,\cdots, \gamma^1_{n_1});
o\right)
$$
that satisfies the following:
\begin{enumerate}
\item For $i = 1,\cdots, a-1$, $u_i \in \CM(x_i,x_{i+1})$ and satisfy
\bea
u_1(-\infty) & = & p,\, u_a(\infty) = q \nonumber \\
u_i(\infty) & = & u_{i+1}(-\infty) \quad \mbox{for $i = 1, \cdots, a-1$}.
\label{eq:matching}
\eea
We call (\ref{eq:matching}) the matching condition and say a pair
$(u,u')$ of Floer trajectories {\it gluable} if it satisfies the
matching condition.
\item $\sigma_i \in \overline \CM_1(J_{t_i};\alpha_i)$ for $i = 1, \cdots, m$
\item $\gamma_j^0 \in \overline \CM_1(L_0, J_0;\beta^0_j)$ for $j = 1, \cdots, n_0$
and $\gamma_k^1 \in \overline \CM_1(L_1, J_1;\beta^1_k)$ for $k = 1, \cdots, n_1$.
\item For each $\ell = 1, \cdots, a$, either the map $u_\ell$ is non-stationary
or $\Theta_\ell \cap \mbox{Im}o \neq \emptyset$.
\end{enumerate}
\end{defn}

We denote the domain of $u$ simply by $\Theta_u$ which is the product of
a broken configuration. This is the union of a finite copies of
$$
\R \times [0,1],
$$
the {\it principal components},
and the pre-stable curves of closed or bordered Riemann surfaces of genus 0
the {\it bubble components}
with their roots attached to the principal components of $\Theta_u$.
Each broken Floer trajectory $u: \Theta_u \to M$ can be
regarded as a broken Floer trajectory into the fiber $M_\zeta$
for a $\zeta \in \vec \zeta$. We denote by
$\widetilde \CM(p,q;B_0;\{\alpha_i\}, \{\beta_{0,j}\}, \{\beta_{1,k}\})$
the set of stable broken Floer trajectories in the prescribed topological
types. The group $Aut(\dot \Sigma)$ acts on the moduli space
$\widetilde \CM(p,q;B_0;\{\alpha_i\}, \{\beta_{0,j}\}, \{\beta_{1,k}\})$
by the simultaneous translation of the roots of the bubbles attached to
the principal component. Then we denote by
$\CM(p,q;B)$ the set of all stable broken trajectories in class $B \in \pi_2(p,q)$.

We then define
$$
\overline \CM(E,\CQ;\vec p;B) = \coprod \CM_{(m;\vec n)}(E,\CQ;\vec
p';B'; \{\alpha_i\}, \{\beta^{i}_j\})\# \left(\prod_i
\CM(p'_i,q_i;B_i)\right)
$$
for all choices of $B_0, \{\alpha_i\}, \{\beta^{i}_j\}$ and $B_i$'s
satisfying
$$
B = B_0 + \sum_i \alpha_i + \sum_{i=0}^k \left(\sum_{j=1}^{n_i} \beta^{(i)}_j\right)
+ \sum B_i
$$
and provide it with a topology of stable maps.
We denote the corresponding decomposition of maps by
$$
s = s^0 \# (\Pi_i v_i) \# \left(\#_{i=0}^k (\Pi_j w^{(i)}_j)
\right) \# (u_i)
$$
and call $s^0$ the {\it principal component} and other fiberwise curves
the {\it bubble components}.

At this stage, we would like to emphasize that this compactification
is defined as a topological space for {\it any} choice of $(j,J)$
for a transversal chain $(L_0,\cdots, L_k)$ of Lagrangian
submanifolds. The topological space $\overline  \CM(E,\CQ;\vec p;B)$
will not be a smooth `manifold' even for a generic choice of $J$,
but will be a space with Kuranish structure \cite{fon}.

\section{Anchored Floer cohomology; review}
\label{sec:fooo}

In this subsection, we recall the exposition from \cite{fooo:anchor}
on the Lagrangian Floer theory of anchored Lagrangian submanifolds.

\subsection{Floer chain complex}
\label{chaincomplex}

Let $(L_0,L_1)$ be a pair with $L_0$
intersecting $L_1$ transversely.

Let $(L_i,\gamma_i)$ $i=0,1$ be anchored Lagrangian submanifolds.
Let $p, \, q \in L_0 \cap L_1$ be admissible
intersection points. We
defined the set $\pi_2(p,q)=\pi_2((L_0,L_1),(p,q))$
in section \ref{sec:anchored}.
We also defined $\pi_2(\ell_{01};p)$ there. We now define:
\begin{defn} Let $R$ be the underlying coefficient field. We define
$CF(\CL_1,\CL_0) = CF((L_1,\gamma_1),(L_0,\gamma_0))$
to be a free $R$-module over the basis $[p,w]$
where $p \in \CL_0\cap \CL_1$ is an admissible intersection points and
$w$ is a map from $[0,1]^2 \to M$ connecting $\ell_{01}$ and $\widehat p$.
\end{defn}
Here $R$ is a ground ring such as $\Q$, $\C$ or $\R$.

\begin{rem}\label{labelPiaction}
We remark that the set of $[p,w]$ where $p$ is the admissible intersection
point is identified with the set of
the critical point of the action functional $\mathcal A$ defined on the
Novikov covering space of $\Omega(L_0,L_1;\ell_{01})$.
The group $\Pi(\CL_0,\CL_1)$ defined in Section \ref{subsec:gamma-equiv} acts freely on it
so that the quotient space is the
set of admissible intersection points.
\end{rem}
We next take a grading $\lambda_i$ to $(L_i, \gamma_i)$ as in
Subsection \ref{subsec:anchored}.
It induces a grading of $[p,w]$ given by $\mu([p,w];\lambda_{01})$, which gives the
graded structure on $CF(\CL_1,\CL_0)$
$$
CF(\CL_1,\CL_0) = \bigoplus_k CF^k(\CL_1,\CL_0;\lambda_{01})
$$
where
$
CF^k(\CL_1,\CL_0;\lambda_{01}) = \operatorname{span}_R\{[p,w] \in CF(\CL_1,\CL_0)\mid
\mu([p,w];\lambda_{01}) = k\}.
$
\par
For given $B \in \pi_2(p,q)$, we denote by ${Map}(p,q;B)$
the set of such $w$'s in class $B$.
\par
We summarize the extra structures added in the discussion of
Floer homology for the anchored Lagrangian submanifolds in the following:

\begin{enumerate}
\item We assume that $(L_0,L_1)$ is a relatively spin pair.
We consider a pair $(L_0,\gamma_0)$, $(L_1,\gamma_1)$ of
anchored Lagrangian submanifolds and the base path
$\ell_{01} = \overline\gamma_0*\gamma_1$.
\item We fix a grading $\lambda_i$ of $\gamma_i$ for $i=0, \, 1$,
which in turn induce a grading of $\ell_{01}$,
$\lambda_{01} = \overline{\lambda_0} * \lambda_1$.
\item
We fix an orientation $o_{p}$ of $\operatorname{Index}\,\delbar_{\lambda_{p}}$
for each $p\in L_0\cap L_1$ as in \cite{fooo:anchor}.
\end{enumerate}

Under these conditions, orientation of the Floer moduli space $\CM(p,q;B)$
is induced. Using virtual fundamental chain
technique \cite{fon}, Appendix A.1 \cite{fooo:book}
we can take a system of multisections and obtain a system of \emph{rational} numbers
$n(p,q;B) = \#(\CM(p,q;B))$
whenever the virtual dimension of $\CM(p,q;B)$ is zero.
Finally we define the Floer `boundary' map $\partial: CF(\CL_1,\CL_0) \to
CF(\CL_1,\CL_0)$ by the sum
\begin{equation}\label{eq:boundary}
\partial ([p,w]) = \sum_{q \in L_0\cap L_1}\sum_{B \in \pi_2(p,q)}
n(p,q;B) [q,w\# B].
\end{equation}
By Remark \ref{labelPiaction}, $CF(\CL_1,\CL_0)$ carries a natural
$\Lambda(\CL_0,\CL_1)$-module structure and $CF^k(\CL_1,\CL_0;\lambda_{01})$
a $\Lambda^{(0)}(\CL_0,\CL_1)$-module structure where
$$
\Lambda^{(0)}(\CL_0,\CL_1) = \left\{ \sum a_g [g] \in
\Lambda(\CL_0,\CL_1) \Big| \mu([g]) = 0 \right\}.
$$
We define
\begin{equation}\label{CFpair}
C(\CL_1,\CL_0)
= CF(\CL_1,\CL_0) \otimes_{\Lambda(\CL_0,\CL_1)}
\Lambda_{nov}
\end{equation}
where we use the embedding $\Lambda(\CL_0,\CL_1) \to \Lambda_{nov}$ given
in \eqref{eq:nov-embedding}.
\par
We write the $\Lambda_{nov}$ module (\ref{CFpair}) also as
$$
C(\CL_1,\CL_0;\Lambda_{nov}).
$$
\begin{defn}\label{efilt}
We define the {\it energy filtration}
$
F^{\lambda}CF(\CL_1,\CL_0)
$
of the Floer chain complex
$CF(L_1,\gamma_1),(L_0,\gamma_0))$ (here $\lambda \in \R$)
such that $[p,w]$ is in $F^{\lambda}CF(\CL_1,\CL_0)$
if and only if $\mathcal A([p,w]) \ge \lambda$.
\end{defn}
This filtration also induces a filtration on (\ref{CFpair}).
\begin{rem}
We remark that this filtration depends (not only of the homotopy class of) but also of $\gamma_i$ itself.
\end{rem}
It is easy to see the following from the definition of $\partial$ above:
\begin{lem}\label{filtpres}
$$
\partial \left(F^{\lambda}CF((L_1,\gamma_1),(L_0,\gamma_0))
\subseteq F^{\lambda}CF((L_1,\gamma_1),(L_0,\gamma_0))\right).
$$
\end{lem}
\par
According to the definition
(\ref{eq:boundary}) of the map $\partial$, we have the formula for
its matrix coefficients
\begin{equation}\label{eq:B=B1B2}
\langle \partial\partial [p,w], [r,w\# B] \rangle = \sum_{q \in L_0\cap L_1}
\sum_{B = B_1\# B_2 \in \pi_2(p,r)} n(p,q;B_1)n(q,r;B_2)T^{\omega(B)}
\end{equation}
where $B_1 \in \pi_2(p,q)$ and $B_2 \in \pi_2(q,r)$.

To prove, $\partial \partial = 0$, one needs to prove
$\langle \partial\partial [p,w], [r,w\# B]\rangle = 0$
for all pairs $[p,w], \, [r,w \# B]$.
On the other hand it follows from definition that each summand
$$
n(p,q;B_1)n(q,r;B_2)T^{\omega(B)} = n(p,q;B_1)T^{\omega(B_1)}n(q,r;B_2)T^{\omega(B_2)}
$$
and the coefficient $n(p,q;B_1)n(q,r;B_2)$
is nothing but the number of broken trajectories lying in
$\CM(p,q;B_1) \# \CM(q,r;B_2)$.
This number is nonzero in the general
situation we work with.

To handle the problem of obstruction to $\del\circ \del = 0$ and of
bubbling-off discs in general, a structure of filtered $A_\infty$
algebra $(C, \mathfrak m)$ \emph{with non-zero $\mathfrak m_0$-term}
is associated to each Lagrangian submanifold $L$
\cite{fooo00,fooo:book}.
\subsection{$A_{\infty}$ algebra}
\label{subsec:objects}
In this subsection, we review the notion and construction of
filtered $A_{\infty}$ algebra associated to a Lagrangian
submanifold.
In order to make the construction consistent to one
in the last section, where $\Lambda(\CL_0,\CL_1)$
is used for the coefficient ring rather than the universal
Novikov ring, we rewrite them
using smaller Novikov ring $\Lambda(L)$ which we define below.
Let $L$ be a relatively spin Lagrangian submanifold.
We have a homomorphism
$$
(E,\mu): H_2(M,L;\Z) \to \R \times \Z
$$
where
$
E(\beta) = \beta \cap [\omega]
$
and $\mu$ is the Maslov index homomorphism.
We put $g\sim g'$ for $g,g' \in H_2(M,L;:\Z)$ if
$E(g) = E(g')$ and $\mu(g) = \mu(g')$.
We write $\Pi(L)$ the quotient with respect to this equivalence
relation. It is a subgroup of $\R\times \Z$.
We define
\beastar
\Lambda(L)
= \Big\{ \sum c_g[g] & \Big|& g \in \Pi(L), c_g \in R,
E(g) \ge 0, \\
&{}& \forall \, E_0 \, \# \{g \mid c_g \ne 0, E(g) \le E_0\} < \infty \Big \}
\eeastar
We have the natural embedding
$
\Lambda(L) \to \Lambda_{0,nov}
$
similarly as in \eqref{eq:nov-embedding}.
\par
Let $\overline C$ be a graded $R$-module and $CF = \overline C
\widehat{\otimes}_R \Lambda(L)$. Here and hereafter we use symbol $CF$ for the modules over
$\Lambda(L)$ or $\Lambda(L_0,L_1)$ and
$C$ for the modules over the universal Novikov ring.
\par
We denote by $CF[1]$
its suspension defined by $CF[1]^k = CF^{k+1}$.
We denote by $\deg(x)=|x|$ the degree of $x \in C$ before the shift and
$\deg'(x)=|x|'$ that after the degree shifting, i.e., $|x|' = |x| - 1$.
Define the {\it bar complex} $B(CF[1])$ by
$$
B_k(CF[1]) = (CF[1])^{k\otimes}, \quad B(CF[1]) =
\bigoplus_{k=0}^\infty B_k(CF[1]).
$$
Here $B_0(CF[1]) = R$ by definition.
The tensor product is taken over $\Lambda(L)$.
We provide the degree of
elements of $B(CF[1])$ by the rule
\begin{equation}\label{eq:degonBC[1]}
|x_1 \otimes \cdots \otimes x_k|': = \sum_{i=1}^k |x_i|' = \sum_{i =1}^k|x_i|
-k
\end{equation}
where $|\cdot|'$ is the shifted degree. The ring $B(CF[1])$
has the structure of {\it graded coalgebra}.
\begin{defn} A filtered $A_\infty$ algebra
over $\Lambda(L)$ is a
sequence of $\Lambda(L)$ module homomorphisms
$$
\mathfrak m_k: B_k(CF[1]) \to CF[1], \quad k =0, 1, 2, \cdots,
$$
of degree +1 such that the coderivation
$d = \sum_{k=0}^\infty \widehat{\mathfrak m}_k$
satisfies $d d= 0$, which
is called the \emph{$A_\infty$-relation}.
Here we denote by $\widehat{\mathfrak
m}_k: B(CF[1]) \to B(CF[1])$ the unique extension of $\mathfrak m_k$
as a coderivation on $B(CF[1])$. A \emph{filtered $A_\infty$ algebra}
is an $A_\infty$ algebra with a filtration for which $\mathfrak m_k$ are
continuous with respect to the induce non-Archimedean topology.
\end{defn}
If we have $\mathfrak m_1
\mathfrak m_1 = 0$, it defines a complex $(CF,\mathfrak m_1)$. We
define the $\mathfrak m_1$-cohomology by
\begin{equation}\label{eq:m1cohom}
H(CF,\mathfrak m_1) = \mbox{Ker}\,\mathfrak m_1/\mbox{Im}\,\mathfrak m_1.
\end{equation}

The first two terms of the $A_\infty$ relation for a
$A_\infty$ algebra are given as
\begin{eqnarray}
\mathfrak m_1(\mathfrak m_0(1)) & = & 0 \label{eq:m1m0=0} \\
\mathfrak m_1\mathfrak m_1 (x) + (-1)^{|x|'}\mathfrak m_2(x, \mathfrak m_0(1)) +
\mathfrak m_2(\mathfrak m_0(1), x) & = & 0. \label{eq:m0m1}
\end{eqnarray}
In particular, for the case $\frak m_0(1)$ is nonzero, $\mathfrak
m_1$ will not necessarily satisfy the boundary property, i.e., $\mathfrak m_1\mathfrak m_1
\neq 0$ in general.

\par
We now describe the $A_\infty$ operators $\mathfrak m_k$
in the context of $A_\infty$ algebra of Lagrangian submanifolds.
For a given compatible almost complex structure $J$, consider the moduli
space of stable maps of genus zero
$$
\CM_{k+1}(\beta;L) =\{ (w, (z_0,z_1, \cdots,z_k)) \mid
\overline \partial w = 0, \, z_i \in \partial D^2, \, [w] = \beta
\, \mbox{in }\, \pi_2(M,L) \}/\sim
$$
where $\sim$ is the conformal reparameterization of the disc $D^2$.
We require that $z_0, \cdots, z_k$ respects counter clockwise cyclic order of
$S^1$.
(We wrote this moduli space $\CM^{\text{\rm main}}_{k+1}(\beta;L)$
in Section 2.1 \cite{fooo:book}. The symbol `main' indicates the compatibility
of $z_0, \cdots, z_k$, with counter clockwise cyclic order. We omit this
symbol in this paper since we always assume it.)
\par
$\CM_{k+1}(\beta;L)$ has a Kuranishi structure and its
dimension is given by
\begin{equation}\label{eq:dim}
n+ \mu(\beta) - 3 + (k+1) = n+\mu(\beta) + k-2.
\end{equation}
Now let
$
[P_1,f_1], \cdots,[P_k,f_k] \in C_*(L;\Q)
$
be $k$ smooth singular simplices of $L$.
(Here we denote by $C(L;\Q)$
a \emph{suitably chosen countably generated} cochain
complex of smooth singular chains of $L$.)
We put the cohomological grading
$\mbox{deg} P_i = n - \dim P_i$ and consider the fiber product
$$
ev_0: \CM_{k+1}(\beta;L) \times_{(ev_1, \cdots, ev_k)}(P_1 \times
\cdots \times P_k) \to L.
$$
A simple calculation shows that the expected dimension of this chain is given by
$
n + \mu(\beta) - 2 + \sum_{j=1}^k(\dim P_j + 1- n)
$
or equivalently we have the degree
$$
\mbox{deg}\left[\CM_{k+1}(\beta;L) \times_{(ev_1, \cdots, ev_k)}(P_1
\times \cdots\times P_k),
ev_0\right] = \sum_{j=1}^n(\mbox{deg} P_j -1) + 2- \mu(\beta).
$$
For each given $\beta \in \pi_2(M,L)$ and $k = 0, \cdots$, we
define $\frak m_{1,0}(P) = \pm \partial P$ and
\be\label{eq:mkbeta}
\aligned
\mathfrak m_{k,\beta}(P_1, \cdots, P_k)
&= \left[\CM_{k+1}(\beta;L) \times_{(ev_1, \cdots, ev_k)}(P_1 \times \cdots \times P_k),
ev_0\right] \\
& \in C(L;\Q)
\endaligned
\ee
(More precisely we regard the right hand side of (\ref{eq:mkbeta})
as a smooth singular chain by taking appropriate multi-valued
perturbation (multisection) and choosing a simplicial decomposition of its zero set.)
\par
We put
$$
CF(L) = C(L;\Q) \,\,\widehat{\otimes}_{\Q}\,\, \Lambda(L).
$$
We define
$
\mathfrak m_k: B_kCF(L)[1] \to B_kCF[1]
$
by
$$
\mathfrak m_k= \sum_{\beta \in \pi_2(M,L)} \mathfrak m_{k,\beta} \otimes [\beta].
$$
\par
Then it follows that the map
$
\mathfrak m_k: B_kCF(L)[1] \to CF(L)[1]
$
is well-defined, has degree 1 and continuous with respect to non-Archimedean topology.
We extend $\mathfrak m_k$ as a coderivation
$\widehat{\mathfrak m}_k: BCF[1] \to BCF[1]$
where $BCF(L)[1]$ is the completion of the direct sum $\oplus_{k=0}^\infty
B_kCF(L)[1]$ where $B_kCF(L)[1]$ itself is the completion of $CF(L)[1]^{\otimes k}$.
$BCF(L)[1]$ has a natural filtration defined similarly as Definition \ref{efilt}.
Finally we take the sum
$$
\widehat d = \sum_{k=0}^\infty \widehat{\frak m}_k: BCF(L)[1] \to BCF(L)[1].
$$
We then have the following coboundary property:
\begin{thm}\label{algebra} Let $L$ be an arbitrary compact relatively
spin Lagrangian submanifold of an arbitrary tame symplectic manifold
$(M,\omega)$. The coderivation $\widehat d$ is a continuous map that
satisfies the $A_\infty$ relation $\widehat d \widehat d = 0$, and so
$(CF(L),\mathfrak m)$ is a filtered $A_\infty$ algebra over $\Lambda(L)$.
\end{thm}
We put
$$
C(L;\Lambda_{0,nov}) = CF(L) \,\widehat{\otimes}_{\Lambda(L)} \,\, \Lambda_{0,nov}
$$
on which a filtered $A_{\infty}$ structure on $C(L;\Lambda_{0,nov})$
(over the ring $\Lambda_{0,nov}$) is induced. This is the filtered
$A_{\infty}$ structure given in Theorem A \cite{fooo:book}.
\par
In the presence of $\mathfrak m_0$,
$\widehat{\mathfrak m}_1 \widehat{\mathfrak m}_1 = 0$ no longer holds
in general. This leads to consider deforming Floer's original definition by a bounding
cochain of the obstruction cycle arising from bubbling-off discs.
One can always deform the given (filtered) $A_\infty$ algebra $(CF(L),\mathfrak m)$ by
an element $b \in CF(L)[1]^0$ by re-defining the $A_\infty$ operators as
$$
\mathfrak m_k^b(x_1,\cdots, x_k) = \mathfrak m(e^b,x_1, e^b,x_2, e^b,x_3, \cdots,
x_k,e^b)
$$
and taking the sum $\widehat d^b = \sum_{k=0}^\infty \widehat{\mathfrak m}_k^b$.
This defines a new filtered $A_\infty$ algebra in general. Here we
simplify notations by writing
$$
e^b = 1 + b + b\otimes b + \cdots + b \otimes \cdots \otimes b +\cdots.
$$
Note that each summand in this infinite sum has degree 0 in $CF(L)[1]$ and
converges in the non-Archimedean topology if $b$ has positive
valuation, i.e., $\frak v(b) > 0$.
(See Section \ref{subsec:gamma-equiv} for the definition of $\frak v$.)
\begin{prop} For the $A_\infty$ algebra $(CF(L),\mathfrak m_k^b)$,
$\mathfrak m_0^b = 0$ if and only if $b$ satisfies
\begin{equation}\label{eq:MC}
\sum_{k=0}^\infty\mathfrak m_k(b,\cdots, b) = 0.
\end{equation}
This equation is a version of \emph{Maurer-Cartan equation}
for the filtered $A_\infty$ algebra.
\end{prop}

\begin{defn}\label{boundchain}
Let $(CF(L),\mathfrak m)$ be a filtered $A_\infty$ algebra in general and
$BCF(L)[1]$ be its bar complex. An element $b \in CF(L)[1]^0 = CF(L)^1$ is called
a \emph{bounding cochain} if it satisfies the equation (\ref{eq:MC})
and $v(b) > 0$. We denote by $\widetilde \CM(L;\Lambda(L))$ the set of bounding
cochains.
\end{defn}

In general a given $A_\infty$ algebra may or may not have a solution
to (\ref{eq:MC}). In our case we define:

\begin{defn}\label{unobstructed}
A filtered $A_\infty$ algebra $(CF(L),\mathfrak m)$ is called \emph{unobstructed over}
$\Lambda(L)$ if the equation
(\ref{eq:MC}) has a solution $b \in CF(L)[1]^0 = CF(L)^1$ with $v(b) > 0$.
\end{defn}
One can define the notion of homotopy equivalence between two
bounding cochains and et al as described in Chapter 4 \cite{fooo:book}.
We denote by $ \CM(L;\Lambda(L))$ the set of equivalence classes of
bounding cochains of $L$.
\begin{rem}
In Definition \ref{boundchain} above we consider bounding cochain
contained in $CF(L) \subset C(L;\Lambda_0)$ only. This is the reason
why we write $\widetilde \CM(L;\Lambda(L))$ in place of $\widetilde
\CM(L)$. (The latter is used in \cite{fooo:book}.)
\end{rem}
\subsection{$A_{\infty}$ bimodule}
\label{subsec:morphisms}
Once the $A_\infty$
algebra is attached to each Lagrangian submanifold $L$, we then
construct a structure of filtered \emph{$A_\infty$ bimodule} on the
module $CF(\CL_1, \CL_0) = CF((L_1,\gamma_1),(L_0,\gamma_0))$, which was introduced in
Section \ref{chaincomplex} as follows. This filtered $A_\infty$
bimodule structure is by definition is a family of operators
$$
\aligned
\mathfrak n_{k_1,k_0}: B_{k_1}(CF(L_1)[1])\,\,
\widehat{\otimes}_{\Lambda(L_1)} \,\, CF((L_1,\gamma_1),
(L_0,\gamma_0))
\,\,&\widehat{\otimes}_{\Lambda(L_0)} \,\, B_{k_0}(CF(L')[1]) \\
&\to CF((L_1,\gamma_1),
(L_0,\gamma_0))
\endaligned$$
for $k_0,k_1\ge 0$.
Here the left hand side is defined as follows:
It is easy to see that there are embeddings
$\Lambda(L_0) \to \Lambda(\CL_0,\CL_1)$,
$\Lambda(L_1) \to \Lambda(\CL_0,\CL_1)$.
Therefore a $ \Lambda(\CL_0,\CL_1)$ module
$CF((L_1,\gamma_1),
(L_0,\gamma_0))$ can be regarded both as
$\Lambda(L_0)$ module and $\Lambda(L_1)$ module.
Hence we can take tensor product in the left hand side.
($\widehat{\otimes}_{\Lambda(L_i)}$ is the
completion of this algebraic tensor product.)
The left hand side then becomes a $ \Lambda(\CL_0,\CL_1)$ module,
since the rings involved are all commutative.
\par
We briefly describe the definition of $\mathfrak n_{k_1,k_0}$.
A typical element of the tensor product
$$
B_{k_1}(CF(L_1)[1]) \widehat{\otimes}_{\Lambda(L_1)} \,\, CF((L_1,\gamma_1),
(L_0,\gamma_0))
\,\,\widehat{\otimes}_{\Lambda(L_0)} \,\, B_{k_0}(CF(L_0)[1])
$$
has the form
$$
P_{1,1} \otimes \cdots, \otimes P_{1,k_1} \otimes [p,w] \otimes
P_{0,1} \otimes \cdots \otimes P_{0,k_0}
$$
with $p \in L_0 \cap L_1$ being an admissible intersection point. Then the image $\mathfrak n_{k_0,k_1}$ thereof is given by
$$
\sum_{q, B}T^{\omega(B)}e^{\mu(B)/2}\# \left(\CM(p,q;B;P_{1,1},\cdots,P_{1,k_1};
P_{0,0},\cdots,P_{0,k_0})\right) [q,B\#w].
$$
Here $B$ denotes homotopy class of Floer trajectories connecting $p$ and $q$,
the summation is taken over all $[q,B]$ with
$$
\dim \CM(p,q;B;P_{1,1},\cdots,P_{1,k_1};
P_{0,1},\cdots,P_{0,k_0}) = 0,
$$
and $\# \left(\CM(p,q;B;P_{1,1},\cdots,P_{1,k_1};P_{0,1},\cdots,P_{0,k_0})\right)$ is
the `number' of elements in
the `zero' dimensional moduli space $\CM(p,q;B;P_{1,1},\cdots,P_{1,k_1};
P_{0,1},\cdots,P_{0,k_0})$. Here the moduli space $\CM(p,q;B;P_{1,1},\cdots,P_{1,k_1};
P_{0,1},\cdots,P_{0,k_0})$ is the Floer moduli space
$
\CM(p,q;B)
$
cut-down by intersecting with the given chains $P_{1,i} \subset L_1$
and $P_{0,j} \subset L_0$. (See Section 3.7 \cite{fooo:book}.)
An orientation on this moduli space can be given in \cite{fooo:book,fooo:anchor}.
\begin{thm}\label{thm:bimodule}
Let $(\CL_0,\CL_1)$ be a pair of anchored
Lagrangian submanifolds.
Then the family $\{\mathfrak n_{k_1,k_0}\}$ defines
a left $(CF(L_1),\mathfrak m)$ and right $(CF(L_0),\mathfrak m)$ filtered
$A_\infty$-bimodule structure on $CF(\CL_1,\CL_0)$.
\end{thm}
See Section 3.7 \cite{fooo:book} and \cite{fooo:anchor} for the definition of filtered $A_{\infty}$ bimodules.
(In \cite{fooo:book} the case of universal Novikov ring as a coefficient is considered.
It is easy to modify to our case of $\Lambda(L_0,L_1)$ coefficient.)
\par
In the case where both $L_0, \, L_1$ are
unobstructed, we can carry out this
deformation of $\mathfrak n$ using bounding cochains $b_0$
and $b_1$ of $CF(L_0)$ and $CF(L_1)$ respectively, in
a way similar to $\frak m^b$. Namely we define
$\delta_{b_1,b_0}: CF(\CL_1,\CL_0) \to CF(\CL_1,\CL_0)$ by
$$
\delta_{b_1,b_0}(x) =
\sum_{k_1,k_0} \frak n_{k_1,k_0} (b_1^{\otimes k_1} \otimes
x \otimes b_0^{\otimes k_0}) = \mathfrak{\widehat n}(e^{b_1},x,e^{b_0}).
$$
We can generalize the story to the case where $L_0$ has clean intersection
with $L_1$, especially to the case $L_0=L_1$.
In the case $L_0=L_1$ we have
$\frak n_{k_1,k_0} = \frak m_{k_0+k_1+1}$. So
in this case, we have
$
\delta_{b_1,b_0}(x) = \frak m(e^{b_1},x,e^{b_0}).
$
\par
We define Floer cohomology of the pair $\CL_0 = (L_0,\gamma_0,\lambda_0)$, $\CL_1=(L_1,\gamma_1,\lambda_1)$
by
$$
HF((\CL_1,b_1),(\CL_0,b_0)) = \operatorname{Ker}\delta_{b_1,b_0}/\operatorname{Im}
\delta_{b_1,b_0}.
$$
This is a module over $\Lambda(\CL_0,\CL_1)$.
\begin{thm}
$HF((\CL_1,b_1),(\CL_0,b_0)) \otimes_{\Lambda(\CL_0,\CL_1)} \Lambda_{nov}$
is invariant under the Hamiltonian isotopies of $\CL_0$ and $\CL_1$ and under the
gauge equivalence of bounding cochains $b_0, \, b_1$.
\end{thm}
We refer to section 4.3 \cite{fooo:book} for the definition of gauge equivalence and
to Theorem 4.1.5 \cite{fooo:book} for the proof of this theorem.
\section{Definitions of Seidel's maps $\frak b$, $\frak  c$ and $\frak h$}
\label{sec:seidel's}

In this section, we recall the definition of Seidel's cochain
maps $\frak b,\, \, \frak c$ and the homotopy $\frak h$ and give the definition of the analogs thereof
in our general setting. They are the maps
\beastar
\frak b: CF(L,L_1) \otimes CF(\tau_L(L_0),L) & \to &
CF(\tau_L(L_0),L_1)\\
\frak c:CF(\tau_L(L_0),L_1) & \to & CF(L_0,L_1)
\eeastar
and the homotopy $\frak h: CF(L,L_1) \otimes CF(\tau_L(L_0),L)
\to CF(L_0,L_1)$ between the composition $\frak c\circ \frak b$ and
the zero map.

We will generalize these maps to our non-exact case and describe all
the necessary properties of the maps in the next section.
We consider a quadruple of anchored Lagrangian submanifolds
\be\label{eq:tau-anchor}
\CL = (L,\gamma), \, \CL_0= (L_0,\gamma_0), \, \CL_1=(L_1, \gamma_1), \,
\tau_*\CL = (\tau_L(L_0), \tau_L(\gamma_0)).
\ee
For the simplicity of notations in this section, we will just denote
the action functional associated to the pair $(L,\gamma)$ and $(L',\gamma')$
just by $\CA_{\CL\CL'}$.

\subsection{The map $\frak b$}
\label{subsec:map-b}

Let $\Sigma$ be a compact surface with boundary marked points
$\vec\zeta =\{\zeta_0,\zeta_1,\zeta_2\}$. We denote $\dot \Sigma =
\Sigma \setminus \vec{\zeta}$ and $\del \dot\Sigma = \cup_{i=0}^2
\del_i \dot \Sigma$. We consider the three anchored Lagrangian
submanifolds $(L_0,\gamma_0)$, $(L,\gamma)$ and $(L_1,\gamma_1)$.
Take the trivial Hamiltonian fibration $\pi: E = \dot \Sigma \times
M \to \dot \Sigma$ with the two form $\Omega$ equal to the two form
pulled back from $\omega$ in $M$. Equip this with the Lagrangian
boundary condition
$$
Q = (\del_0 \dot \Sigma \times \tau_L(L_0)) \cup
(\del_1 \dot \Sigma \times L) \cup (\del_3 \dot \Sigma \times L_1).
$$
$Q$ is an exact Lagrangian boundary with $\kappa_Q = 0$.
We note that $E$ has the trivial connection given by $K \equiv 0$
and hence has zero curvature.

\begin{lem} Suppose that $\CL = (L,\gamma), \, \CL_0= (L_0,\gamma_0), \, \CL_1=(L_1, \gamma_1)$
are given anchors of the type \eqref{eq:tau-anchor}
and $\tau_*\CL = (\tau_L(L_0), \tau_L(\gamma_0))$.
Let $s: \dot \Sigma \to \dot \Sigma \times M$ be a
section with the given exact Lagrangian boundary condition $Q$ as above.
Let $[p_0,w_0] \in \Crit \CA_{\tau_L(\CL_0)\CL_1}$,
$[p,w] \in \Crit \CA_{\CL\CL_1}$ and
$[p_1,w_1] \in \Crit \CA_{\tau_L(\CL_0)\CL}$.
Suppose the homotopy class $[s,\del s]$ is admissible and  satisfies
\be\label{eq:[w0]}
[w_0] = [s,\del s]*[w]*[w_1] \quad \mbox{in }\,
\pi_2(\tau_L(L_0),L_1;\tau_L(\overline\gamma_0)*\gamma_1).
\ee
Then we have
\be\label{eq:action}
\int s^*\Omega = \CA_{\tau_L(\CL_0)\CL_1}([p_0,w_0])
- \CA_{\CL\CL_1}([p,w]) - \CA_{\tau_L(\CL_0)\CL}([p_1,w_1]).
\ee
\end{lem}
\begin{proof} This immediately follows from Proposition \eqref{prop:coherence}.
\end{proof}

\begin{defn} Let $L, \, L' \subset M$ be a pair of Lagrangian submanifolds.
We say that $J$ lies in $\CJ^{reg}(M;L,L')$ if all Floer trajectories
associated to $J$ for the pair $(L,L')$ are Fredholm-regular.
\end{defn}

At this point, we fix
\beastar
J^{(1)} & \in & \CJ^{reg}(M;\tau_L(L_0),L) \\
J^{(2)} & \in & \CJ^{reg}(M;L,L_1)\\
J^{(3)} & \in & \CJ^{reg}(M;\tau_L(L_0),L_1).
\eeastar
By choosing a
horizontal $J \in \CJ(E,\pi,j,J^{(1)},J^{(2)},J^{(3)})$, i.e., $J$ satisfying
$J_x(TE^h_x) = TE^h_x$ for all $x \in E\setminus E^{crit}$ with the
above trivial fibration $E = \dot \Sigma \times M$, this gives rise
to the standard  `pants product map'
\be\label{eq:pants}
CF(\CL,\CL_1)
\otimes CF(\tau_L(\CL_0),\CL)
\to CF(\tau_L(\CL_0),\CL_1)
\ee
in the cochain level. Seidel's map $\frak b$ in the cochain level is nothing but this
pants product map.

By choosing bounding cochains $b_0, \, b_1$ of $L_0, \, L_1$ respectively
and considering $(\tau_L)_*b_0$ on $\tau_L(L_0)$, we define the deformed $\frak m_2$
$$
\frak m_2^{\vec b}: CF(\CL,\CL_1)
\otimes CF(\tau_L(\CL_0),\CL) \to CF(\tau_L(\CL_0),\CL_1)
$$
where $\vec b = (0,(\tau_L)_*b_0,b_1)$ and $\frak m_i^{\vec b}$ is defined as
in (8.15) \cite{fooo:anchor}.

According to \cite{fooo:book,fukaya:mirror2,fooo:anchor}, this induces a cochain map
up to the higher homotopy map $\mathfrak m_3^{\vec b}$ and induces a homomorphism
in cohomology. For readers' convenience, we summarize the construction
of this product map $\mathfrak m_k^{\vec b}$ in Appendix.

\subsection{The map $\frak c$}
\label{subsec:map-c}

Let $(E^L,\pi^L)$ be the standard fibration over a disc
$D(\frac{1}{2})$, whose monodromy around $\del D(\frac{1}{2})$ is $\tau_L$,
as defined in section \ref{sec:dehn} by choosing $r$ small.
Let $0 < r < \frac{1}{2}$ be given, and choose a function $g$ with $g(t) = t$
for small $t$, $g(t) \equiv r$ for $t \geq r$ and $g'(t) \geq 0$ everywhere.
We consider the map $p: D(\frac{1}{2}) \to D(\frac{1}{2})$ defined by
$p(z) = g(|z|)\frac{z}{|z|}$ and consider the pull-back fibration
$$
(E^p, \pi^p) = p^*(E^L,\pi^L).
$$
This is flat on the annulus $D(\frac{1}{2})
\setminus \Int(D(r))$.

Now take the surface
$$
\Sigma^f = \R \times [-1,1] \setminus \Int D(1/2)
\subset \R^2,
$$
with coordinates $(s,t)$ and divide it into two parts
$\Sigma^{f,\pm} = \Sigma^f \cap \{t \in \R_\pm\}$ so that
$$
\Sigma^{f,+} \cap \Sigma^{f,-} = \left((-\infty, - 1/2] \cup
[1/2,\infty)\right) \times \{0\}.
$$
Consider trivial fibrations
$$
\pi^{f,\pm}: E^{f,\pm} = \Sigma^{f,\pm} \times M \to \Sigma^{f,\pm}
$$
over the two parts, equip them with two forms $\Omega^{f,\pm}$
the pull-back of $\omega$.
We define a fibration $(E^f,\pi^f)$ over $\Sigma^f$ by identifying the
fibers $E^{f,+}_{(s,0)} \to E^{f,-}_{(s,0)}$ via $id_M$ for $s \geq \frac{1}{2}$
and via $\tau_L$ for $s \leq - \frac{1}{2}$. Because $\tau_L$ is symplectic,
this defines a flat Hamiltonian fibration.

Using the fact that two fibrations $E^p$ and $E^f$ are flat close to the
curve $|z| = \frac{1}{2}$, we now paste them along the curve. Denote the
resulting fibration over
$$
\Sigma = \Sigma^p \cup \Sigma^f = \R \times [-1,1]
$$
by $(E^c,\pi)$. Equip $(E^c,\pi)$ with the Lagrangian boundary condition
$$
Q^c = \begin{cases}
\R \times \{1\} \times L_1 & \subset E^{f,+} \\
\R \times \{-1\} \times \tau_L(L_0) & \subset E^{f,-}.
\end{cases}
$$
This defines an exact Lagrangian boundary with the one form
$\kappa_{Q^c} = 0$. As explained in section 3.3 \cite{seidel:triangle}, $(E^c,\pi)$ is modeled
on $(\tau_L(L_0),L_1)$ over the positive end of $\Sigma$ and over the negative
end, it is modeled on $(L_0,L_1)$ due to the momodromy effect around the
critical value $(0,0) \in \Sigma$.

Let $j$ be some complex structure on $\Sigma$, standard over the ends.
Take some $J^{(3)} \in \CJ^{reg}(M,\tau_L(L_0),L_1)$ as in the previous
subsection and choose an additional $J^{(5)} \in \CJ^{reg}(M;L_0,L_1)$.
Using a regular $J^{(6)} \in \CJ(E,\pi,Q,j,J^{(3)},J^{(5)})$,
we define a map
$$
\frak c = C\Phi_0^{rel}(E^c,\pi,Q^c,J^{(6)}): CF(\tau_L(\CL_0),\CL_1)
\to CF(\CL_1,\CL_0)
$$
as defined in section 3.3 \cite{seidel:triangle}.
\par
Since we use the deformed Floer complex by bounding cochains,
we need to construct the deformed version of the map $\frak c$.
The construction will resemble that of $A_\infty$-bimodule structure
on $CF(\CL_1, \CL_0)$ carried out in subsection \ref{subsec:morphisms}.

We first define a family of operators
$$
\aligned
\mathfrak c_{k_1,k_0}: B_{k_1}(CF(L_1)[1])\,\,
\widehat{\otimes}_{\Lambda(L_1)} \,\, CF(\CL_1, \tau_L(\CL_0))
\,\,&\widehat{\otimes}_{\Lambda(L_0)} \,\, B_{k_0}(CF(L_0)[1]) \\
&\to CF(\CL_1,\CL_0)
\endaligned$$
for $k_0,k_1\ge 0$.
A typical element of the tensor product
$$
B_{k_1}(CF(L_1)[1]) \widehat{\otimes}_{\Lambda(L_1)} \,\, CF(\CL_1,
\tau_L(\CL_0))
\,\,\widehat{\otimes}_{\Lambda(L_0)} \,\, B_{k_0}(CF(L_0)[1])
$$
has the form
$$
P_{1,1} \otimes \cdots \otimes P_{1,k_1} \otimes [p,w] \otimes
P_{0,1} \otimes \cdots \otimes P_{0,k_0}
$$
with $p \in \tau_L(\CL_0) \cap \CL_1$ being an admissible intersection point.
Then the image $\mathfrak c_{k_0,k_1}$ thereof is given by
$$
\sum_{q, B}T^{\omega(B)}e^{\mu(B)/2}\# \left(\CM^c(p,q;B;P_{1,1},\cdots,P_{1,k_1};
P_{0,0},\cdots,P_{0,k_0})\right) [q,B\#w].
$$
Here $B$ denotes section class of $J$-holomorphic sections connecting $p$ and $q$,
the summation is taken over all $[q,B]$ with
$$
\dim \CM^c(p,q;B;P_{1,1},\cdots,P_{1,k_1};
P_{0,1},\cdots,P_{0,k_0}) = 0,
$$
and $\# \left(\CM^c(p,q;B;P_{1,1},\cdots,P_{1,k_1};P_{0,1},\cdots,P_{0,k_0})\right)$ is
the `number' of elements in
the `zero' dimensional moduli space $\CM^c(p,q;B;P_{1,1},\cdots,P_{1,k_1};
P_{0,1},\cdots,P_{0,k_0})$. Here the moduli space $\CM^c(p,q;B;P_{1,1},\cdots,P_{1,k_1};
P_{0,1},\cdots,P_{0,k_0})$ is the moduli space
$
\CM^c(p,q;B)
$
of $J$-holomorphic sections of $\pi: E^c \to \Sigma$ cut-down by intersecting with the given chains $P_{1,i} \subset L_1$
and $P_{0,j} \subset L_0$.
\begin{thm}\label{thm:deformedc}
Let $(\CL_0,\CL_1) = ((L_0,\gamma_0), (L_1,\gamma_1))$ be a pair of anchored
Lagrangian submanifolds.
Then the family $\{\mathfrak c_{k_1,k_0}\}$ defines
a left $(CF(L_1),\mathfrak m)$ and right $(CF(L_0),\mathfrak m)$ filtered
$A_\infty$-bimodule homomorphism from $CF(\CL_1, \tau_L(\CL_0))$ to $CF(\CL_1,\CL_0)$.
\end{thm}
The proof of Theorem \ref{thm:deformedc} is similar to that of Theorem \ref{thm:bimodule}
given in the proof of Theorem 3.7.21 \cite{fooo:book}.
\par
When both $L_0, \, L_1$ are unobstructed, we can carry out this
deformation of $\mathfrak c$ using bounding cochains $b_0$
and $b_1$ of $CF(L_0)$ and $CF(L_1)$ respectively, in
a way similar to $\frak n^{b_0,b_1}$. Namely we define
$\frak c^{b_1,b_0}: CF(\CL_1, \tau_L(\CL_0)) \to CF(\CL_1, \CL_0))$ by
$$
\frak c^{b_1,b_0}(x) =
\sum_{k_1,k_0} \frak c_{k_1,k_0} (b_1^{\otimes k_1} \otimes
x \otimes b_0^{\otimes k_0}) = \mathfrak{\widehat c}(e^{b_1},x,e^{b_0}).
$$
The following proposition is all we need to add to our context for the
construction of Seidel's map $\frak c^{b_1,b_0}$
in \cite{seidel:triangle}.

\begin{prop}\label{prop:deformedc} Let $b_0,\, b_1$ be bounding cochains of $L_0, \, L_1$
respectively. Then $\frak c$ defines a chain map
$$
\frak c^{b_1,b_0} = C\Phi_0^{rel}(E,\pi,Q,J^{(6)}): (CF(\tau_L(\CL_0),\CL_1), \frak m_1^{(\tau_L)_*(b_0)})
\to (CF(\CL_1,\CL_0), \frak m_1^{b_1})
$$
and hence induces a homomorphism
$$
\frak c^{b_1,b_0}:HF((\tau_L(\CL_0), (\tau_L)_*b_0,(\CL_1, b_1))
\to HF((\CL_1,b_1),(\CL_0,b_0)).
$$
\end{prop}
\begin{proof} The proof proceeds in the same way as that of Theorem \ref{thm:bimodule}.
The only difference from the latter is that the moduli space
$\CM^c(p,q;B;P_{1,1},\cdots,P_{1,k_1};P_{0,1},\cdots,P_{0,k_0})$ does not have
$\R$-action anymore and so we consider the moduli space of sections without
considering the quotient.
\end{proof}

\subsection{A simple invariant and its vanishing theorem}
\label{subsec:vanishing}

Let $\Phi_1(E,\pi,Q)$ be the invariant
$$
\Phi_1(E,\pi,Q) = (ev_\zeta)_*[\CM_J] \in H_*(Q_\zeta;\Z)
$$
for the Calabi-Yau Lefschetz fibration.
The following proposition replaces a similar proposition,
Proposition 2.13 \cite{seidel:triangle},  for the present Calabi-Yau Lefschetz
fibration setting.

For each given section class $A \in \pi_2^{sec}(E,Q)$, we define
the moduli space $\CM_J(A)$ of $J$-holomorphic section $s:D \to E$ with
$[s,\del s] = A$ and define an invariant
$$
\Phi_1(E,\pi,Q) = \sum_{A \in \pi_2^{sec}(E,Q)} (ev_\zeta)_*(\CM_J(A))
T^{\Omega(A)} \in H_*(Q_\zeta;\Lambda_{0,nov}).
$$
We start with the following slight generalization of Proposition 2.2 in the
general context of Lagrangian spheres in general symplectic
manifolds.

\begin{prop}
Let $(L,[f])$ be a framed Lagrangian sphere in $M$. There is a
one-parameter family of Lefschetz Hamiltonian fibrations
$(E^L_r,\pi^L_r) \to D(r)$ together with an isomorphism $\phi^L_r:
E_r^L \to M$ of symplectic manifolds, such that
\begin{enumerate}
\item Consider the re-scaling map $\lambda_r: D(r) \to D(1)$
defined by $z \mapsto \frac{z}{r}$. Then
$$
(\lambda_r)^*E^L_1 = E^L _r.
$$
\item If $\rho^L_r$ is the symplectic monodromy around $\del
\overline D(r)$, then $\phi^L_r \circ \rho^L \circ
(\phi^L_r)^{-1}$ is a Dehn twist along $(L,[f])$.
\end{enumerate}
We denote any of these maps by $\tau_L$ as before.
\end{prop}

\begin{proof}
The proof follows from Seidel's proof of Proposition 1.11
\cite{seidel:triangle} stripping all the things related to the
exactness requirement in the proof. In fact the proof is easier
because we do not have to concern the exactness requirement in the
construction.
\end{proof}

We also recall that each fiber $E_z^L, \, z \neq 0$ of the fibration
$E^L$ contains a distinguished Lagrangian sphere $\Sigma_z^L$. We
call this fibration a \emph{standard Calabi-Yau Lefschetz
fibration}. The following is a crucial proposition needed in
Seidel's construction of long exact sequence in
\cite{seidel:triangle}, whose proof goes through in our current
context.

\begin{prop}\label{prop:vanishing} Let $L \subset M$ be any Calabi-Yau Lagrangian brane
and $\pi^L: E^L \to D(r)$ be the associated standard Calabi-Yau
Lefschetz fibration. Then we have
$$
\Phi_1(E^L,\pi^L,Q^L) = 0.
$$
\end{prop}
\begin{proof} The proof of this proposition will verbatim follow
that of Proposition 2.13 \cite{seidel:triangle} and so omitted.
\end{proof}

\subsection{The gluing $\frak b \#_\rho \frak c$, composition
$\frak c\circ \frak b$ and the homotopy $\frak h$}
\label{subsec:homotopy}

We denote by $(E^b,\pi^b)$, $\Sigma^b$ and $Q^b$ the fibration, surface
and the boundary condition associated to the map $b$ and $(E^c,\pi^c)$,
$\Sigma^c$ and $Q^c$ those associated to the map $c$ constructed in the
previous subsections. We glue the positive end of $E^c$ and the
negative end of $E^b$ to obtain $(E^{bc}_\rho,\pi^{bc}_\rho)$, $\Sigma^{bc}_\rho$ and
$Q^{bc}_\rho$ for a sufficiently large gluing parameter $\rho$ in the
glued ends.
This fibration provides a cochain map
$$
\frak b \#_\rho \frak c: CF(\CL,\CL_1) \otimes CF(\tau_L(\CL_0),\CL) \to CF(\CL_1,\CL_0).
$$
Here and henceforth we denote
$$
\frak b = \frak m_2^{\vec b}, \quad \frak c = \frak c^{b_1,b_0}
$$
for notational simplicity.

\begin{lem}
Let $(j,J)$ be such that $j$ is a complex structure on $\Sigma^{bc}_\rho$
which is standard on its ends, and $J \in \CJ(E^{bc}_\rho,\pi^{bc}_\rho,
Q^{bc}_\rho,j,J^{(0)}, J^{(2)}, J^{(5)})$.
The cochain map $b \#_\rho c$ coincides with
$c\circ b$ for  any $\rho \geq \rho_0$ with $\rho_0$ sufficiently large.
\end{lem}
\begin{proof}
We note that the composition $\frak c \circ \frak b$ is defined by
counting the elements of the fiber product
$$
\CM(E^b,\CQ^b;\vec p^b;B^b) {}_{ev_+}\#_{ev_-} \CM(E^c,\CQ^c;\vec p^c;B^c)
$$
while the map $\frak b \#_\rho \frak c$ by counting those in the moduli space
$$
\CM(E^{bc}_\rho,\CQ^{bc}_\rho;\vec p^{bc}_\rho;B^{bc}_\rho)
$$
where $B^{bc}_\rho = B^b \#_\rho B^c$ is the obvious glued
homotopy class. By a gluing theorem in the
Floer complex (see e.g., \cite{fooo:book}), the two moduli spaces are diffeomorphic to
each other and hence the proof.
\end{proof}

Finally we can construct a homotopy from $\frak b\#_\rho \frak c$ to the zero
map verbatim following Seidel's argument from \cite{seidel:triangle}. (See Figure
9 and Figure 12 of \cite{seidel:triangle} in particular.) We omit
details of the construction.

%

\section{Thick-thin decompositions}
\label{sec:thickthin}

In this section, we will study the thick-thin decomposition of
contributions in the cochain maps $\frak b,\, \frak c$ and the homotopy $\frak h$.

For this purpose, we first note the Dehn twist $\tau_L^{-1}: M \to M$
acts by $(L_i,\gamma_i) \to (\tau_L^{-1}(L_i), \tau_L^{-1}(\gamma_i))$ and
induces a one-one correspondence
$$
\tau_L(L_0) \cap L \to L_0 \cap L; \quad x \mapsto \tau_L^{-1}(x)
$$
Since $\tau_L|_L = id|_L$. This lifts to a diffeomorphism
$$
\widetilde\Omega(\tau_L(\CL_0), \CL)\to
\widetilde\Omega(\CL_0, \CL);
\quad [p,w] \mapsto [\tau_L^{-1}(p),\tau_L^{-1}(w)].
$$
This latter diffeomorphism induces a filtration preserving isomorphism
$$
(\tau_L^{-1})_*: CF(\tau_L(\CL_0),\CL) \to CF(\CL_0,\CL)
$$
i.e., satisfies
$$
\CA_{\tau_L(\CL_0),\CL}([p,w]) = \CA_{\CL_0,\CL}([\tau_L^{-1}(p),\tau_L^{-1}(w))]).
$$
By perturbing $L_0$ and $L_1$ if necessary and choosing $\e > 0$
sufficiently small, we may assume:

\begin{enumerate}
\item
$L\cap L_0$, $L\cap L_1$ and $L_0 \cap L_1$ are transverse intersections,
and $L \cap L_0 \cap L_1 = \emptyset$.
\item
Each $\iota^{-1}(L_k) \subset T(r)$ is a union of fibres; one can
write this as
$$
\iota^{-1}(L_k) = \bigcup_{y \in \iota^{-1}(L \cap L_k)} T(r)_y.
$$
\item $R$ satisfies $0 \geq 2\pi R(0)> -\e$, and is such that
$\tau_L$ is $\delta$-wobbly.
\end{enumerate}

The following lemma is a part of Lemma 3.2 \cite{seidel:triangle}.
For readers' convenience, we provide its proof.

\begin{lem}\label{lem:mappq} Suppose
$L_0 \cap L \cap L_1 = \emptyset$.
Then we can choose $\supp \tau_L$ so close to $L$
that $L_0 \cap L_1 \subset M \setminus \operatorname{im} \tau_L$ and
$\tau_L(L_0)$, $L_1$ intersect transversally, and there
are injective maps
\bea
&{}& p: (\tau_L(L_0) \cap L) \times (L \cap L_1) \to \tau_L(L_0) \cap L_1, \\
&{}& q: L_0 \cap L_1 \to \tau_L(L_0) \cap L_1
\eea
such that $\tau_L(L_0) \cap L_1$ is the disjoint union of their images.
\end{lem}
\begin{proof}
The conditions (1), (2) right before Lemma \ref{lem:mappq} imply
$L_0 \cap L_1 \cap U = \emptyset$. Since $\tau_L$ is the identity outside
$U$, one has $L_0 \cap L_1 = (\tau_L(L_0) \cap L_1) \setminus U$, so that
$q$ can indeed be defined to be the inclusion.
There is a bijective correspondence between pairs $(\widetilde x_0,x_1)
\in (\tau_L(L_0)\cap L) \times (L \cap L_1)$ and $(y_0,y_1)
\in \iota^{-1}(L_0\cap L) \times \iota^{-1}(L \cap L_1)$, given by
setting $y_0 = \iota^{-1}(\tau_L^{-1}(\widetilde x_0)), \, y_1 = \iota^{-1}(x_1)$.
As a consequence of the condition (3) above,
\be\label{eq:iota-1}
\iota^{-1}(\tau_L(L_0) \cap L_1) = \bigcup_{y_0,y_1}\tau(T(\lambda)_{y_0})
\cap T(\lambda)_{y_1}.
\ee
It is clear from their definitions that $p, \, q$ are injective.
A point of $\tau_L(L_0) \cap L_1$ falls into $\operatorname{im}(q)$ or
$\operatorname{im}(p)$ depending on whether it lies inside or outside
$\operatorname{im}(\iota)$, hence the two images are disjoint and
cover $\tau_L(L_0) \cap L_1$. The transversality follows from the
definition of $\tau_L$ for $\operatorname{im}(p)$ and from that of
$L_0 \cap L_1$ for $\operatorname{im}(q)$.
\end{proof}

We consider the triple
$$
\CL = (L,\gamma), \, \CL_0 = (L_0, \tau_L^{-1}\circ \gamma), \, \CL_1= (L_1,\gamma_1).
$$
We note $\tau_L(\CL_0) = (\tau_L(L_0), \gamma_0) = (\tau_L(L_0), \gamma)$.
To make our discussion nontrivial, we may assume
\be\label{eq:nonempty}
\CL\cap \CL_0 \neq \emptyset, \quad \CL \cap \CL_1 \neq \emptyset.
\ee
We fix an element $x_1 \in \CL \cap \CL_1$ and $\widetilde x_0 \in \tau_L(\CL_0) \cap \CL$
and $\widetilde z_0 = p(\widetilde x_0,x_1)$ where $p$ is the injective map given in
Lemma \ref{lem:mappq}.

\begin{lem}\label{lem:smalltri} Suppose $x_1 \in \CL \cap \CL_1$ and $\widetilde x_0 \in \tau_L(\CL_0) \cap \CL$
and $\widetilde z_0 = p(\widetilde x_0,x_1)$.
Then we have $\widetilde z_0 \in \CL_1 \cap \tau_L(\CL_0)$.
\end{lem}
\begin{proof} We first note that there is a canonical
homotopy class $B_{can} = B(\widetilde x_0,x_1,\widetilde z)$
spanned by a `thin' triangle contained in $U =
\operatorname{im}\iota$. Choose a path $w_1$ from $\overline\gamma *
\gamma_1$ and $w_0$ from $\overline \gamma_0 * \gamma$. Then it
follows that we can choose a path $w$ from $\overline \gamma_1 *
\gamma_0$ defined by $ u \# w_1 \# w_0 $ where $u$ is the above thin
triangle, i.e., any $w$ such that
$$
[w] = B_{can} \# [w_1] \# [w_0].
$$
This finishes the proof.
\end{proof}

Now we state the following lemma which is a variation of
Lemma 3.2 \cite{seidel:triangle} in our context.

\begin{prop}\label{mappq} Let $L_0, \, L_1$ and $L$ be as in
Lemma \ref{lem:mappq} and consider the maps $p, \, q$ defined
therein. Then we can choose $\operatorname{im} \tau_L$ so close to $L$
so that $\tau_L(L_0) \cap L_1$ satisfy the following properties
in addition:
\begin{enumerate}
\item $q$ is the inclusion $q(x) = x$.
Moreover, for any $z \in \tau_L(\CL_0) \cap \CL_1$ and $z \in \CL_0 \cap \CL_1$ with
$z \neq q(x)$, one has
$$
\CA_{\CL_0\CL_1}([x,w]) - \CA_{\CL_0\CL_1}([z,w']) \not \in [0;3 \e)
$$
whenever the corresponding Floer moduli space $\CM(x,z;[\overline w \# w'])$
is non-empty.
\item
Set $\widetilde z = p(\widetilde x_0,x_1)$. Then there is a canonical
homotopy class $B_{can} = B(\widetilde x_0,x_1,\widetilde z)$ spanned by
a `thin' triangle contained in $U = \operatorname{im}\iota$. And we have
\be\label{eq:OmegaBcan}
|\Omega(B_{can})| < \e.
\ee
\item
For any $z \in \tau_L(\CL_0)\cap \CL_1$ and $(\widetilde x_0,x_1) \in
(\tau_L(\CL_0) \cap \CL) \times (\CL \cap \CL_1)$ with $z \neq p(\widetilde x_0,x_1)$,
or for $z = \widetilde x = p(\widetilde x_0,x_1)$ with $B \neq B_{can}$,
we have
\be\label{eq:energy}
\Omega(B) \geq C = C(\CE;J)
\ee
independent of $\e > 0$
whenever $\CM_J(\widetilde x_0,x_1,z;B) \neq \emptyset$ for some $J$.
\item
Suppose that there are $x_k \in \CL \cap \CL_k$, $k = 0, \, 1$, whose preimages
$y_k = \iota^{-1}(x_k)$ are antipodes on $S^n$. Since $\tau|_{S^n}$ is the
antipodal map, $\widetilde x_0 = \tau_L(x_0)$ is equal to $x_1$
(hence $x_1 \in \tau_L(\CL_0) \cap \CL \cap \CL_1$, and these are all such
triple intersection points). In that case $p(\widetilde x_0,x_1) = \widetilde x_0
= x_1$.
\end{enumerate}
\end{prop}
\begin{proof} The proof is a slight modification of that of Lemma 3.2
\cite{seidel:triangle}. Since we need to strip all the exact Lagrangian
setting away from Seidel's proof thereof and incorporate contributions
coming from different choices of homotopy classes $B$,
we give a complete proof of the proposition.

Now let $s(z) = (z, u(z))$ be the section of $E \to \dot \Sigma$
in class $B_{can}$ satisfying the Lagrangian boundary condition.
By definition of $B_{can}$ we can choose $u$ so that
its image is contained in $\operatorname{im}\iota$. Then
(\ref{eq:OmegaBcan}) follows from the identity
$$
\Omega(B_{can}) = \pi^*\omega(B_{can}) = \omega_0(\iota^{-1}\circ u)
$$
which can be made as small as we want by choosing $r$ small
in the definition of the Dehn twist $\tau_L$.

We now turn to (3). First we recall that since $E$ is trivial (and so of
zero curvature), we have
$$
\Omega(B) = \frac{1}{2}\int_{\dot \Sigma} \|(Du)^v\|_{J}^2
$$
and hence whenever $\CM_J(\widetilde x_0,x_1,z;B) \neq \emptyset$ for some $J$
we have $\Omega(B) \geq 0$. Define
\beastar
C(\CE;J) & = & \inf_{u}\{ \omega(u) \mid u\not \equiv const., \, u \in
\CM_J(\widetilde x_0,x_1,z;B), \\
&{}& \hskip0.8in B \neq B_{can} \, \mbox{ in }\,
\Pi(E,Q)\}.
\eeastar

\begin{prop} Let $B$ be an admissible class and $u \in \CM_J(\widetilde x_0,x_1,z;B)$
with $\mu(B) = 0$ such that
$$
\operatorname{im}u \subset U.
$$
Then we have $B = B_{can}$ in $\Pi(E,Q)$. In particular, we have
\be\label{eq:CLJ}
C(\CE;J) > 0.
\ee
\end{prop}
\begin{proof} Since we know $\mu(B_{can}) = 0$, it is enough to prove
$\omega(B) = \omega(B_{can})$ by definition of $\Pi(E,Q)$. For this purpose, we compare the
action for the paths whose image is contained in the Darboux
neighborhood $U = \iota(V)$ with $\iota: (V,\omega_0) \hookrightarrow (M,\omega)$
and whose end points lie either on $L$ or on the fibers $F$ of the cotangent
bundle $T^*L \cap V$ or in the Dehn twists $\tau_L(F)$.
We recall that the model Dehn twist $\tau$ is a Hamiltonian diffeomorphism
that satisfies
$$
\tau^*\theta_T - \theta_T = dK_\tau
$$
for $K_\tau = 2\pi(\mu R'(\mu) - R(\mu))$.

On the cotangent bundle $T^*T$, the action functionals $\CA_{o_TF_1}, \, \CA_{\tau(F_0)o_T}$
and $\CA_{\tau(F_0)F_1}$ are defined by
\beastar
\CA_{o_TF_1}(z) & = & \int z^*\theta_T \\
\CA_{\tau(F_0)o_T}(z) & = & \int z^*\theta_T + K_\tau \circ \tau^{-1}(z(0)) \\
\CA_{\tau(F_0)F_1}(z) & = & \int z^*\theta_T + K_\tau \circ \tau^{-1}(z(0))
\eeastar
for a path $z: [0,1] \to T^*L$. Here we use the fact that
$$
\theta_T|_{F} \equiv 0\equiv \theta_T|_{o_T}\quad
\theta_T|_{\tau(F)} = d(K_\tau \circ \tau^{-1}|_{\tau(F)}).
$$
(See (2.28) \cite{oh:jdg} or (1.1) \cite{seidel:triangle}.) Since it is
easy to realize such $B = [u]$ as the gluing $-[w_0]\#[w]\#[w_1]$ so that
all $w_i$, $w$ have their images contained in $U$, we can write
$$
\omega(B)
= \CA_{o_TF_1}(\widehat p_0) - \CA_{\tau(F_0)o_T}(\widehat p) -
\CA_{\tau_L(F_0)F_1}(\widehat p_1) = \omega(B_{can}).
$$
Here `hat' denotes the constant path, e.g., $\widehat p$ is the
constant path $\widehat p (t) \equiv p$. This proves $B = B_{can}$ in
$\Pi(E,Q)$.
\par
For the proof of (\ref{eq:CLJ}), it follows from the first part of
the proof that for any $u \in \CM_J(\widetilde x_0,x_1,z;B)$ with
$B \neq B_{can}$ the image of $u$ must go out of $U$. Now a simple
application of Gromov-type compactness theorem implies
(\ref{eq:CLJ}). This finishes the proof.
\end{proof}

This finishes the proof of (3).
The statements in (4) are obvious from definition.
\end{proof}

Now we consider a Darboux neighborhood $U = \operatorname{im}\iota$ of $L$ such that
$\supp \tau_L \subset U$, and Lemma 2.1 and Proposition 2.2 hold.
We denote by $\CM_J(\tau_L(L_0), L, L_1)$ the moduli space of
$J$-holomorphic sections of the fibration
$(E^b,\pi^b,Q^b)$ where $J$ is chosen as before. Then the arguments in
subsection 3.2 \cite{seidel:triangle} give rise to the following decomposition
\be\label{eq:decomp-b}
\frak b = \beta + (\frak b - \beta)
\ee
that satisfies that $\beta$ is of order $[0;\e)$, while $(\frak b-\beta)$ has
order $[3\e;\infty)$. Furthermore $b$ is precisely the class induced by
the canonical `small' map continued from the constant map $p$.

Similar consideration as subsection 3.3 \cite{seidel:triangle} also gives rise
to the decomposition
\be\label{eq:decomp-c}
\frak c = \gamma + (\frak c -\gamma)
\ee
such that $\gamma$ has order 0 while $(\frak c -\gamma)$ has order $[3\e;\infty)$.
We refer to \cite{seidel:triangle} for the detailed proofs of the decomposition results.

Finally we can construct a homotopy from $\frak b\#_\rho \frak c$ to the zero
map following Seidel's argument from \cite{seidel:triangle}. For this purpose we
use Proposition \ref{prop:vanishing} which is a slight generalization
of Proposition 2.2 \cite{seidel:triangle} in the
general context of Lagrangian spheres in general symplectic
manifolds. We omit the details of this construction referring readers to
\cite{seidel:triangle} for the details.

%
%
%
%

\section{Construction of long exact sequence}
\label{sec:sequence}

In this section, we will combine all the results obtained in the previous
section to construct the required long exact sequence.
We first recall two basic lemmas that Seidel used in his construction of
long exact sequence for the \emph{exact} case.

\begin{lem}\label{deltadD} Let $D$ be an $\R$-graded vector space with a differential
$d_D$ of order $[0;\infty)$. Suppose that $D$ has gap $[\e;2\e)$ for some
$\e > 0$. One can write $d_D = \delta + (d_D - \delta)$ with $\delta$ of
order $[0;\e)$, satisfying $\delta^2 = 0$, and $(d_D - \delta)$ of order
$[2\e;\infty)$. Suppose in addition, $H(D,\delta) = 0$; then $H(D,d_D) = 0$.
\end{lem}

Seidel then applied this lemma to the direct sum
$$
D = C' \oplus C \oplus C''
$$
with the differentials given by
$$
d_D = \left(\begin{matrix} d_{C'} & 0 & 0 \\
b & d_C & 0 \\
h & c & d_{C''}
\end{matrix}\right),\quad
\delta = \left(\begin{matrix} 0 & 0 & 0 \\
\beta & 0& 0 \\
0 & \gamma & 0 \end{matrix}\right)
$$
where
\beastar
C'& = & CF(\CL,\CL_1) \otimes CF(\tau_L(\CL_0),\CL), \\
C & = & CF(\tau_L(L_0),L_1), \quad C'' = CF(\CL_1,\CL_0).
\eeastar
and the entries of the matrices are given as stated in the lemma below.

\begin{lem}[Lemma 2.32 \cite{seidel:triangle}]\label{seidellemma} Take three $\R$-graded
vector spaces $C', \, C, \, C''$, each of them with a differential of
order $(0;\infty)$. Suppose that we have the differential maps
$b: C' \to C$, $c: C\to C''$ and a homotopy $h: C' \to C''$ between
$c \circ b$ and the zero map, such that the following conditions are
satisfied for some $\e > 0$:
\begin{enumerate}
\item $C', \, C''$ have gap $(0,3\e)$ and $C$ has gap
$(0,2\e)$.
\item For all $r \in \supp(C')$ and $s \in \supp(C'')$, $|r-s| \geq 4 \e$.
\item One can write $$
b = \beta + (b -\beta), \quad c = \gamma + (c - \gamma)
$$
with $\beta$ of order $[0;\e)$ and $(b-\beta)$ of order $[2\e;\infty)$
and with the same properties for $\gamma$ and $(c-\gamma)$. The lower order parts
(which do not need be differential maps) fit into a short exact
sequence of modules
\be\label{eq:lowerexact}
0 \to C' \stackrel{\beta} \to C \stackrel{\gamma} \to C'' \to 0.
\ee
\item $h$ is of order $[0;\infty)$.
\end{enumerate}
Then the maps on cohomology induced by $b, \, c$ fit into a long exact
sequence
$$
\cdots \to H(C'; d_{C'}) \stackrel{b_*}\to H(C;d_C) \stackrel{c_*} \to H(C'';d_{C''})
\stackrel{\delta} \to H(C'; d_{C'}) \to \cdots.
$$
\end{lem}

The proofs of both lemmata rely on an argument involving spectral sequences.
For the exact case, all the complexes involved are finite dimensional
vector spaces with \emph{bounded} filtration and \emph{gap}
and so existence of spectral sequence such a complex is easy.

On the other hand, for the case of our current interest, the Floer
complex as a $\Q$-vector space is infinite dimensional with
\emph{unbounded} filtration and \emph{without gap} on the vector
space itself in general. Existence of spectral sequence in this case
is much more non-trivial which has been studied by
Fukaya-Oh-Ohta-Ono in \cite{fooo00,fooo:book}. This is what we summarize
in section \ref{subsec:dgcfz}.

In the mean time, we would like to remark that the proof of exactness of
\eqref{eq:lowerexact} is exactly the same as that of \cite{seidel:triangle}
based on Lemmata \ref{lem:mappq}, \ref{lem:smalltri} and on some uniqueness
result on small pseudoholomorphic triangle. (See section 3.2 \cite{seidel:triangle}
for the details.)

\subsection{$CF(\CL,\CL')$ versus gapped d.g.c.f.z. $C(\CL,\CL';\Lambda_{0,nov})$}
\label{subsec:CFLL'}

We first recall basic results on the structure of the Floer cochain
group $CF(\CL,\CL')$ as a module of the Novikov ring $\Lambda_(\CL,\CL')$.
We note that in the current study of Calabi-Yau Lagrangian branes we can use the
Novikov ring
$$
\Lambda_{0,nov}^{(0)}
$$
as the coefficient ring where $\Lambda_{0,nov}^{(0)}$ is the degree zero part of
$\Lambda_{0,nov}$ which is a field. Having this in mind, we
first recall the basic construction on the spectral sequence of the
the $\Lambda_{0,nov}^{(0)}$-module $C(L,L';\Lambda_{0,nov})$ from Chapter 6
\cite{fooo:book} (or Appendix \cite{fooo00}) restricting to the
\emph{finitely generated case}.

Under the given condition on $L_0,\, L_1$ and the given
embedding $f: S^n \to L \subset M$, $CF(\tau_L(\CL_0),\CL), \, CF(\CL,\CL_1)$ and
$CF(\CL_1,\CL_0)$ are all finitely generated over $\Lambda_{nov}$:
Here we note that there is a natural injective homomorphism
$$
I_{\CL,\CL}: \Lambda(\CL,\CL') \to \Lambda_{0,nov}
$$
and so we naturally extend their coefficient rings to $\Lambda_{0,nov}$.

The thick-thin decompositions of the maps
$\frak b$ and $\frak c$ given in section \ref{sec:thickthin} implies that
both maps are \emph{gapped} in the sense of Definitions \ref{gap-delta} and \ref{gap-varphi}.
However these vector spaces, as they are, do not quite manifest the structure of
$d.g.c.f.z$ yet. Because of this, we follow the procedure
given in section 12.4 \cite{fooo:book} turning these into a $d.g.c.f.z.$.
For readers' convenience, we collect the definition of $d.g.c.g.z.$ and
construction of spectral sequence given in \cite{fooo:book} in Appendix.

Let $(\CL_0,\CL_1)$ be a general relatively spin pair of anchored
Lagrangian submanifolds of $M$. We first construct a $\Lambda_{0,nov}$-module
$C(\CL_1,\CL_0;\Lambda_{0,nov})$ which will have a filtered $A_{\infty}$-bimodule structure over
$(C(L_0;\Lambda_{0,nov}),\mathfrak m^{(0)})$ and
$(C(L_1;\Lambda_{0,nov}),\mathfrak m^{(1)})$ where
the latter are the filtered $A_{\infty}$ algebras defined in \cite{fooo:book}.
We consider the intersection $\CL_1\cap \CL_0$ and
the $\R$-filtered set
$$
\widehat I(\CL_0,\CL_1): = \{T^{\lambda}e^{\mu} [p,w] \mid
p \in \CL_0\cap \CL_1, \, [p,w] \in \widetilde\Omega(\CL_0,\CL_1), \,
\lambda \in \R, \, \mu \in \Z \}.
$$
This is a $\R \times \Z$ principal bundle over $\CL_0 \cap \CL_1$.

We define an equivalence relation $\sim$ on $\widehat I(L_0,L_1)$
as follows. We say
$$
T^{\lambda}e^{\mu}[p , w]
\sim T^{\lambda^{\prime}}e^{\mu^{\prime}}[p^{\prime} , w^{\prime}]
$$
for $T^{\lambda}e^{\mu}[p , w],\,
T^{\lambda^{\prime}}e^{\mu^{\prime}}[p^{\prime} , w^{\prime}]
\in \widehat I(\CL_0,\CL_1)$,
if and only if the following conditions are satisfied:
\beastar
p & = & p^{\prime} \\
\lambda + \int w^*\omega  & = & \lambda^{\prime} + \int (w^{\prime})^*\omega
\\
2\mu + \mu ([p,w];\lambda_{01}) & = &
2\mu^{\prime} + \mu ([p^\prime,w^{\prime};\lambda_{01}]).
\eeastar
Here $\mu ([p,w];\lambda_{01})$ is the Maslov-Morse index.
It is easy to see that this relation is compatible with the conditions
of the $\Gamma$-equivalence given and so $\sim$ defines an equivalence relation on
$\widehat I(\CL_0,\CL_1)$.
Furthermore we define the {\it action level} on
$E: \widehat I(\CL_0,\CL_1) \to \R$ by
$$
E(T^{\lambda}e^{\mu}[p,w]) = \lambda + \int w^*\omega
$$
and the associated filtration on the set by setting
$$
T^{\lambda}e^{\mu}[p,w]\in F^{\lambda'}(\widehat I(\CL_0,\CL_1))
$$
if
$$
\lambda + \int w^*\omega \ge\lambda'.
$$
We now define
$$
I(\CL_0,\CL_1)  = \widehat I(\CL_0,\CL_1)/\sim
$$
and somewhat ambiguously denote an element thereof still by
$T^\lambda e^{\mu/2} [p,w]$ as long as no danger of confusion arises.
The above mentioned filtration on $\widehat I(\CL_0,\CL_1)$
obviously induces on on the quotient $I(\CL_0,\CL_1)$.
We now define
$$
I_{\geq 0}(\CL_0,\CL_1): = \left\{ T^\lambda e^{\mu/2} [p,w] \in
I(\CL_0,\CL_1)  \, \Big|\, \lambda + \int w^*\omega \ge 0 \right\}.
$$
Consider the formal sum
$$
\alpha = \sum_{\lambda, \mu, [p,w]} a_{\lambda,\mu,[p,w]}
T^\lambda e^{\mu/2}[p,w]
$$
for $\lambda \in \R, \, \mu \in \Z$ and $[p,w] \in \operatorname{Crit}\CA$ and
define $\supp \alpha$ to be
$$
\supp \alpha = \{T^\lambda e^{\mu/2} [p,w] \in I(\CL_0,\CL_1)
\mid a_{\lambda,\mu,[p,w]} \neq 0\}.
$$
\begin{defn}
We define by $C(\CL_1,\CL_0; \Lambda_{0,nov})$
the $\Lambda_{0,nov}$-module
$$
C(\CL_1,\CL_0; \Lambda_{0,nov})
:= \{\alpha \mid E(\alpha) \geq 0, \, \#(\supp \alpha \cap
E^{-1}((-\infty,\lambda])) < \infty \, \mbox{ for all $\lambda \in \R$} \}
$$
\end{defn}
Obviously $C(\CL_1,\CL_0; \Lambda_{0,nov})$ has a structure of
$\Lambda_{0,nov}$-module. In addition, we have

\begin{prop}
$C(\CL_1,\CL_0; \Lambda _{0,nov})$ is a $d.g.c.f.z.$.
\end{prop}

Giving the grading of an element $T^{\lambda}e^{\mu} [p,w]$
$2\mu +\mu ([p,w];\lambda_{01})$,
it becomes a filtered graded free $\Lambda _{0,nov}$ module.
Following section 5.1.3 \cite{fooo:book}, we write
\beastar
\langle p \rangle = \begin{cases}
T^{-\CA_{\CL_0\CL_1}([p,w])}e^{-\mu([p,w];\lambda_{01})} [p,w] \quad & \mbox{if $\mu([p,w];\lambda_{01})$ is even,}\\
T^{-\CA_{\CL_0\CL_1}([p,w])}e^{-(\mu([p,w];\lambda_{01}) -1)/2} [p,w] \quad & \mbox{if $\mu([p,w];\lambda_{01})$ is odd.}
\end{cases}
\eeastar
Thus we have $E(\langle p \rangle) = 0$ and $\mbox{deg}(\langle p \rangle)$ is either
0 or 1 depending on the parity of $\mu([p,w];\lambda_{01})$.

It is easy to see that $C(\CL_1,\CL_0; \Lambda _{0,nov})$
is isomorphic to the completion (with respect to
the filtration on $\Lambda_{0,nov}$) of the free $\Lambda_{0,nov}$
module generated by $\langle p \rangle$ for the intersection points $p \in \CL_0\cap \CL_1$.
Namely we have a canonical isomorphism
$$
C(\CL_1,\CL_0; \Lambda _{0,nov}) \cong
\underset{p\in \CL_0\cap \CL_1 }
{\widehat \bigoplus} \Lambda _{0,nov} \langle p \rangle
$$
as a ($\Z_2$-graded) $\Lambda _{0,nov}$ module.

Recall the definition  of the Floer cochain
module. By definition, we have an inclusion
$$
CF^*(\CL_1,\CL_0) \to
C^*(\CL_1,\CL_0;\Lambda_{nov})
$$
defined by
\be\label{eq:pwto}
[p,w] \mapsto e^{(\mu([p,w]) -\mu(\langle p
\rangle))/2} T^{\AA_{\ell_0}([p,w])} \langle p \rangle.
\ee
It is compatible with the obvious inclusion
$I_{\CL_0,\CL_1}: \Lambda(\CL_0,\CL_1) \to \Lambda_{nov}$

We take the coefficient $R = \Q$ and recall that Floer cohomology
$$
HF(\CL^{(1)},\CL^{(0)}) = \operatorname{Ker} \delta^{\vec b}/\operatorname{Im} \delta^{\vec b}, \quad
\vec b = (b_1,b_0)
$$
is defined as a $\Lambda(\CL^{(0)},\CL^{(1)})$-module.
\par
We remark that
$$
CF(\CL_1,\CL_0) \cong \Lambda(\CL_1,\CL_0)^{\#(\CL_0 \cap \CL_1)},
$$
where $\# \CL_0 \cap \CL_1$ is finite by the transversality
hypothesis. Therefore we have the isomorphism
$$
C(\CL_1,\CL_0;\Lambda_{nov}) \cong
CF(\CL_1,\CL_0;\ell_{01})
\otimes_{\Lambda(\CL_1,\CL_0)} \Lambda_{nov}.
$$
On the other hand, the Novikov ring
$\Lambda(\CL_1,\CL^{(0)})$ is a field if the ground ring is
$\Q$. Therefore, this leads to the isomorphism
\be\label{eq:HF}
HF((\CL_1,b_1),(\CL_0,b_0);\Lambda_{nov}) \cong
HF(\CL_1,\CL_0)
\otimes_{\Lambda(\CL_1,\CL_0)} \Lambda_{nov}.
\ee
\par
Finally, we explain how we combine the above
discussed anchored versions into a single non-anchored version of
Floer cohomology following section 5.1.3 \cite{fooo:book}.

We first note that the filtered $\Lambda_{0,nov}$-module structure of
$C(\CL_1,\CL_0;\Lambda_{0,nov})$ depends only on the homotopy
class $\ell_{01}$. So we form the completed direct sum
$$
C(L_1,L_0;\Lambda_{0,nov}) = \widehat \bigoplus_{[\ell_0] \in \pi_0(\Omega(L_0,L_1))}
C(L_1,L_0;\ell_0;\Lambda_{0,nov}).
$$
We note that we have the natural inclusion map
$$
C(\CL_1,\CL_0) \to C(L_1,L_0;\overline\gamma_0*\gamma_1;\Lambda_{0,nov})
\subset C(L_1,L_0;\Lambda_{0,nov})
$$
defined as \eqref{eq:pwto}. We define the corresponding Floer cohomology by
$$
HF(L_1,L_0;\Lambda_{0,nov}) : = \operatorname{Ker} \delta^{\vec b}/\operatorname{Im} \delta^{\vec b}.
$$
Then we have
$$
HF(L_1,L_0;\Lambda_{0,nov}) \cong \bigoplus_{[\ell_0] \in \pi_0(\Omega(L_0,L_1))}
HF(\CL_1,\CL_0) \otimes_{\Lambda(L_1,L_0;\ell_0)} \Lambda_{nov}.
$$
(See Proposition 5.1.17 \cite{fooo:book}.)

\subsection{Wrapping it up} \label{subsec:wrapping}

Now combining all the discussions in the previous subsections,
we are ready to prove the main theorem which we now re-state here.

\begin{thm}\label{CYexactsequence}
Let $(M,\omega)$ be a compact (symplectically) Calabi-Yau.
Let $L$ be a Lagrangian sphere in $M$ together with a preferred
diffeomorphism $f: S^2 \to L$. Denote by $\tau_L =\tau_{(L,[f])}$ be the
Dehn twist associated to $(L,[f])$.

Consider any Calabi-Yau Lagrangian branes $L_0, \, L_1$.
Then for any pair $(b_0,b_1)$ of the MC-solutions $b_0 \in \CM(L_0;\Lambda(L_0))$,
$b_1 \in \CM(L_1;\Lambda(L_1))$, there is a long
exact sequence of $\Z$-graded Floer cohomologies
\bea\label{eq:exacttriangle-b}
& \longrightarrow & HF((\tau_L(\CL_0),(\tau_L)_*(b_0)),(\CL_1,b_1))
\longrightarrow HF((\CL_0,b_0),(\CL_1,b_1)) \nonumber \\
&\longrightarrow & HF((\CL,0),(\CL_1,b_1)) \otimes
HF((\CL_0,b_0),(\CL_1,b_1)) \longrightarrow \eea as a
$\Lambda_{nov}$-module where the Floer cohomologies involved are the
deformed Floer cohomology constructed in \cite{fooo00,fooo:book}.
We also have the non-anchored version of the exact sequence.
\end{thm}

The same exact sequence still holds for any orientable relatively spin
pair $(L_0,L_1)$ if they are just unobstructed whose Maslov classes do
not necessarily vanish.

To highlight the main points of the construction, let us first assume
that $b_0 = b_1 = 0$ are MC-solutions. In this case, the Floer cohomology
is the standard one which uses the Floer boundary map $\delta$.
We first state the following lemma which is a consequence of Corollary
\ref{vanishing} and a variation of Lemma \ref{seidellemma}

\begin{lem}\label{H-vanishing} Let $D$ be an $[0;\infty)$-graded vector
space with a differential $d_D$ of order $[0;\infty)$, which is not
necessarily finite dimensional but forms a $d.g.c.f.z.$ in the sense of
Definition \ref{dgcfz}.
Suppose that one can then write $d_D = \delta + (d_D - \delta)$ with $\delta$ of
order $[0;\e)$, satisfying $\delta^2 = 0$, and $(d_D - \delta)$ of order
$[2\e;\infty)$. Suppose in addition, $H(D,\delta) = 0$; then $H(D,d_D) = 0$.
\end{lem}

Now we apply this lemma to the direct sum
$$
D = C' \oplus C \oplus C''
$$
with the differentials given by
$$
d_D = \left(\begin{matrix} d_{C'} & 0 & 0 \\
b & d_C & 0 \\
h & c & d_{C''}
\end{matrix}\right),\quad
\delta = \left(\begin{matrix} 0 & 0 & 0 \\
\beta & 0& 0 \\
0 & \gamma & 0 \end{matrix}\right)
$$
where this time we consider the $\Lambda_{0,nov}$-modules
\beastar
C'& = & C(\CL,\CL_1) \otimes C(\tau_L(\CL_0),\CL), \\
C & = & C(\tau_L(\CL_0),\CL_1), \quad C'' = C(\CL_0,\CL_1).
\eeastar
and the entries of the matrices are given as stated as before.

Now Lemma \ref{mappq} and the thick-thin decomposition results
in section \ref{sec:thickthin} give rise to a long exact sequence
\bea\label{eq:Lambda0nov}
& \longrightarrow & HF(\tau_L(\CL_0),\CL_1;\Lambda_{0,nov})
\longrightarrow HF(\CL_0,\CL_1;\Lambda_{0,nov}) \nonumber \\
&\longrightarrow & HF(\CL,\CL_1;\Lambda_{0,nov})
\otimes HF(\CL_0,\CL_1;\Lambda_{0,nov})
\longrightarrow
\eea
for $\Lambda_{0,nov}^{(0)}$-modules. Since we have
$$
HF(\CL,\CL';\Lambda_{nov}) \cong
HF(\CL,\CL';\Lambda_{0,nov}) \otimes_{\Lambda_{0,nov}} \Lambda_{nov}
$$
from (\ref{eq:HF}) and $\Lambda_{nov}$ is a field, tensoring
(\ref{eq:Lambda0nov}) with $\Lambda_{nov}$ produces the exact
sequence (\ref{eq:exacttriangle-b}) for the case $b_0 = b_1 = 0$.

Now the same reasoning as for the
case $b_i = 0$ induces the long exact sequence (\ref{eq:exacttriangle-b}).
This finishes the proof of Theorem \ref{CYexactsequence}.

\section{Appendix}
\label{sec:appendix}

\subsection{Index formula for $E_{x_0} \setminus \{x_0\}$}
\label{subsec:index}

In this section, we prove the index formula (\ref{eq:index}). There
is an index formula stated various literature in terms of the
`capping surfaces' stated as in \cite{EGH}, \cite{bourgeois}, which
however does not fit our need. For this reason, we give a complete
proof of (\ref{eq:index}).

In fact we will consider the following general set-up. Consider a
symplectic manifold $W$ with a contact type boundary of the type
$$
\del W \cong S^1(T^*N)
$$
with \emph{negative} end for an oriented compact manifold $N$. We
attach the cylinder $\R_+ \times \del W$ and also denote by $W$ the
completed manifold. We denote by $(r(x),\Theta(x))$ for a point $x
\in \R^+ \times S^1(T^*N)$. Composing this with the diffeomorphism
$$
(s,\Theta) \mapsto (e^s, \Theta);\quad \R \times \del W \to \R_+
\times \del W
$$
we put a translational invariant almost complex structure $J$ on the
end.

Next let $\gamma$ be a Reeb orbit of $S^1(T^*N)$ with period $T$. We
note that the symplectic vector bundle $\gamma^*T(T^*N)$ carries a
splitting
$$
\gamma^*T(T^*N) = \C \oplus \gamma^*\xi_N
$$
where $\xi_N$ is the contact distribution of $S^1(T^*N)$.
Furthermore we fix a Riemannian metric $g$ on $N$ and consider the
canonical almost complex structure $J_g$ on $T^*N$. The projection
of $\gamma$ to $N$ is nothing but a geodesic on $N$ with respect to
$g$. Denote by $c= c_\gamma$ the associated geodesic on $N$. Since
we assume $N$ is oriented, we can take a trivialization
$\gamma^*T(S^1(T^*N))$ which is tangent to the vertical
fibers of $T(T^*N)$. Using this we can define the Conley-Zehnder
index of $\gamma$ when $\gamma$ is nondegenerate, which we denote by $\mu_{CZ}(\gamma)$.
For the Bott-Morse case, one uses the generalized Conley-Zehnder index
defined by Robbin and Salamon \cite{robbin-sal}.

Next this choice of trivialization of $\gamma^*(T(S^1(T^*N))) =
\gamma^*T(\del W))$ also allows one to define a \emph{relative Chern
number} of a map $u: \dot \Sigma \to W$ with the asymptotic
condition
\be\label{eq:asymp}
\lim_{\tau \to \infty}\Theta\circ u(\tau,t) = \gamma(t),
\quad \lim_{\tau \to \infty} s \circ u(\tau,t) = -\infty.
\ee
Denote by $\overline u: (\widehat{\dot \Sigma}, \del \widehat{\dot \Sigma})
 \to (W,\gamma)$ the obvious compactified map.

Then $\overline u^*(TW)$ is a symplectic vector bundle
with a trivialization $\phi_\gamma: \gamma^*(T(\del W)) \to S^1 \times \C^{n-1}$
constructed above.

This gives rise to the main definition

\begin{defn} We define the relative Chern number, denoted by $c_1(u;\gamma)$,
by
$$
c_1(u;\gamma) = c_1(u^*TW;\phi_\gamma).
$$
\end{defn}

Once we have made the notions of relative Chern number and Conley-Zehnder
index precise, the following index formula can be derived from
the formula in Corollary 5.4, \cite{bourgeois}

\begin{thm} The expected dimension of
$\CM(W,J;\gamma;A)$ is given by
$$
-\mu_{CZ}(\gamma) + (n-3) + 2c_1(u;\gamma), \quad [u] = A
$$
for a non-degenerate geodesic.

For the Morse-Bott case in which $\CR_{sim}$ forms a smooth manifold, the expected
dimension of the moduli space
$\CM(W,J;A;1)$ consisting of $J$-holomorphic $u$'s with
asymptotics
$$
\lim_{\tau \to \inf} \Theta\circ u(e^{2\pi(\tau + it)}) = \gamma(t)
$$
where $\gamma$ is a simple Reeb orbit, is given by
$$
-\mu_{CZ}(\gamma) + \frac{\dim \CR_{sim}}{2} + (n-3) + 2c_1(u;\gamma)
$$
where $\mu_{CZ}$ is the generalized Conley-Zehnder index of $\gamma$.
\end{thm}

\subsection{d.g.c.f.z. and spectral sequence} \label{subsec:dgcfz}

We first start from the following situation. Let $V =
(\Lambda_{0,nov}^{(0)})^{\oplus I}$ be a free
$\Lambda_{0,nov}^{(0)}$ module with $\#(I)$ finite. We define a
filtration on $V$ in the obvious way which will induce a topology on
$V$. Let $\widehat V$ be the completion of $V$. We call such
$\widehat V$ a {\it completed free filtered $\Lambda_{0,nov}^{(0)}$
module generated by energy zero elements}, or in short {\it c.f.z}.
If $V$ is finitely generated (as a $\Lambda_{0,nov}^{(0)}$ module)
in addition,  we say that it is a {\it finite} c.f.z. We define a
function, which we call the \emph{(action) level},
$$
E: \widehat V  \setminus \{0\} \to \R_{\ge 0}
$$
such that
$$
\mathfrak v \in F^{E(\mathfrak v)}V, \qquad \mathfrak v \notin
F^{\lambda}V \quad\text{if $\lambda > E(\mathfrak v)$.}
$$
Let $\overline V = V/\Lambda_{0,nov}^{+,(0)}V \cong R^I$. {\it We
always take an embedding $($splitting$)$
$$\overline V \subset V$$
as the energy 0 part of $V$ so that its composition with the
projection $V \to \overline V$ is the identity map.}
\par
Let $\mathfrak v \in V$.  We put
$$
\mathfrak v = \sum T^{\lambda_i}v_{i},
$$
where $v_i \in \overline V$, $\lambda_i < \lambda_{i+1}$,
$\lim_{i\to \infty}\lambda_i = \infty$ and $v_i \ne 0$. We call
$T^{\lambda_i}v_i$ the {\it components} of $\mathfrak v$, $
T^{\lambda_1}v_{1}$ the {\it leading component} and $v_{1}$ the {\it
leading coefficient} of $\mathfrak v$. We denote the leading
coefficient $v_1$ of $\mathfrak v$ by $\sigma(\mathfrak v)$.  We
also define the leading component and the leading coefficient of an
element of $\Lambda_{0,nov}^{(0)}$ in the same way.

Now we consider the case of graded $\Lambda_{0,nov}$ modules.

\begin{defn}[Definition 6.3.8 \cite{fooo:book}]\label{dgcfz}
Let $\widehat C$ be a graded $\Lambda_{0,nov}$ module. We assume
that $\widehat C^k$ is a c.f.z. for each $k$. A {\it differential
graded c.f.z.} (abbreviated as d.g.c.f.z) is a pair $(\widehat
C,\delta)$ with a degree $1$ operator $\delta: \widehat C \to
\widehat C$ such that
$$\delta\circ \delta = 0, \qquad
\delta(F^{\lambda} \widehat C) \subseteq F^{\lambda} \widehat C.
$$
We call the pair a {\it finite} d.g.c.f.z. if each $\widehat C^{k}$
is a finite c.f.z.
\end{defn}
The following proposition is essential for the proof of some
convergence properties of the spectral sequence

\begin{prop}[Proposition 6.3.9 \cite{fooo:book}]
Let $W$ be a finitely generated $\Lambda_{0,nov}^{(0)}$ submodule of
$\widehat C^k$.  Then there exists a constant $c$ depending only on
$W$ but independent of $\lambda$ such that
$$
\delta(W)\cap F^{\lambda}\widehat C^{k+1} \subset \delta(W \cap
F^{\lambda-c}\widehat C^k).
$$
\end{prop}

Now let $(\widehat C,\delta)$ be a d.g.c.f.z. and $\widehat C^k$ a
completion of $C^k$. We assume that $C^k$ is {\it free over
$\Lambda_{0,nov}^{(0)}$}. We put
$$
\overline C = C/\Lambda_{0,nov}^{+ (0)}C \cong \widehat
C/\Lambda_{0,nov}^{+ (0)}\widehat C,
$$
and let $\overline\delta$ be the induced derivation on $\overline
C$. We again embed $\overline C \subseteq C \subseteq \widehat C$ as
the energy 0 part. In general $\overline C$ is {\it not} a
differential graded subalgebra of $\widehat C$. Let $\{\mathfrak
e_i\}$ be a basis of $C$ (over $\Lambda^{(0)}_{0,nov}$) and
$\overline{\mathfrak e}_i$ be the corresponding basis of $\overline
C$ over $R = \Lambda_{0,nov}^{(0)}/\Lambda_{0,nov}^{+ (0)}$. We put
$$
\overline\delta(\overline{\mathfrak e}_i) = \sum
\delta_{0,ij}\overline{\mathfrak e}_j,
$$
and define $\delta_0: \widehat C \to \widehat C$ by $\delta_0 =
\overline \delta \otimes 1$ i.e., by
$$
\delta_0 \mathfrak e_i= \sum \delta_{0,ij}\mathfrak e_j .
$$
\begin{defn}\label{gap-delta}
We say that $(\widehat C,\delta)$ satisfies the {\it gapped
condition} if $(\delta - \delta_0)$ has order $[\lambda'',\infty)$,
i.e., if there exists $\lambda'' > 0$ such that for any $\lambda$ we
have
$$
\delta \mathfrak v - \delta_0 \mathfrak v \in F^{\lambda +
\lambda''} \widehat C
$$
for all $\mathfrak v \in F^\lambda \widehat C$.
\end{defn}

Under the gapped condition, we take a constant $\lambda_0$ with $0 <
\lambda_0< \lambda''$ and define a filtration  on $\widehat C$ by
$$
F^n\widehat C = F^{n\lambda_0}\widehat C.
$$
\cite{fooo:book} then uses this filtration to define a spectral
sequence.

\begin{lem}[Lemma 6.3.20 \cite{fooo:book}] Denote
$$
\Lambda^{(0)}(\lambda_0) =
\Lambda^{(0)}_{0,nov}/F^{\lambda_0}\Lambda^{(0)}_{0,nov}.
$$
Then there exists a $\Lambda^{(0)}(\lambda_0)$ module homomorphism
$$
\delta^{p,q}_r: E_r^{p,q}(\widehat C) \to E_r^{p+1,q+r-1}(\widehat
C)
$$
such that
\begin{enumerate}
\item
$\delta^{p+1,q+r-1}_r\circ \delta^{p,q}_r = 0$.
\item
$
\operatorname{Ker}(\delta^{p,q}_r)/\operatorname{Im}(\delta^{p-1,q-r+1}_r)
\cong E_{r+1}^{p,q}(\widehat C). $
\item
$e^{\pm 1} \circ \delta^{p,q}_r = \delta^{p\pm 2,q}_r \circ e^{\pm
1}$.
\end{enumerate}
\end{lem}

Of course, the construction of $E_r^{p,q}(\widehat C)$ is quite
standard. One difference from the standard case is that the
filtration used here is not bounded.  Namely we do {\it not} have
$F^n\widehat C = 0$ for large $n$. Hence the convergence property of
our spectral sequence is far from being trivial in general. However
it is stable from below in that $F^0\widehat C = \widehat C $. As a
consequence we have:

\begin{lem}[Lemma 6.3.22 \cite{fooo:book}]\label{injection}
There exists an injection
$$
E_{r+1}^{p,q}(\widehat C ) \to E_{r}^{p,q}(\widehat C )
$$
if $q-r+2\le 0$.
\end{lem}

An immediate consequence of Lemma \ref{injection} is the following
convergence result.

\begin{prop} The projective limit
$$
E_{\infty}^{p,q}(\widehat C):= \lim_{\longleftarrow}
E_{r}^{p,q}(\widehat C)
$$
exists.
\end{prop}

Furthermore from the construction, we have the description of the
$E_2$-term of the associated spectral sequence

\begin{lem}[Lemma 6.3.24 \cite{fooo:book}]
We have an isomorphism
$$
E_2^{*,*}(\widehat C) \cong H(\overline C;\overline\delta)
\otimes_{R} gr_*(F\Lambda_{0,nov}).
$$
as $gr_*(F\Lambda_{0,nov})$ modules.
\end{lem}
\begin{proof}
By definition we have
$$
E_1^{*,*}(\widehat C) \cong \overline C \otimes_{R} gr_*(
F\Lambda_{0,nov}).
$$
It follows from the gapped condition that $\delta_1 = \overline
\delta$. Hence it finishes the proof.
\end{proof}

\begin{defn}
We define $F^qH(\widehat{C},\delta)$ to be the image of $H(F^q
\widehat{C},\delta)$ in $H(\widehat{C},\delta)$.
\end{defn}

To relate the limit $E_\infty^{p,q}$ of the spectral sequence and
$F^qH(\widehat{C},\delta)$, we need some finiteness assumption which
we now describe. Let $(C,\delta)$ and $(C', \delta')$ be d.g.c.f.z's
satisfying the gap condition. Let $\varphi: C \to C'$ be a map such
that $\varphi \delta = \delta' \varphi$ and let $\overline \varphi
:\overline C \to \overline C'$ be the map induced on $\overline C =
C/\Lambda_{0,nov}^{+ (0)}C$ and $\overline C' =
C'/\Lambda_{0,nov}^{+ (0)}C'$ respectively. The induced map
$\overline\varphi$ lifts to $\varphi_0 = \overline \varphi \otimes 1
: C \to C'$.

\begin{defn}[Definition 6.3.26 \cite{fooo:book}]\label{gap-varphi}
Under the situation above, we say that $\varphi: C \to C'$
satisfies {\it a gapped condition}, or is {\it a gapped cochain
map}, if there exists $\lambda''$ such that
$$
(\varphi - \varphi_0)(F^{\lambda}\widehat C) \subset F^{\lambda +
\lambda''}\widehat C.
$$
\end{defn}

Using these definitions, \cite{fooo:book} proves
\begin{thm}[Theorem 6.3.28 \cite{fooo:book}]
If $C$ is finite, then there exists $r_0$ such that:
$$
E_{r_0}^{p,q}(\widehat C) \cong E_{r_0+1}^{p,q}(\widehat C) \cong
\cdots \cong E_{\infty}^{p,q}(\widehat C) \cong F^qH^p(\widehat
C,\delta)/F^{q+1}H^p(\widehat C,\delta)
$$
as $\Lambda^{(0)}(\lambda_0) =
\Lambda^{(0)}_{0,nov}/F^{\lambda_0}\Lambda^{(0)}_{0,nov}$ modules.
\end{thm}

We summarize the above discussion into the following vanishing
result which will be crucial in our spectral sequence arguments.

\begin{cor}\label{vanishing} Let $C$ be a finite d.g.c.f.z.
Then if $H(\overline C, \overline \delta) = 0$, then $H(C,\delta) =
0$.
\end{cor}

\subsection{Products}\label{subsecprod}

In this subsection, we recall the description of the deformed
products $\frak m_k^{\vec b}$ from \cite{fooo:anchor}. We refer to
\cite{fukaya:mirror2} and \cite{fooo:anchor} for the relevant proofs
of the statements we make without proofs here.

Let $\frak L = (L_0, L_1, \cdots, L_k)$ be a
chain of compact Lagrangian submanifolds in $(M,\omega)$
that intersect pairwise transversely without triple intersections.
\par
Let $\vec z = (z_{0k},z_{k(k-1)},\cdots,z_{10})$ be a set of distinct points on $\partial D^2
= \{ z\in \C \mid \vert z\vert = 1\}$. We assume that
they respect the counter-clockwise cyclic order of $\partial D^2$.
The group $PSL(2;\R)\cong \operatorname{Aut}(D^2)$ acts on the set
in an obvious way. We denote by $\mathcal M^{\text{main},\circ}_{k+1}$ be
the set of $PSL(2;\R)$-orbits of $(D^2,\vec z)$.
\par
In this subsection, we consider only the case $k \geq 2$
since the case $k=1$ is already discussed in the last subsection.
In this case there is no automorphism on the domain $(D^2, \vec z)$, i.e.,
$PSL(2;\R)$ acts freely on the set of such $(D^2, \vec z)$'s.
\par
Let
$p_{j(j-1)} \in L_j \cap L_{j-1}$
($j = 0,\cdots k$), be a set of intersection points.
\par
We consider the pair $(w;\vec z)$ where $w: D^2 \to M$ is a
pseudo-holomorphic map that satisfies the boundary condition
\begin{subequations}\label{54.15}
\begin{eqnarray}
&w(\overline{z_{j(j-1)}z_{(j+1)j}}) \subset L_j, \label{54.15.1} \\
&w(z_{(j+1)j}) = p_{(j+1)j}\in L_j \cap L_{j+1}. \label{54.15.2}
\end{eqnarray}
\end{subequations}
We denote by $\widetilde{\CM}^{\circ}(\frak L, \vec p)$
the set of such
$((D^2,\vec z),w)$.
\par
We identify two elements $((D^2,\vec z),w)$, $((D^2,\vec z'),w')$
if there exists $\psi \in PSL(2;\R)$ such that
$w \circ \psi = w'$ and $\psi(z'_{j(j-1)}) = z_{j(j-1)}$.
Let ${\CM}^{\circ}(\frak L, \vec p)$ be the set of equivalence classes.
We compactify it by including the configurations with disc or sphere bubbles
attached, and denote it by ${\CM}(\frak L, \vec p)$.
Its element is denoted by $((\Sigma,\vec z),w)$ where
$\Sigma$ is a genus zero bordered Riemann surface with one boundary
components, $\vec z$ are boundary marked points, and
$w: (\Sigma,\partial\Sigma) \to (M,L)$ is a bordered stable map.
\par
We can decompose $\CM(\frak L, \vec p)$ according to the homotopy
class $B \in \pi_2(\frak L,\vec p)$ of continuous maps satisfying
\eqref{54.15.1}, \eqref{54.15.2} into the union
$$
\CM(\frak L, \vec p) = \bigcup_{B \in \pi_2(\frak L;\vec p)}
\CM(\frak L, \vec p;B).
$$
\par
In the case we fix an anchor $\gamma_i$ to each of $L_i$ and put $\CE =
((L_0,\gamma_0),\cdots,(L_k,\gamma_k))$, we consider only
admissible classes $B$ and put
\par
$$
\CM(\CE, \vec p) = \bigcup_{B \in \pi_2^{ad}(\CE;\vec p)}
\CM(\CE, \vec p;B).
$$
\begin{thm}\label{58.21} Let $\frak L
= (L_0,\cdots,L_k)$ be a chain of
Lagrangian submanifolds and
$B \in \pi_2(\frak L;\vec p)$.
Then $\CM(\frak L, \vec p;B)$ has an oriented Kuranishi structure
(with boundary and corners). Its (virtual) dimension satisfies
\begin{equation}\label{dimensionformula}
\dim \CM(\frak L, \vec p;B) = \mu(\frak L,\vec p;B) + n + k-2,
\end{equation}
where $\mu(\frak L,\vec p;B)$ is the polygonal Maslov index of $B$.
\end{thm}

We next take graded anchors $(\gamma_i,\lambda_i)$
to each $L_i$. We assume that $B$ is admissible and write
$
B = [w^-_{01}]\#[w^-_{12}] \# \cdots \# [w^-_{k0}]
$ as in Definition \ref{classB}.
We put $w^+_{(i+1)i}(s,t) = w^-_{i(i+1)}(1-s,t)$.
We also put $w^+_{k0}(s,t) = w^+_{0k}(s,1-t)$.
($[w^+_{k0}] \in \pi_1(\ell_{k0};p_{k0})$.)
We also put $\lambda_{k0}(t) = \lambda_{0k}(1-t)$.
\begin{lem}[Lemma 8.11 \cite{fooo:anchor}] \label{dimanddeg}
If $\dim \CM(\frak L, \vec p;B) = 0$, we have
\begin{equation}\label{misdeg1}
(\mu([p_{k0},w^+_{k0}];\lambda_{0k}) - 1)
=
1 + \sum_{i=1}^k (\mu([p_{i(i-1)},w^+_{i(i-1)}];\lambda_{(i-1)i}) - 1).
\end{equation}
\end{lem}

Using the case $\dim \CM(\frak L, \vec p;B) = 0$, we define the $k$-linear operator
$$
\frak m_k:
CF((L_k,\gamma_k),(L_{k-1},\gamma_{k-1}))
\otimes \ldots \otimes
CF((L_1,\gamma_1),(L_{0},\gamma_{0}))\to CF((L_k,\gamma_k),(L_0,\gamma_0))
$$
as follows:
\begin{equation}\label{catAinifwob}
\aligned
\frak m_{k}([p_{k(k-1)},w^+_{k(k-1)}],& [p_{(k-1)(k-2)},w^+_{(k-1)(k-2)}],\cdots, [p_{10},w^+_{10}])) \\
&= \sum \#(\CM_{k+1}(\frak L;\vec p;B)) \, [p_{k0},w^+_{k0}]).
\endaligned\end{equation}
Here the sum is over the basis $[p_{k0},w^+_{k0}]$ of
$CF((L_k,\gamma_k),(L_0,\gamma_0))$, where
$\vec p = (p_{0k},p_{k(k-1)},\cdots,p_{10})$,
$B$ is as in Definition \ref{classB}, and
$w^+_{(i+1)i}(s,t) = w^-_{i(i+1)}(1-s,t)$.
\par
The formula (\ref{misdeg1}) implies that $\frak m_k$ above has
degree one.
\par
In general the operator $\frak m_k$ above does {\it not} satisfy the
$A_{\infty}$ relation by the same reason as that of the case of
boundary operators (see Section \ref{chaincomplex}). We need to use
bounding cochains $b_i$ of $L_i$ to deform $\frak m_k$ in the same
way as the case of $A_\infty$-bimodules (Subsection
\ref{subsec:morphisms}), whose explanation is now in order.
\par
Let $m_0,\cdots,m_k \in \Z_{\ge 0}$ and $\CM_{m_0,\cdots,m_k}(\frak L,
\vec p;B)$ be the moduli space obtained from the set of $((D^2,\vec
z),(\vec z^{(0)},\cdots,\vec z^{(k)}),w))$ by taking the quotient by
$PSL(2,\R)$-action and then by taking the stable map
compactification as before. Here $ z^{(i)} =
(z^{(i)}_1,\cdots,z^{(i)}_{k_i})$ and $z^{(i)}_{j} \in
\overline{z_{(i+1)i}z_{i(i-1)}}$ such that
$z_{(i+1)i},z^{(i)}_1,\cdots,z^{(i)}_{k_i}, z_{i(i-1)}$ respects the counter
clockwise cyclic ordering.
$$
((D^2,\vec z),(\vec z^{(0)},\cdots,\vec z^{(k)}),w))
\mapsto (w(z^{(0)}_1),\cdots,w(z^{(k)}_{m_k}))
$$
induces an evaluation map:
$$
ev=(ev^{(0)},\cdots,ev^{(k)}): \CM_{m_0,\cdots,m_k}(\frak L, \vec p;B)
\to \prod_{i=0}^k L_i^{m_i}.
$$
Let $P^{(i)}_j$ be smooth singular chains of $L_i$ and put
$$
\vec P^{(i)} = (P^{(i)}_1,\cdots,P^{(i)}_{m_i}),
\qquad
\vec{\vec P} = (\vec P^{(0)},\cdots,\vec P^{(k)})
$$
We then take the fiber product to obtain:
$$
\CM_{m_0,\cdots,m_k}(\frak L, \vec p;\vec{\vec P};B) =
\CM_{m_0,\cdots,m_k}(\frak L, \vec p;B) \times_{ev} \vec{\vec P}.
$$
We use this to define
$$
\aligned &\frak m_{k;m_0,\cdots,m_k}: B_{m_k}(CF(L_k)) \otimes
CF((L_k,\gamma_k),(L_{k-1},\gamma_{k-1}))
\otimes \cdots\\
&\quad\otimes CF((L_{1},\gamma_{1}),(L_0,\gamma_{0}))
\otimes B_{m_0}(CF(L_0))
\to CF((L_{k},\gamma_{k}),(L_0,\gamma_{0}))
\endaligned
$$
by
$$
\aligned
\frak m_{k;m_0,\cdots,m_k}
(\vec P^{(k)},[p_{k(k-1)},w^+_{k(k-1)}],&\cdots,[p_{10},w^+_{10}],
\vec P^{(0)})
\\
& = \sum \#(\CM_{k+1}(\frak L;\vec p;\vec{\vec P};B)) \, [p_{k0},w_{k0}].
\endaligned$$
Finally for each given $b_i \in CF(L_i)[1]^0$ ($b_i \equiv 0 \mod \Lambda_+$),
$\vec b =(b_0,\cdots,b_k)$, and $x_i \in CF((L_i,\gamma_i),(L_{i-1},\gamma_{i-1}))$, we put
\begin{equation}\label{mkcorrected}
\frak m_k^{\vec b}(x_k,\cdots,x_1) = \sum_{m_0,\cdots,m_k} \frak
m_{k;m_0,\cdots,m_k} (b_k^{m_k},x_k,b_{k-1}^{m_{k-1}},\cdots,x_1,b_0^{m_0}).
\end{equation}
\begin{thm}
If $b_i$ satisfies the Maurer-Cartan equation $(\ref{eq:MC})$ then
$\frak m_k^{\vec b}$ in $(\ref{mkcorrected})$ satisfies the
$A_{\infty}$ relation
\begin{equation}\label{Ainftyrel}
\sum_{k_1,k_2,i}
(-1)^* \frak m_{k_1}^{\vec b}(x_k,\cdots,\frak m_{k_2}^{\vec b}(x_{k-i-1},\cdots,x_{k-i-k_2}),\cdots,x_1) = 0
\end{equation}
where we take sum over $k_1+k_2=k+1$, $i=-1,\cdots,k-k_2$.
(We write $\frak m_k$ in place of $\frak m^{\vec b}_k$ in $(\ref{Ainftyrel})$.)
The sign $*$ is
$
* = i + \deg x_k +\cdots + \deg x_{k-i}.
$
\end{thm}

We summarize the above discussion as follows:
\begin{thm}\label{anchoredAinfty}
We can associate an filtered $A_{\infty}$ category to a
symplectic manifold $(M,\omega)$ such that:
\begin{enumerate}
\item Its object is $(\CL,b,sp)$ where
$\CL=(L,\gamma,\lambda)$ is a graded anchored Lagrangian submanifold, $[b]
\in \CM(CF(L))$
is a bounding cochain and $sp$ is a spin structure of $L$.
\par
\item The set of morphisms is $CF((L_1,\gamma_1),(L_0,\gamma_0))$.
\par
\item $\frak m^{\vec b}_k$ are the operations defined in $(\ref{mkcorrected})$.
\end{enumerate}
\end{thm}
\begin{rem}
Here we spelled out
the choice of orientations $o_p$ of $\text{Index}\, \overline{\partial}_{\lambda_p}$
is included. This choice in fact does not affect the module structure
$CF((L_1,\gamma_1),(L_0,\gamma_0))$ up to isomorphism:
if we take an alternative choice $o'_p$ at $p$, then all the
signs appearing in the operations $\frak m_k$ that involves $[p,w]$ for some $w$
will be reversed. Therefore $[p,w] \mapsto -[p,w]$ gives the required isomorphism.
\end{rem}

The operations $\frak m_k$ are compatible with the filtration.
Namely we have
\begin{prop}\label{filprod}
If $x_i \in F^{\lambda_i}CF((L_i,\gamma_i),(L_{i-1},\gamma_{i-1}))$,
then
$$
\frak m_k^{\vec b}(x_k,\cdots,x_1)
\in F^{\lambda}CF((L_k,\gamma_k),(L_0,\gamma_0))
$$
where
$
\lambda = \sum_{i=1}^{k} \lambda_i.
$
\end{prop}

\end{document}